\documentclass[a4paper,10pt,reqno]{article}
\usepackage{latexsym,amsmath,amssymb,amsfonts,amsthm}
\usepackage{mathrsfs}
\usepackage{a4wide}
\usepackage[usenames, dvipsnames]{xcolor}
\usepackage{graphicx}
\usepackage[colorlinks=true, linkcolor=blue, citecolor=ForestGreen]{hyperref}
\usepackage[utf8]{inputenc}
\definecolor{red(ncs)}{rgb}{0.77, 0.01, 0.2}

\newcommand{\Gdue}{{{g}}}
\newcommand{\Gzero}{{{g_0}}}

\newcommand{\xupref}[2]{\hspace{-0.3ex}\stackrel{\eqref{#1}}{#2}} 

\newtheorem{theorem}{Theorem}[section]
\newtheorem{lemma}[theorem]{Lemma}
\newtheorem{proposition}[theorem]{Proposition}
\newtheorem{corollary}[theorem]{Corollary}

\newtheorem{remark}[theorem]{Remark}

\numberwithin{equation}{section}

\newcommand{\e}{\varepsilon}
\newcommand{\vphi}{\varphi}

\newcommand{\N}{\mathbb N}

\newcommand{\R}{\mathbb R}

\DeclareMathOperator*{\loc}{loc}

\newcommand{\hn}{{\mathcal H}^{n-1}}
\newcommand{\Lu}{{\mathcal L}^1}

\newcommand{\hz}{{\mathcal H}^{0}}
\newcommand{\wto}{\rightharpoonup}

\newcommand{\de}{\,\mathrm{d}}
\renewcommand{\setminus}{\backslash}
\newcommand{\mres}{\mathbin{\vrule height 1.6ex depth 0pt width 0.13ex\vrule height 0.13ex depth 0pt width 1.3ex}}
\newcommand{\sfrac}{s_{\mathrm{frac}}}
\newcommand{\AC}{{\rm AC}}
\newcommand{\BV}{{\rm BV}}
\newcommand{\SBV}{{\rm SBV}}

\newcommand{\f}{\mathcal{F}}
\newcommand{\fe}{\mathcal{F}_\e}

\newcommand{\fk}{\mathcal{F}_{k}}
\newcommand{\febar}{{\hat{\mathcal{F}}_\e}}
\newcommand{\fbar}{{\hat{\mathcal{F}}}}
\newcommand{\g}{\mathcal{G}}
\newcommand{\gs}{\mathcal{G}_s}
\newcommand{\jump}{\Gzero(|[u](\bar{x})|)}
\newcommand{\tif}{f_1}
\newcommand\slopefuno{\eta}

\title{Cohesive fracture in {1D}: quasi-static evolution and { derivation} from static phase-field models}
\author{Marco Bonacini\thanks{\emailmarco} \and Sergio Conti\thanks{\emailsergio} \and Flaviana Iurlano\thanks{\emailfla}}
\date{${}^*$\addmarco \\ ${}^\dag$\addsergio \\ $^\ddag$\addfla \\[3ex]\today}

\newcommand{\email}[1]{E-mail: \tt #1}
\newcommand{\emailmarco}{\email{marco.bonacini@unitn.it}}
\newcommand{\emailsergio}{\email{sergio.conti@uni-bonn.de}}
\newcommand{\emailfla}{\email{iurlano@ljll.math.upmc.fr}}
\newcommand{\addmarco}{\emph{\small Department of Mathematics, University of Trento\\ Via Sommarive 14, 38123 Povo (TN), Italy}}
\newcommand{\addsergio}{\emph{\small University of Bonn, Institute for Applied Mathematics\\ Endenicher Allee 60, 53115 Bonn, Germany}}
\newcommand{\addfla}{\emph{\small Sorbonne-Université, CNRS, Université de Paris,\\ Laboratoire Jacques-Louis Lions (LJLL), F-75005 Paris, France}}

\begin{document}
	
\maketitle

\begin{abstract}
In this paper we propose a notion of irreversibility for the evolution of cracks in presence of cohesive forces, which allows for different responses in the loading and unloading processes, motivated by a variational approximation with damage models. We investigate its applicability to the construction of a quasi-static evolution in a simple one-dimensional model. The cohesive fracture model arises naturally via  $\Gamma${-}convergence from a phase-field model of the generalized Ambrosio-Tortorelli type, which may be used as regularization for numerical simulations.
\end{abstract}
	
\tableofcontents


\section{Introduction} \label{sect:intro}


The variational approach to fracture, as formulated by Francfort and Marigo \cite{FraMar} (see also \cite{BFM}), is based on Griffith's idea \cite{Gri20} that the crack growth is determined by the competition between the energy spent to increase the crack and the corresponding release in the bulk elastic energy.
By representing the crack as a $(n-1)$-dimensional surface $\Gamma$ and assuming for simplicity that the material behaves as a linearly elastic body in the unbroken part $\Omega\setminus\Gamma$ of the reference configuration, 
{for pristine materials}
one is led to consider an energy of the form
\begin{equation} \label{intro1}
E(u,\Gamma) := \int_{\Omega} |\nabla u|^2\de x + \int_{\Gamma} 
\Gzero(|[u](x)|)\de\hn,
\end{equation}
where $[u](x):=u^+(x)-u^-(x)$ is the difference of the traces of the displacement ${u:\Omega\to\R}$ on the two sides of $\Gamma$. In the original Griffith's theory for \emph{brittle fracture} the energy associated with the crack is proportional to the measure of the surface of the crack itself, and $\Gzero$ is assumed not to {depend} on the crack opening $[u]$.
It was however early recognized (Barenblatt \cite{Bar}) that in many situations fracture should be regarded as a gradual process, and that the presence of cohesive forces between the lips of the crack should be taken into account: in \emph{cohesive fracture} models Griffith's energy is replaced by various surface energies depending on the actual opening of the crack; in this case the function $\Gzero$ in \eqref{intro1} is typically assumed to be an increasing and concave function, with $\Gzero(0)=0$, with positive and finite slope $\Gzero'(0)\in(0,\infty)$ at the origin, and asymptotically converging to the value of the fracture toughness. Variants of \eqref{intro1} with $\Gzero$ not necessarily concave have been used to study the formation of micro-cracks \cite{DPT1,DPT2}.

The notion of \emph{irreversible quasi-static evolution} in the brittle case ($\Gzero$ constant), proposed in \cite{FraMar}, under the action of time-dependent loads (for instance, in the form of a prescribed boundary displacement $b(t)$), requires three essential ingredients: one looks for a map $t\mapsto(u(t),\Gamma(t))$ satisfying
\begin{enumerate}
	\item (\emph{irreversibility}) $\Gamma(t_1)\subset\Gamma(t_2)$ for $t_1<t_2$;
	\item (\emph{static equilibrium}) $(u(t),\Gamma(t))$ is a global minimizer of the energy \eqref{intro1} with respect to the boundary condition $b(t)$;
	\item (\emph{energy balance}) the increment in stored energy plus the energy spent in crack increase equals the work of external forces.
\end{enumerate}
The common approach in showing the existence of a variational evolution with the properties above is based on a time-discretization algorithm, in which one selects at each time step a global minimizer of the total energy, and then recovers the time-continuous limit by sending to zero the discretization parameter. This program has been successfully accomplished in the brittle case in a series of contributions: the existence of a quasi-static evolution {was first proven} by Dal Maso and Toader \cite{DMToa} in a two-dimensional, antiplane shear setting and with a uniform bound on the number of connected components of the crack, {then} extended to the case of planar elasticity by Chambolle \cite{Cha}. The general $n$-dimensional case, without restrictions on the number of connected components, {was} studied by Francfort and Larsen \cite{FraLar} (antiplane case) and by Dal Maso, Francfort and Toader \cite{DMFraToa} (finite elasticity). Further results in this direction {were} obtained in \cite{KneMieZan08,DMLaz,Laz,ALL}.

In the cohesive case the picture is less clear and there is no agreement on a single notion of irreversibility. Indeed, while in the brittle case the irreversibility constraint is a purely geometric condition, in the cohesive context one has to take into account also the amplitude of the cracks, through the inclusion of some {internal} variable which keeps track of the complete history of the fracture process. In particular, it would be desirable to model the possibility of having different responses of the material to loading and unloading processes. Up to now, several plausible choices have been proposed in the mathematical literature 
{under the assumption that the crack path is prescribed}
{(see also the discussion in \cite[Chapter~5]{BFM})}: Dal Maso and Zanini \cite{DMZan} adopt as {internal} variable  the maximal opening of the crack, postulating that the energy is dissipated only once for a given value of the jump, and additional dissipation can only occur if the distance between the {crack} lips is further increased; { Bourdin, Francfort and Marigo \cite{BFM} propose a different model, which adopts as the memory variable the cumulated opening of the crack and allows for the phenomenon of fatigue, as the surface energy evolves during each loading phase}; in the model by Cagnetti and Toader \cite{CagToa} part of the dissipated energy is recovered when the crack opening is decreased; finally also Crismale, Lazzaroni and Orlando \cite{CriLazOrl} include {fatigue assuming that} some energy is dissipated also during the unloading phase. 
In {these} cases the existence of a quasi-static evolution (encoding in the definition the relevant notion of irreversibility) is established. In the cohesive fracture context, further results on evolutions of local minimizers or critical {points} (rather than global minimizers) are obtained in \cite{Cag,Alm,ArtCagForSol,NegSca}.

\medskip
{ In the first part of this paper (Sections~\ref{sect:setting}--\ref{sect:evol2}) we discuss the construction of a quasi-static evolution for a one-dimensional model of cohesive fracture, which allows for a different response of the material to loading and unloading. This is achieved by introducing a cohesive energy density {$\Gdue(s,s')$  which depends both on the crack opening $s$ and on the internal variable $s'$, which keeps track of the maximum previous opening
(see Section~\ref{sect:setting}
for the detailed assumptions on $\Gdue$)}. It is the goal of the second part of the paper (Sections~\ref{sect:gnew}--\ref{sect:proofs}) to show how those assumptions on the energy density ${\Gdue}$ can be naturally derived by introducing  an irreversibility condition on a phase-field approximation.

The construction of a quasi-static evolution follows a rather standard approach, which we now outline.} We let $\Omega=(0,1)$ represent an elastic bar, on which we prescribe a time-dependent boundary displacement $t\mapsto (b(t,0),b(t,1))$. We also include in the energy a lower-order penalization of the form $\gamma\|u-w(t)\|_{L^2(0,1)}^2$, where $\gamma>0$ and $t\mapsto w(t)$ is a given function {which prescribes the average behavior of the material. This foundation-type term, as usual in one-dimensional simplifications of models in solid mechanics, represents the boundary conditions in the directions that have been eliminated. A simple way to understand its physical significance is to imagine solving for ${\tilde u}:(0,1)\times(-H,H)\to\R$ with boundary conditions ${\tilde u}(x,\pm H)=w(x)$ on the top and bottom boundaries, and then focusing on $u(x):={\tilde u}(x,0)$. The difference $u-w$ would be penalized by a $\partial_2 {\tilde u}$ term in the energy. If dealing with the Dirichlet functional, the $H^{1/2}$ norm of the difference would be the most natural way to model this effect; for general energy densities there is no clear natural form. For simplicity we use the $L^2$ norm.}
Due to the presence of this term the bar can be fractured at multiple points, while if $\gamma=0$ one can prove that minimizers of the energy have at most one jump.

We keep track of the state of the fracture at each time $t\in[0,T]$ by means of a pair {of internal variables} $(\Gamma(t),s(t))$, where $\Gamma(t)$ is a discrete set of points (representing the sites of the crack points), and $s(t):\Gamma(t)\to(0,\infty)$ is the maximal amplitude of the jump at the points of $\Gamma(t)$ for all the previous times. For technical reasons, we fix a (arbitrarily small) positive threshold $\bar{s}>0$ and we keep track only of the jumps that overcome this value; in other words, fracture points where the jump is smaller than $\bar{s}$ are not affected by the irreversibility condition, and the opening can be reduced with a complete recover of the dissipated energy. From the mathematical point of view this guarantees a uniform bound on the total number of crack points (this assumption can be easily removed in the case $\gamma=0$, that is if we do not include the lower order penalization). {It is often convenient to switch to the equivalent formulation in which $s$ is  first extended by 0 outside $\Gamma$, and $\Gamma$ 
is implicitly replaced by the set $\{s>0\}$.}

{ {One main novelty in the problem we consider here} is that a new energy density ${\Gdue}(s,s')$ appears at the points where the bar was previously broken. The first argument $s$ of ${\Gdue}$ represents the current opening of the crack at a given point; the second argument $s'$ represents the maximal amplitude of the jump reached at that point at the previous times,
{in particular $\Gdue(s,0)=\Gzero(s)$ represents the fracture energy in a pristine material. More in general,}
when $s\geq s'$ the additional constraint  is irrelevant and one has ${\Gdue}(s,s')=\Gzero(s)$: this means that, in the process of crack opening, the energy per unit area of the fracture at a point $x$ is given by $\Gzero(|[u(x)]|)$, the ``original'' cohesive energy density. The behaviour of ${\Gdue}(s,s')$ is instead different when $s<s'$: in this case $s\mapsto{\Gdue}(s,s')$ is monotone nondecreasing and, {generally,} ${\Gdue}(s,s')>\Gzero(s)$. Hence, when $|[u](x)|$ is smaller than the maximal opening reached at previous times, the energy density follows a curve which is above the graph of $\Gzero$, see Figure~\ref{figggbarintro}. {
If a fracture with opening $s$ is closed, a certain amount of energy $\Gdue(0,s)$ is permanently dissipated, so that a successive opening only has the smaller cost $\Gdue(s,s)-\Gdue(0,s)$.
We refer to Section \ref{subsect:hyp} and Section \ref{subsect:barg} below for the construction of specific examples for the functions $\Gzero$ and $\Gdue$, respectively.}

We consider a time discretization $0=t_0<t_1<\ldots<t_N=T$ and we initially select the displacement $u_0$ by minimizing the energy
\begin{equation} \label{intro5}
\int_0^1 |u'|^2\de x + \sum_{x\in J_u} \Gzero(|[u](x)|) + \gamma\int_0^1 |u-w(t_0)|^2\de x
\end{equation}
among all $u\in\SBV(0,1)$ attaining the boundary conditions $b(t_0)$ at the endpoints of the bar. To be more precise, we need to consider the lower semicontinuous envelope of the energy and a minimizer will belong in general to $\BV(0,1)$: indeed the presence of the lower-order term allows for the possibility of having more than one jump, while some regularity can still be proved, see Proposition~\ref{prop:nocantor}. In the case $\gamma=0$ it is well-known that minimizers in dimension one are in $\SBV$ with a single jump, see \cite{BDG}; see also \cite{DMG,CCF} for further regularity results for cohesive energies.

We then define
\begin{equation*}
\Gamma_0:= \bigl\{ x\in J_{u_0}^{b(t_0)} : |[u_0](x)|>\bar{s} \bigr\}, \qquad s_0(x):=|[u_0](x)| \quad\text{for }x\in\Gamma_0
\end{equation*}
(where the superscript $b(t_0)$ in the jump set of $u_0$ indicates that we consider as part of the crack also the points at which the boundary conditions are not attained).
Iteratively, assuming to have constructed $(u_{i},\Gamma_i,s_i)$ for $i=0,\ldots,k-1$, we select the displacement $u_k$ at time $t_k$ by minimizing (the relaxation of) the energy
\begin{equation} \label{intro6}
\int_0^1 |u'|^2\de x + \sum_{x\in\Gamma_{k-1}} {\Gdue}(|[u](x)|,s_{k-1}(x)) + \sum_{x\in J_u\setminus\Gamma_{k-1}} \Gzero(|[u](x)|) + \gamma\int_0^1 |u-w(t_k)|^2\de x
\end{equation}
with respect to the boundary condition $b(t_k)$, and we update the state of the fracture by setting
\begin{equation*}
\Gamma_k:= 
{\Gamma_{k-1}\cup}
\bigl\{ x\in J^{b(t_k)}_{u_k} \,:\, |[u_k](x)|>\bar{s}  \bigr\}, \qquad s_k(x):=|[u_k](x)|\vee s_{k-1}(x) \quad\text{for }x\in\Gamma_k.
\end{equation*}}

{Passing} to the limit as the time-step goes to zero we obtain in Theorem~\ref{thm:evolution} the existence of a quasi-static evolution $(u(t),\Gamma(t),s(t))$ satisfying the { by now classical unilateral minimality and energy conservation conditions of quasi-static energy minimizing evolutions, {see \cite{MieRou}}.} The irreversibility condition is reflected in the fact that the sets $\Gamma(t)$ and the maps $s(t)$ are monotonically increasing in time.

\begin{figure}
 \begin{center}
  \includegraphics[width=6cm]{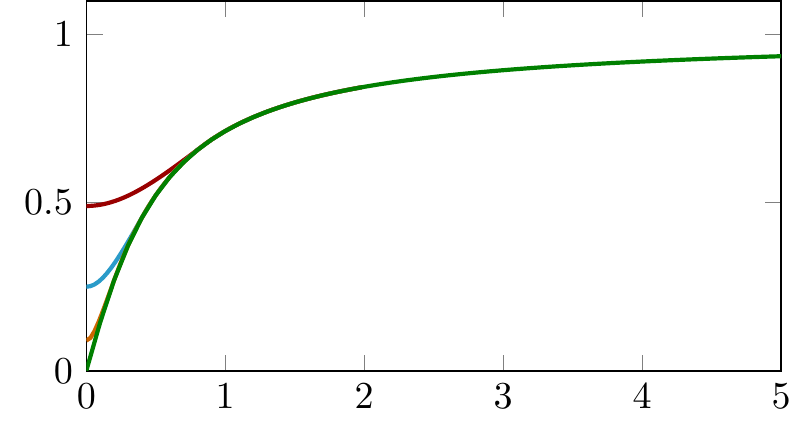}
  \end{center}
\caption{Graph of $\Gzero(s)$ (bottom curve) and $\Gdue(s,s')$ for $m_{s'}=0.3, 0.5$ and $0.7$ (from bottom to top)
using the function $f_1^b(s)$ defined in \eqref{eqexamplefb} with $\ell=1.5$, $\ell_1=0.2$. 
See {Section~\ref{subsect:hyp} and in particular} Figure~\ref{figgambbar1} below for further details.}
\label{figggbarintro}
\end{figure}

\medskip
{ The second part of the paper is devoted to explain how the notion of irreversibility, used in the construction of the quasi-static evolution in the first part, is related to a corresponding irreversibility constraint on approximating damage models. In particular, one of the goals of the second part is to derive the properties of the cohesive energy density ${\Gdue}$, stated in Section~\ref{sect:setting}, from a phase-field approximation of the cohesive functional \eqref{intro6} in the spirit of \cite{CFI}.}

In order to explain the main idea, it is first convenient to refer once again to the brittle case, and to recall the classical Ambrosio-Tortorelli approximation \cite{AmbTor90,AmbTor92} of the Mumford-Shah functional
\begin{equation*}
E(u) := \int_{\Omega} |\nabla u|^2\de x + \hn(J_u), \qquad u\in\SBV(\Omega),
\end{equation*}
by the regularizing energies
\begin{equation*}
E_\e(u_\e,v_\e) := \int_{\Omega} (v_\e^2 + o(\e)) |\nabla u_\e|^2\de x + \frac12\int_{\Omega} \Bigl( \e|\nabla v_\e|^2 + \frac{(1-v_\e)^2}{\e} \Bigr)\de x, \quad (u_\e,v_\e)\in H^1(\Omega)\times H^1(\Omega),
\end{equation*}
with $0\leq v_\e\leq 1$. Here the function $u_\e$ has to be understood as a regularization of the displacement $u$, possibly creating jumps in the limit as $\e\to0$, while $v_\e$ can be interpreted as a damage variable which concentrates in the regions where the singularities of $u$ are created: more precisely, $v_\e\to1$ almost everywhere in $\Omega$, and $v_\e\to0$ on the $(n-1)$-dimensional surface where the limit jump will appear. The notion of quasi-static evolution in the regular context of the Ambrosio-Tortorelli functionals can be formulated including an irreversibility condition in the form of a global monotonicity constraint on $v_\e$ (the damage can only increase in time). It is proved in \cite{Gia05} that this evolution converges as $\e\to0$ to a quasi-static evolution for brittle fracture.

Returning to the cohesive setting, in the spirit of Ambrosio and Tortorelli it is shown in \cite{CFI} (see also \cite{FreIur} for a numerical insight) that, {for a suitable choice of the function $f_\e$,} functionals of the form
\begin{equation} \label{intro3}
F_\e(u_\e,v_\e) := \int_{\Omega} \Bigl( f_\e^2(v_\e)|\nabla u_\e|^2 + \frac{(1-v_\e)^2}{4\e} + \e|\nabla v_\e|^2 \Bigr)\de x,
\quad (u_\e,v_\e)\in H^1(\Omega)\times H^1(\Omega),
\end{equation}
with $0\leq v_\e\leq 1$, $\Gamma$-converge as $\e\to0$ to the relaxation of an energy of the form \eqref{intro1} (see Section~\ref{subsect:blowup} for the precise statement). The function $\Gzero$ obtained as limit surface energy density can be characterized as the solution to a minimum problem,
\begin{multline} \label{intro2}
\Gzero(s) = \inf\biggl\{ \int_{-\infty}^\infty \biggl( {s^2}f^2(\beta)|\alpha'|^2 + \frac{(1-\beta)^2}{4} + |\beta'|^2 \biggr) \de t \;:\; (\alpha,\beta)\in H^1_{\loc}(\R)\times H^1_{\loc}(\R),\, \alpha'\in L^1(\R), \\
\bigg|\int_{-\infty}^\infty \alpha'(t)\de t\bigg| = {1},\, 0\leq\beta\leq1,\, \lim_{|t|\to\infty}\beta(t)=1 \biggr\} ,
\end{multline}
and satisfies the natural requirements for a cohesive energy {(see Figure \ref{figggbarintro})}. As before, $u_\e$ is a regularization of $u$ developing singularities in the limit, and $v_\e$ converges to the value 1 almost everywhere and {the measure $|\nabla v_\e|\mathcal L^n$} concentrates on the limit jump. The crucial observation is that, in contrast with the Ambrosio-Tortorelli case, $v_\e$ does not tend to zero everywhere on the jump, but reaches a value depending on the amplitude of the limit jump of $u$.
Indeed, the one-dimensional blow-up analysis of the behaviour of recovery sequences $(u_\e,v_\e)\to(u,1)$ around a jump point $\bar{x}\in J_u$, which we perform in Section~\ref{subsect:blowup}, shows that the rescaled functions $z_\e(x):=u_\e({x_\e+}\e x)$, $w_\e(x):=v_\e({x_\e+}\e x)$ converge to an optimal profile $(\alpha_s,\beta_s)$ for the minimum problem \eqref{intro2} defining $\Gzero(s)$, for the value $s=|[u](\bar{x})|$ (Theorem~\ref{thm:blowup}); moreover the minimum value reached by $v_\e$ in a small neighbourhood of $\bar{x}$, $m_s:=\min_{\R}\beta_s$, is in one-to-one correspondence with the limit amplitude $s$ of the jump.

Therefore the relevant information about the amplitude of the limit jump $|[u]|$ is carried only by the value of the minimum of the damage variable $v_\e$.
This suggests to impose, at level $\e$, an irreversibility condition in the form of a monotonicity constraint on the value of local minima of $v_\e$: in the points where $v_\e$ has a local minimum, its value can only decrease.
By looking at the $\Gamma$-limit of the functionals $F_\e$ in \eqref{intro3} with this additional constraint (see Section~\ref{subsect:constrainedpbl}), one finds that the limit energy density has the form
\begin{multline} \label{intro4}
{\Gdue}(s,s') = \inf\biggl\{ \int_{-\infty}^\infty \biggl( {s^2} f^2(\beta)|\alpha'|^2 + \frac{(1-\beta)^2}{4} + |\beta'|^2 \biggr) \de t \;:\; (\alpha,\beta)\in H^1_{\loc}(\R)\times H^1_{\loc}(\R),
{\alpha'\in L^1(\R),}\\ \bigg|\int_{-\infty}^\infty \alpha'(t)\de t\bigg| = {1},\, 0\leq\beta\leq1,\, \lim_{|t|\to\infty}\beta(t)=1,\, \inf_\R \beta \leq m_{s'} \biggr\} ,
\end{multline}
see Figure \ref{figggbarintro}.
In other words, we restrict the class of admissible profiles by imposing that the minimum value of the profile $\beta$ does not exceed the minimum value $m_{s'}$ of the optimal profile $\beta_{s'}$ corresponding to the amplitude $s'$.

{ We then show that the function ${\Gdue}$ derived in this way has exactly the properties stated in Section~\ref{sect:setting}. Notice that when $s\geq s'$ the additional constraint in \eqref{intro4} is irrelevant and one has $\Gzero(s)={\Gdue}(s,s')$. At the same time, $\Gdue(0,s')$ is (at least for {$s'$ not too large}, see Section~\ref{sect:proofs} for details) strictly smaller than $\Gzero(s')=\Gdue(s',s')$. Indeed, upon closure of the fracture, the part of the energy that originates from the elastic degrees of freedom (corresponding to the $f^2(\beta)|\alpha'|^2$ term in \eqref{intro4}) is recovered, whereas the part that originates from the internal variable (corresponding to the two terms $ \frac{(1-\beta)^2}{4} $ and $|\beta'|^2$ in \eqref{intro4}, with $\beta(\pm\infty)=1$ and $\beta(0)\le m_{s'}$) is permanently dissipated. This behaviour is consistent with models where some of the dissipated energy is recovered when the crack opening is decreased \cite{CagToa,Alm,NegSca,NegVit} and our analysis shows that it can be recovered as the limit of damage models.}

\medskip
It is worth to briefly comment on the connection of this approach with other results in the mathematical literature. As remarked above, in our model the response of the material when the crack opening decreases is similar to that considered in \cite{CagToa}, that is some dissipated energy is recovered. 
{However, the {memory of the fracture} is permanent, and further iterations of opening and closing the same fracture are reversible.}
In particular, the phenomenon of fatigue is completely absent. Fatigue effects are instead taken into account { in \cite{BFM} and in \cite{CriLazOrl} (see also \cite{AleCriOrl}): in particular, in \cite{CriLazOrl}} also when the crack opening is reduced some energy is dissipated, and oscillations of jumps can produce a complete fracture even if the maximal crack amplitude remains small. This behaviour is consistent with models where cohesive effects result from the interaction between damage and plasticity (see \cite{AleMarVid14,AleMarVid15,DMOrlToa}); on the contrary, our approach is based on the approximation \cite{CFI}, where the cohesive behaviour is seen as the effect of the interplay between elasticity and damage. 
{
We expect that our approach could be extended to models of fatigue, by replacing the
maximal amplitude  of the opening by the 
{cumulated} opening as an internal variable. This extension is, however, beyond the scope of this paper.}

We further remark that the functionals \eqref{intro3}, when a monotonicity constraint on the minimum values of $v_\e$ is included, provide a variational approximation of the limit energy \eqref{intro6} in a static setting (see Section~\ref{subsect:constrainedpbl} for the rigorous $\Gamma$-convergence result). It remains an open question whether a result in the spirit of Giacomini \cite{Gia05} holds also in this case, that is if one can construct a variational evolution for the functionals $F_\e$ and show its convergence, as $\e\to0$, to a quasi-static evolution for the limit cohesive model.

We believe that this approach is also suitable for numerical implementation. An appropriate algorithm should identify the regions where the damage $v_\e$ is concentrated, and impose the monotonicity constraint only at the local minimum points. This might be numerically less expensive than imposing the monotonicity constraint at any point.
{However, it is unclear how to handle the interplay between the asymptotic process of $\Gamma$-convergence and evolution. This problem is closely related to the difficulties that arise in studying the interplay of relaxation and evolution, and which may lead to the definition of complex concepts of solutions \cite{GarroniLarsen,FrancfortGarroni}.}

The generalization to the higher dimensional case poses several challenges. {A first step in this direction has been obtained in \cite{HelMT}, where the $\Gamma$-convergence result in Section~\ref{subsect:constrainedpbl} has been extended to general dimensions, under minimal regularity assumptions on the set $\Gamma$ on which the irreversibility constraint is imposed. The next step would consist in the construction of a quasi-static evolution under the assumption of a prescribed crack path.} {We stress, however, that in the present 1D analysis the position of the crack is not prescribed a priori.}

\medskip\noindent\textbf{Structure of the paper.}
{ The paper is mainly divided into two parts. The first part (Sections~\ref{sect:setting}--\ref{sect:relaxation}) is devoted to the proof of the existence of a quasi-static evolution in Theorem~\ref{thm:evolution}. In particular, in Section~\ref{sect:setting} we state the main assumptions on the cohesive energy density ${\Gdue}$ and we set up the evolution problem; then the construction of a quasi-static evolution for the one-dimensional model is discussed in Section~\ref{sect:evol} (time-discretization) and Section~\ref{sect:evol2} (continuous limit); finally in Section~\ref{sect:relaxation} we compute the relaxation of the cohesive energy.

In the second part of the paper (Sections~\ref{sect:gnew}--\ref{sect:proofs}) we explain the connection of the model with the phase-field approximation. In Section~\ref{sect:gnew} we report the approximation result from \cite{CFI} and we discuss the main properties of the cohesive energy density $\Gzero(s)={\Gdue(s,0)}$ of the {pristine} material; the $\Gamma$-convergence result in \cite{CFI} is also improved with the analysis of the behaviour of recovery sequences in {Section}~\ref{subsect:blowup}. In Section~\ref{sect:barg} we introduce the new surface energy density ${\Gdue}$, and we show in {Section}~\ref{subsect:constrainedpbl} how it can be recovered, via $\Gamma$-convergence, from damage models which include an irreversibility constraint. The proofs of the properties of $\Gzero$ and ${\Gdue}$ in Sections~\ref{sect:gnew} and \ref{sect:barg} are rather technical and are postponed to Section~\ref{sect:proofs}.}

\medskip\noindent\textbf{Notation.}
Along the paper we work in dimension one and for simplicity in the interval $\Omega:=(0,1)$. For a function of bounded variation $u\in\BV(0,1)$ we denote by $Du$ its distributional derivative, which is a bounded Radon measure on $(0,1)$, and by $|Du|$ its total variation. The standard decomposition of $Du$ is given by
\begin{equation*}
Du = u'\Lu + D^cu + \sum_{x\in J_u} [u](x)\delta_x\,,
\end{equation*}
where $u'\in L^1(0,1)$ denotes the density of the absolutely continuous part of $Du$ with respect to the Lebesgue measure $\Lu$, $D^cu$ is the Cantor part of $Du$, $J_u$ is the jump set of $u$ (sometimes also denoted by $J(u)$ in the following), $[u](x):=u^+(x)-u^-(x)$, and $u^+(x)$ and $u^-(x)$ are the {approximate} limits from the right and from the left of $u$ at $x$ respectively. We denote by $\SBV(0,1)$ the space of functions $u\in\BV(0,1)$ such that $D^cu\equiv0$. For the properties of functions of bounded variation we refer to the monograph \cite{AFP}. {We use standard notation for $\Gamma$-convergence, see \cite{Br02,Dalmaso1993}}.



\section{Cohesive quasi-static evolution: general setting and assumptions of the model} \label{sect:setting}

{ We start by introducing the main assumptions on the cohesive energy densities ${\Gdue}(s,s')$ considered in this paper. We refer to the Introduction for the interpretation of the meaning of the two variables of ${\Gdue}$, and to Figure~\ref{figggbarintro} for a visualization of the qualitative behaviour of ${\Gdue}$.

We let $\Gdue\colon [0,+\infty)\times[0,+\infty)\to [0,+\infty)$ satisfy the following assumptions:}
\begin{enumerate}
		\item \label{item1gbar1} ${\Gdue}$ is { continuous and} monotone nondecreasing in both variables, ${\Gdue}(s,s')=\Gdue(s,0)$ if $s\geq s'$;
		\item \label{item3gbar1} ${\Gdue}(s,s') \leq 1$
		{ and $\lim_{s\to+\infty}{\Gdue}(s,s')=1$ {for any $s'$}};
		\item \label{item5gbar1} there exist $\ell,\tilde{\ell}>0$ and $1<p<2$ such that
		\begin{equation} \label{gexpansion1}
		\Gdue(s,0) = \ell s - \tilde{\ell}s^{p} + o(s^{p}) \qquad\text{as }s\to0^+;
		\end{equation}
		\item \label{item2gbar1} for every $s_1,s_2,s'\geq0$,
		\begin{equation} \label{gbarsub1}
		{\Gdue}(s_1+s_2,s') \leq \Gdue(s_1,0) + {\Gdue}(s_2,s')\,;
		\end{equation}
		if in addition  $s_1>0$ and $s_2\vee s'>0$, the inequality is strict.
	\end{enumerate}
Notice that for fixed $s'$, the map ${\Gdue}(\cdot,s')$ is Lipschitz continuous with Lipschitz constant $\ell$ due to its monotonicity, its subadditivity, and the bound {\eqref{gexpansion1}; in particular} $\Gdue(s,0)\leq \ell s$.
Moreover, the global continuity and the strict subadditivity \ref{item2gbar1} of $\Gdue$ imply that for every $\e>0$ there exists a constant $c_{\e}>0$, depending on $\e$, such that
\begin{eqnarray} \label{gsub1}
& \Gdue(s_1+s_2,0)-\Gdue(s_1,0)-\Gdue(s_2,0) \leq - c_\e, \qquad\text{for every } s_1,s_2\geq\e, \\
& \label{gsub21}
{\Gdue}(s_1+s_2,s')-\Gdue(s_1,0)-{\Gdue}(s_2,s') \leq -c_{\e}, \qquad\text{for every } s_2\geq 0 \text{ and } s_1,s'\geq\e.
\end{eqnarray}

{
\begin{remark}
We briefly discuss the role of the previous assumptions. We first observe that the function $\Gzero(s)={\Gdue}(s,0)$ (that is, the cohesive energy of a ``new'' jump) has the expected behaviour of a cohesive energy in the sense of Barenblatt. The assumption on the behaviour of ${\Gdue}$ at infinity in \ref{item3gbar1} is just a normalization condition and can be assumed without loss of generality. The assumption \eqref{gexpansion1} is instead fundamental and guarantees that minimizers of the associated energy {(see \eqref{defPhi} below)} in one dimension have no Cantor part (see Proposition~\ref{prop:nocantor}); we expect that the range of exponents $p\in(1,2)$ is not optimal. Finally, the subadditivity condition \ref{item2gbar1} is a natural assumption and is related to the lower semicontinuity of the associated functional \eqref{defPhi}.
{The condition of strict subadditivity, which means that one larger jump is energetically more efficient than two small jumps which are spatially very close to each other, is important for obtaining regularity of minimizers, see for example Proposition \ref{prop:discretetime}\ref{item3discr} {and Theorem \ref{thm:evolution}\ref{item1ev}}.}
{We further observe that monotonicity and strict subadditivity imply that any fracture which is opened and then closed again dissipates a positive amount of energy, $\Gdue(0,s)>0$ for all $s>0$.}
\end{remark}}

\medskip
{ We next prepare the setting for the construction of a quasi-static evolution.}  Let $T>0$ be the final time. The state of the fracture at each time $t\in[0,T]$ is modeled by a discrete set of points (representing the sites where the crack is localized), to each of which is attached a scalar quantity (representing the maximal amplitude of the jump at all the previous times):
\begin{equation} \label{previousjump}
\Gamma \subset [0,1]\, \text{ finite set,} \qquad s:\Gamma\to(\bar{s},\infty).
\end{equation}
Here $\bar{s}>0$ is a fixed positive threshold (which can be taken arbitrarily small but strictly positive, for technical reasons). The effect of the irreversibility condition affects only the fracture points where the jump amplitude exceeds the threshold $\bar{s}$, since only these points will be included in the crack set $\Gamma$; in other words, for small fractures the process is completely reversible. The {internal} variable $s$ in \eqref{previousjump} will be always implicitly extended as $s(x)=0$ for $x\notin\Gamma$.

The bar is subject to a time-dependent boundary displacement
\begin{equation} \label{assb}
b\in H^1([0,T];\R^2), \qquad\text{with } b(t)=(b(t,0),b(t,1)).
\end{equation}
It is also convenient to denote, for $t$ fixed, by $b(t):[0,1]\to\R$ the affine function interpolating between the boundary values $b(t,0)$ and $b(t,1)$.
We will usually write $b(t,x):=b(t)(x)$ and we will denote by $\dot{b}$ the time-derivative of $b$, by $b'(t)$ the spatial derivative of the function $b(t)$.

We would then iteratively minimize a functional of the form
\begin{equation} \label{defPhi0}
\int_0^1 |u'|^2\de x +{\sum_{x\in\Gamma\cup J_u} {\Gdue}(|[u](x)|,s(x))}
\end{equation}
among all possible displacements $u$ attaining the boundary conditions $u(0)=b(t,0)$, $u(1)=b(t,1)$. The surface energy in the new fracture points $J_u\setminus\Gamma$, where the bar was previously {pristine}, {is $\Gdue(|[u]|,0)$}; at the points of $\Gamma$, where the fracture was present at previous times, an irreversibility condition appears: the surface energy density is given by { ${\Gdue}(|[u]|,s)$, with $s>\bar s$}, which takes into account the previous work made on $\Gamma$.

Since the function { $\Gdue(\cdot,0)$} has finite slope at the origin, the minimization of the functional \eqref{defPhi0} is in principle not well-posed in $\SBV$, and we have to consider its relaxation with respect to the weak*-topology of $\BV$, which is given by
\begin{equation} \label{defPhi}
\Phi(u;\Gamma,s,b) := \int_0^1 h(|u'|)\de x + {\sum_{x\in J_u^b\cup \Gamma} {\Gdue}(|[u](x)|,s(x))} + \ell|D^cu|(0,1)\,,
\end{equation}
where we highlighted the dependence of $\Phi$ on the state of the fracture (described by the pair $(\Gamma,s)$) and on the boundary datum $b$. The rigorous proof of the relaxation result will be given in Section~\ref{sect:relaxation}. Here the elastic energy density $h$ is 
\begin{equation}\label{defh1}
h(\xi):=
{\begin{cases}
\xi^2
& \text{if } |\xi| \leq \frac{\ell}{2},\\
\ell |\xi| - \frac{\ell^2}{4}
& \text{if  } |\xi| > \frac{\ell}{2},
\end{cases}}
\end{equation}
and to take into account the boundary conditions, we {use} the following notation:
\begin{equation}\label{Dirichlet}
J_u^b := J_u \cup \bigl\{ x\in\{0,1\} \,:\, u(x)\neq b(x) \bigr\}\,,
\end{equation}
\begin{equation} \label{Dirichlet2}
[u](0) := u^+(0)-b(0)\,,\qquad [u](1) := b(1)- u^-(1)\,.
\end{equation}
Finally, we also add to the total energy a lower-order term of the form $\gamma\|u-w(t)\|_{L^2(0,1)}^2$, where $\gamma\geq0$ is a fixed constant and
\begin{equation} \label{assw}
w\in\AC([0,T];L^\infty(0,1))
\end{equation}
is a given map, absolutely continuous in time.


\section{Cohesive quasi-static evolution: the time-discrete evolution} \label{sect:evol}

{ We construct in this section a time-discrete evolution for the one-dimensional cohesive fracture model, relative to a given time-dependent boundary displacement $b(t)$. We assume in the following that the boundary displacement $b$ satisfying the assumption \eqref{assb} is given, as well as the lower-order term $w$ satisfying \eqref{assw} {and that a  positive threshold $\bar s>0$ for the minimal opening which leads to irreversibility is fixed}.}

We fix a discretization step $\tau>0$ and we consider a subdivision of $[0,T]$ of the form
\begin{equation*}
0=t_0<t_1<\ldots<t_{N_\tau}<t_{N_\tau+1}=T,
\end{equation*}
where $N_\tau$ is the largest integer such that $\tau N_\tau< T$, and $t_k:=k\tau$ for $k\in\{0,\ldots,N_\tau\}$. We also set $b^\tau_k=b(t_k)$, $w^\tau_k=w(t_k)$.

The time-discrete evolution in this setting is defined as follows. For $k=0$, select a solution $u^\tau_0\in\BV(0,1)$ of the minimum problem
\begin{align} \label{min0}
\min\biggl\{ \Phi(u;b^\tau_0) + \gamma\int_0^1|u-w_0^\tau|^2\de x \,:\, u\in \BV(0,1) \biggr\}
\end{align}
(with the convention that we do not indicate the dependence on $(\Gamma,s)$ in the functional $\Phi$ defined in \eqref{defPhi} when $\Gamma=\emptyset$).
We set
\begin{equation} \label{min0bis}
\Gamma^\tau_0 := \bigl\{ x\in J^{b^\tau_0}(u^\tau_0) : |[u^\tau_0](x)|>\bar{s} \bigr\}\,, \qquad s^\tau_0(x):=|[u^\tau_0](x)| \,,
\end{equation}
where $\bar{s}>0$ is the threshold fixed at the beginning of this section.
Assume now to have constructed $u^\tau_i\in\BV(0,1)$ and pairs $(\Gamma^\tau_i,s^\tau_i)$ in the form \eqref{previousjump}, for $i=0,\ldots,k-1$. We let $u^\tau_k\in\BV(0,1)$ be a minimum point of
\begin{align} \label{mink}
\min\biggl\{ \Phi(u;\Gamma^\tau_{k-1},s^\tau_{k-1},b^\tau_k) + \gamma\int_0^1|u-w^\tau_k|^2\de x \,:\, u\in \BV(0,1) \biggr\} \,,
\end{align}
and we set
\begin{equation} \label{minkbis}
\begin{split}
\Gamma^\tau_k &:= \Gamma^\tau_{k-1} \cup \bigl\{ x\in J^{b^\tau_k}(u^\tau_k) : |[u^\tau_k](x)|>\bar{s} \bigr\} \,, \\
s^\tau_k(x) &:=
\begin{cases}
|[u^\tau_k](x)|\vee s^\tau_{k-1}(x) & \text{if }x\in\Gamma^\tau_{k-1}, \\
|[u^\tau_k](x)| & \text{if }x\in \Gamma^\tau_k\setminus \Gamma^\tau_{k-1}.
\end{cases}
\end{split}
\end{equation}
Notice in particular that all the sets $\Gamma^\tau_k$ are finite, since at each step the minimizer $u_k^\tau$ has at most a finite number of jump points where the amplitude of the jump is larger than the fixed threshold $\bar{s}$.
The well-posedness of this construction is proved in the following proposition.
{We show below that $u_k^\tau\in \SBV(0,1)$.}

\begin{proposition} \label{prop:discretetime}
	For all $k=0,1,\ldots N_\tau$ there exists $u^\tau_k\in\BV(0,1)$ such that, by defining $(\Gamma^\tau_k,s^\tau_k)$ as in \eqref{minkbis} (with $\Gamma^\tau_{-1}=\emptyset$), the following hold:
	\begin{enumerate}
		\item\label{item1discr} $u^\tau_k$ is a minimizer of problem \eqref{mink};
		\item\label{item2discr} $\|u^\tau_k\|_\infty \leq \max\bigl\{ \|b^\tau_k\|_\infty, \|w^\tau_k\|_\infty \bigr\}$;
		\item\label{item3discr} $\inf_{\tau} \inf_{k} \min \bigl\{ |x-y| \,:\, x,y\in\Gamma^\tau_k,\, x\neq y \bigr\} >0$.
	\end{enumerate}
\end{proposition}

\begin{proof}
{As this is a static result, we fix $\tau$ and $k$ and drop them from the notation.}
We need to show that, given a pair $(\Gamma,s)$ as in \eqref{previousjump}, an affine function $b:[0,1]\to\R$, and $w\in L^\infty(0,1)$, the minimum problem
	\begin{equation} \label{minphi}
	\min\biggl\{ \Phi(u;\Gamma,s,b) + \gamma\int_0^1|u-w|^2\de x \,:\, u\in \BV(0,1) \biggr\}
	\end{equation}
	has a solution. Let $(u_n)_n$ be a minimizing sequence for problem \eqref{minphi}. By a truncation argument we can take $\|u_n\|_{\infty}\leq\max\{\|b\|_\infty,\|w\|_\infty\}$. By comparing $u_n$ with the function $b$ {(and possibly replacing $u_n$ with $b$)} we obtain the uniform bound
	\begin{equation*}
	\Phi(u_n;\Gamma,s,b)+\gamma\int_0^1|u_n-w|^2\de x \leq \int_0^1 h(|b'|)\de x + \sum_{x\in\Gamma}{\Gdue}(0,s(x)) + \gamma\int_0^1|b-w|^2\de x  \leq C.
	\end{equation*}
	Let now ${s_\ast}>0$ be such that { $\Gdue({s_\ast},0)=\frac12$}, and let ${\ell_\ast}\in(0,\ell)$ be such that  {$\Gdue(s,0)\geq{\ell_\ast}s$} for all $s\in[0,{s_\ast}]$. Then the lower estimates $h(t)\geq\ell t - \frac{\ell^2}{4}$, { ${\Gdue}(t,s)\geq \Gdue(t,0)$} yield
	\begin{align} \label{boundbv}
	C
	& \geq \Phi(u_n;\Gamma,s,b)
	\geq \ell\int_0^1 |u_n'|\de x -\frac{\ell^2}{4} + \ell|D^cu_n|(0,1) + \sum_{x\in J_{u_n}^b} {\Gdue(|[u_n](x)|,0)} \nonumber\\
	& \geq \ell\int_0^1 |u_n'|\de x -\frac{\ell^2}{4} + \ell|D^cu_n|(0,1) + \sum_{|[u_n](x)|\leq {s_\ast}} {\ell_\ast}|[u_n](x)| + \frac12\mathcal{H}^0\bigl(\{|[u_n](x)|>{s_\ast}\}\bigr) \\
	& \geq \ell\int_0^1 |u_n'|\de x -\frac{\ell^2}{4} + \ell|D^cu_n|(0,1) + \sum_{|[u_n](x)|\leq {s_\ast}} {\ell_\ast}|[u_n](x)| + \frac{1}{4\|u_n\|_\infty} \sum_{|[u_n](x)|>{s_\ast}} |[u_n](x)| \,.\nonumber
	\end{align}
	Therefore the sequence $(u_n)_n$ is bounded in $\BV(0,1)$ and, up to subsequences, it converges weakly* in $\BV$ to some function $u\in\BV(0,1)$, which is a minimizer of \eqref{minphi} by the lower semicontinuity of the functional $\Phi$, proved in Theorem~\ref{thm:relaxation} {below}. In particular, it also follows that $\|u\|_\infty\leq\max\{\|b\|_\infty,\|w\|_\infty\}$.
	
	It remains to prove \ref{item3discr}. We first show that there is a minimal distance between any two new fracture points in $\Gamma^\tau_k\setminus\Gamma^\tau_{k-1}$, independent of $k$ and $\tau$. Consider $x_1,x_2\in J^{b^\tau_k}(u^\tau_k)\setminus \Gamma^\tau_{k-1}$, with $s_i:=[u^\tau_k](x_i)$ and $|s_i|>\bar{s}$, $i=1,2$. Suppose also $x_1<x_2$. Let $v:=u^\tau_k+s_2\chi_{(x_1,x_2)}$: by minimality of $u^\tau_k$ in problem \eqref{mink} we have
	\begin{equation*}
	0 \leq { \Gdue(|s_1+s_2|,0) - \Gdue(|s_1|,0) - \Gdue(|s_2|,0)} + 2\gamma s_2\int_{x_1}^{x_2} (u^\tau_k-w^\tau_k)\de x + \gamma s_2^2|x_1-x_2| \leq -c_{\bar{s}} + C|x_1-x_2|,
	\end{equation*}
	where the last inequality { follows} {from \eqref{gsub1}} and the uniform $L^\infty$-bound on $u^\tau_k$ and $w$.
	This proves that the distance between $x_1$ and $x_2$ is larger than a uniform positive constant.
	
	Similarly, we prove that any new fracture point of $u^\tau_k$ {cannot} be too close to the points of $\Gamma^\tau_{k-1}$. Consider $x_1\in J^{b^\tau_k}(u^\tau_k)\setminus\Gamma^\tau_{k-1}$ with $s_1:=[u^\tau_k](x_1)$, $|s_1|>\bar{s}$. Let $x_2\in\Gamma^\tau_{k-1}$ and {let} $s_2:=[u^\tau_k](x_2)$. Assuming without loss of generality that $x_1<x_2$ (the construction in the other case is symmetric), by comparing the energy of $u^\tau_k$ and $v:=u^\tau_k-s_1\chi_{(x_1,x_2)}$ we have
	\begin{equation*}
	\begin{split}
	0 &\leq {\Gdue}(|s_1+s_2|,{s^\tau_{k-1}}(x_2)) - { \Gdue(|s_1|,0)} - {\Gdue}(|s_2|,{s^\tau_{k-1}}(x_2)) - 2\gamma s_1\int_{x_1}^{x_2} (u^\tau_k-w^\tau_k)\de x + \gamma s_1^2|x_1-x_2| \\
	&\leq -c_{\bar{s}} + C|x_1-x_2|,
	\end{split}
	\end{equation*}
	where the last inequality { follows from \eqref{gsub21}} and the uniform $L^\infty$-bound on $u^\tau_k$ and $w$.
\end{proof}

\begin{proposition}\label{prop:nocantor}
	{Any minimizer $u$ of \eqref{mink}  obeys $D^cu=0$ and $|u'|\le \frac\ell2$ almost everywhere.}
\end{proposition}
\begin{proof}
	{We define the measure $\mu:=(u'-\frac\ell2)_++(D^cu)_+$. We shall  show that $\mu=0$, 
		the argument with $(u'+\frac\ell2)_-+(D^cu)_-$ is identical.
		
		Assume that there are  $a\ne b\in(0,1)$ such that $\mu((a-\e,a+\e))>0$ and $\mu((b-\e,b+\e))>0$ for all $\e>0$ sufficiently small. 
		If such a pair does not exist, then the measure $\mu$ is concentrated on $\{0,1\}$ and at most one additional point, and since $\mu$ does not contain any Dirac measure 
		we obtain $\mu=0$ and we are done. Let $I_\e:=(a-\e,a+\e)$, $J_{\e'}:=(b-\e', b+\e')$,
		where $\e,\e'$ are positive numbers such that these two intervals are disjoint subsets of $(0,1)$.
		We set 
		\begin{equation*}
		\psi_{\e,\e'}(x):=\nu_{\e,\e'}((0,x)), \hskip5mm\text{with}\hskip5mm
		\nu_{\e,\e'}:= \frac{1}{\mu(I_\e)} \mu\mres I_\e -  \frac{1}{\mu(J_{\e'})} \mu\mres J_{\e'}.
		\end{equation*}
		We compare $u$ with $u+\rho\psi_{\e,\e'}$ for some $\rho$ chosen below. This gives, since 
		$\psi_{\e,\e'}\in \BV(0,1)$ with $\psi_{\e,\e'}(0)=\psi_{\e,\e'}(1)=0$ and $J_{\psi_{\e,\e'}}=\emptyset$,
		\begin{equation*}
		\begin{split}
		0\le& \int_0^1 \left[ h(u'+\rho\psi'_{\e,\e'})-h(u')\right] \de x +\ell |D^c(u+\rho\psi_{\e,\e'})|(0,1)-\ell |D^cu|(0,1)\\
		&+\gamma\int_0^1 \left[(u-{w_k^\tau}+\rho\psi_{\e,\e'})^2-(u-{w_k^\tau})^2\right] \de x .
		\end{split}  
		\end{equation*}
		We observe that
		\begin{equation*}
		u'+\rho\psi'_{\e,\e'}-\frac\ell2=u'-\frac\ell2 + \frac{\rho}{\mu(I_\e)} \chi_{I_\e}(u'-\frac\ell2)_+
		- \frac{\rho}{\mu(J_{\e'})} \chi_{J_{\e'}} (u'-\frac\ell2)_+ 
		\end{equation*}
		where $\chi_{I_\e}$ and $\chi_{J_{\e'}}$ are the characteristic functions of the two intervals.
		Therefore for any $\rho\in (-\mu(I_\e),\mu(J_{\e'}))$ we have $ u'+\rho\psi'_{\e,\e'}-\frac\ell2\ge0$ whenever
		$\rho\psi'_{\e,\e'}\ne0$, so that
		$h(u'+\rho\psi_{\e,\e'}')-h(u')=\ell\rho\psi_{\e,\e'}'$.
		{Analogously, for the same set of values of $\rho$ we have 
			\begin{equation*}
			(D^cu)_++\rho D^c\psi_{\e,\e'}=\left(1+\frac\rho{\mu(I_\e)}\chi_{I_\e} - \frac{\rho}{\mu(J_{\e'})}\chi_{J_{\e'}}\right) (D^cu)_+ \ge0
			\end{equation*}
			which implies $|D^c(u+\rho\psi_{\e,\e'})|(0,1)= |D^cu|(0,1)+\rho D^c\psi_{\e,\e'}(0,1)$.}
		The above expression then becomes
		\begin{equation*}
		\begin{split}
		0\le&
		\int_0^1 \ell \rho\psi_{\e,\e'}'\de x + \ell \rho D^c\psi_{\e,\e'}(0,1)
		+\gamma \int_0^1 \left[2\rho (u-{w_k^\tau})\psi_{\e,\e'} + \rho^2\psi_{\e,\e'}^2\right]\de x.
		\end{split}  
		\end{equation*}
		The sum of  the first two terms is $\ell\rho\nu_{\e,\e'}(0,1)=0$, and since $\rho$ can still be chosen 
		arbitrarily small we conclude that $\int_0^1 (u-{w_k^\tau})\psi_{\e,\e'} \de x=0$ for all 
		admissible $\e, \e'$. By Fubini's theorem we have
		\begin{equation*}
		\begin{split}
		0=&\int_0^1 (u-{w_k^\tau})(x)\psi_{\e,\e'}(x)\de x = 
		\int_0^1 (u-{w_k^\tau})(x)\int_0^x \de \nu_{\e,\e'}(y) \de x = \int_0^1 F(y) \de\nu_{\e,\e'}(y)
		\end{split}
		\end{equation*}
		for all $\e,\e'$, where $F(y):=\int_y^1(u-{w_k^\tau})(x)\de x$ is continuous.}
	
	We now consider the measure $\hat\nu_\e:=\frac{1}{\mu(I_\e)} \mu\mres I_\e -\delta_b$ and set 
	$\hat\psi_\e(x):=\hat\nu_\e((0,x))$. 
	Since $\nu_{\e,\e'}\wto\hat\nu_\e$ as $\e'\to0$, and $F$ is continuous, we obtain 
	$\int_0^1 F\de\hat\nu_\e=0$.
	We compare
	$u$ with $u+\rho\hat\psi_\e$ and obtain
	\begin{equation*}
	\begin{split}
	0\le& \int_0^1 \left[ h(u'+\rho\hat\psi_\e')-h(u')\right] \de x +\ell |D^c(u+\rho\hat\psi_\e)|(0,1)-\ell |D^cu|(0,1)\\
	&+\gamma\int_0^1 \left[(u-{w_k^\tau}+\rho\hat\psi_\e)^2-(u-{w_k^\tau})^2\right] \de x 
	+\Gdue(|[u](b)-\rho|,s^\tau_{k-1}(b))-\Gdue(|[u](b)|,s^\tau_{k-1}(b))
	\end{split}  
	\end{equation*}
	where as usual $s^\tau_{k-1}(b)=0$ if $b\not\in \Gamma_{k-1}^\tau$.
	By the same computation as above,  using that $\hat\nu_\e(I_\e)=1$,
	\begin{equation*}
	\begin{split}
	0\le& \rho\ell+2\gamma\rho \int_0^1 F \de\hat\nu_\e + \gamma\rho^2\int_0^1\hat\psi_\e^2\de x
	+\Gdue(|[u](b)-\rho|,s^\tau_{k-1}(b))-\Gdue(|[u]|(b),s^\tau_{k-1}(b)).
	\end{split}  
	\end{equation*}
	{ By \eqref{gbarsub1} and \eqref{gexpansion1}} we have {for $|\rho|$ small}
	\begin{equation*}
	\Gdue(|[u](b)-\rho|,s^\tau_{k-1}(b))\le \Gdue(|[u]|(b),s^\tau_{k-1}(b))+ { \Gdue(|\rho|,0)}
	\le \Gdue(|[u]|(b),s^\tau_{k-1}(b))+ \ell |\rho| - { \tilde\ell |\rho|^{p} +o(|\rho|^{p})}.
	\end{equation*}
	Recalling $\int_0^1 F\de\hat\nu_\e=0$ we have
	\begin{equation*}
	\begin{split}
	0\le& \rho\ell+\ell|\rho|- { \tilde\ell |\rho|^{p} +o(|\rho|^{p})} + \gamma\rho^2\int_0^1\hat\psi_\e^2\de x
	\end{split}  
	\end{equation*}
	for all $\rho$ sufficiently small (positive or negative). This is a contradiction, hence we conclude that $\mu=0$.
\end{proof}

Proposition~\ref{prop:discretetime} allows to construct a \textit{piecewise constant evolution}, relative to the discrete boundary values $b^\tau_k=b(t_k)$ and to $w^\tau_k:=w(t_k)$,
\begin{equation} \label{interpolant0}
t\mapsto \bigl( u_\tau(t), \Gamma_\tau(t), s_\tau(t) \bigr) \qquad\text{for }t\in[0,T],
\end{equation}
where the piecewise constant interpolations in time are defined as
\begin{equation} \label{interpolant1}
u_\tau(t) := u^\tau_k\,, \quad {\Gamma_\tau(t) := \Gamma^\tau_{k}\,, \quad s_\tau(t) := s^\tau_{k}}
\qquad\qquad \text{for } t\in[t_k, t_{k+1})\,,
\end{equation}
with $k=0,1,\ldots, N_\tau$,
{with $\Gamma^\tau_{-1}=\emptyset$, $s^\tau_{-1}=0$}. We also consider the piecewise constant and the piecewise affine interpolations in time of the boundary data, which are respectively defined as
\begin{equation} \label{interpolant2}
b_\tau(t) := b^\tau_k\,, \qquad b^\tau(t) := b^\tau_k + \frac{t-t_k}{\tau}(b^\tau_{k+1}-b^\tau_k) \qquad \text{for } t\in[t_k, t_{k+1})\,,
\end{equation}
and similarly
\begin{equation} \label{interpolant3}
w_\tau(t) := w^\tau_k\,, \qquad w^\tau(t) := w^\tau_k + \frac{t-t_k}{\tau}(w^\tau_{k+1}-w^\tau_k) \qquad \text{for } t\in[t_k, t_{k+1})\,.
\end{equation}
In the following we use the notation $u_\tau(t,x):=(u_\tau(t))(x)$ (and similarly for the other functions introduced above).
{We recall that we extend ${s_k^\tau}$ by $0$ outside ${\Gamma_k^\tau}$.}

{We observe that for any $k$
	\begin{equation}\label{eqPhikkm}
	\Phi(u^\tau_k;\Gamma^\tau_{k-1},s^\tau_{k-1},b^\tau_k) =\Phi(u^\tau_k;\Gamma^\tau_{k},s^\tau_{k},b^\tau_k)  .
	\end{equation}
	To see this, it suffices to show that
	\begin{equation*}
	{\Gdue}\bigl( |[u^\tau_{k}](x)|,s^\tau_{k}(x) \bigr)
	=  {\Gdue}\bigl( |[u^\tau_{k}](x)|,s^\tau_{k-1}(x) \bigr) 
	\text{ for any $x\in  J^{b^\tau_{k}}(u^\tau_{k})\cup \Gamma^\tau_{k}$.}
	\end{equation*}
	This is true, since in
	\eqref{minkbis} we defined $s^\tau_{k}(x)=s^\tau_{k-1}(x)\vee |[u^\tau_{k}]|(x)$.
	If $s^\tau_{k-1}(x)\ge |[u^\tau_{k}]|(x)$ then
	$s^\tau_{k-1}(x)=s^\tau_{k}(x)$ and the two terms are identical. If not,
	then $s^\tau_{k}(x)= |[u^\tau_{k}]|(x)$ and,
	{ recalling the assumption \ref{item1gbar1} on $\Gdue$ in Section~\ref{sect:setting}}, we have 
	$   {\Gdue}\bigl( |[u^\tau_{k}]|(x),s^\tau_{k}(x) \bigr)=
	{ \Gdue(|[u^\tau_{k}]|(x),0)}
	= {\Gdue}\bigl( |[u^\tau_{k}]|(x),s^\tau_{k-1}(x) \bigr)$.
	Therefore \eqref{eqPhikkm} holds.
}

We are interested in the evolution of the discrete total energy at time $t$, with timestep $\tau$, which is defined for $t\in[0,T]$ by
\begin{equation} \label{energydisc}
\begin{split}
\mathcal{E}_\tau(t)
& := \Phi(u_\tau(t);\Gamma_\tau(t), s_\tau(t), b_\tau(t))+ \gamma\int_0^1 |u_\tau(t)-w_\tau(t)|^2\de x \,.
\end{split}
\end{equation}

\begin{lemma} \label{lem:energy}
	For any $t\in[0,T]$, letting $t_k$ be the discretization point such that $t\in[t_k,t_{k+1})$, we have
	\begin{equation} \label{energydisc2}
	\mathcal{E}_\tau(t)
	\leq \mathcal{E}_\tau(0) + \int_{0}^{t_k}\int_0^1 \Bigl( h'((u_\tau(r))')(\dot{b}^\tau(r))' + 2\gamma (u_\tau(r)-w_\tau(r))(\dot{b}^\tau(r)-\dot{w}^\tau(r)) \Bigr)\de x\de r + R_\tau(t),
	\end{equation}
	where
	\begin{multline} \label{energydisc3}
	R_\tau(t) := \int_0^{t_k} \biggl( \tau \int_0^1|(\dot{b}^\tau(r))'|^2\de x + \gamma\tau\int_0^1 |\dot{b}^\tau(r)|^2\de x \\
	{+ \gamma \left|\int_0^1 (w_\tau(r+\tau)-w_\tau(r)) \dot{w}^\tau(r)\de x  \right|
		+ 2\gamma \left|\int_0^1 (w_\tau(r+\tau)-w_\tau(r)) \dot{b}^\tau(r)\de x  \right| \biggr)\de r .}
	\end{multline}
\end{lemma}

\begin{proof}
	We compare $u^\tau_k$ with ${v:=}u^\tau_{k-1}+b^\tau_k-b^\tau_{k-1}$: by minimality of $u^\tau_k$ in problem \eqref{mink} we have {by Proposition~\ref{prop:nocantor}}
	\begin{align*}
	\Phi(u^\tau_k;&\Gamma^\tau_{k-1},s^\tau_{k-1},b^\tau_k) + \gamma\int_0^1 |u^\tau_k-w^\tau_k|^2\de x \nonumber\\
	& \leq \int_0^1 h\bigl( |( u^\tau_{k-1}+b^\tau_k-b^\tau_{k-1} )'| \bigr)\de x 
	+ \gamma\int_0^1 |u^\tau_{k-1}+b^\tau_k-b^\tau_{k-1}-w^\tau_k|^2\de x \nonumber\\
	& \qquad + \sum_{x\in\Gamma^\tau_{k-1}{\cup  J^{b^\tau_{k-1}}(u^\tau_{k-1})}} {\Gdue}\bigl( |[u^\tau_{k-1}](x)|,s^\tau_{k-1}(x) \bigr) . 
	\end{align*}
	{If $x\in\{0,1\}$, we remark that $v(x)-b^\tau_k(x)=u^\tau_{k-1}(x)-b^\tau_{k-1}(x)$, so that
		the notation $[u^\tau_{k-1}](x)$ is appropriate. 
		Therefore, recalling \eqref{eqPhikkm},}
	\begin{align} \label{proofenergydisc}
	\Phi(u^\tau_k;&\Gamma^\tau_{k-1},s^\tau_{k-1},b^\tau_k) + \gamma\int_0^1 |u^\tau_k-w^\tau_k|^2\de x \nonumber\\
	& \le \Phi(u^\tau_{k-1};\Gamma^\tau_{k-2},s^\tau_{k-2},b^\tau_{k-1}) + \gamma \int_0^1 |u^\tau_{k-1}-w^\tau_{k-1}|^2\de x \nonumber\\
	& \qquad + \int_0^1 \Bigl[ h\bigl( |( u^\tau_{k-1}+b^\tau_k-b^\tau_{k-1} )'| \bigr) - h\bigl(|( u^\tau_{k-1} )'|\bigr) \Bigr] \de x \nonumber\\
	& \qquad + \gamma\int_0^1 \Bigl( |u^\tau_{k-1}+b^\tau_k-b^\tau_{k-1}-w^\tau_k|^2 - |u^\tau_{k-1}-w^\tau_{k-1}|^2 \Bigr) 
	\,.
	\end{align}
	{Since $h'$ is  a Lipschitz function with $|h''|\le 2$, the mean-value theorem gives
		\begin{equation*}
	{	h(x+y)\le h(x)+y h'(x) +y^2 \hskip5mm\text{ for all } x,y\in\R.}
		\end{equation*}
		In particular,
		\begin{equation*}
		h\bigl(( u^\tau_{k-1}+b^\tau_k-b^\tau_{k-1} )'\bigr) - h\bigl(( u^\tau_{k-1} )'\bigr)
		\leq h'((u_{k-1}^\tau)')(b^\tau_{k}-b^\tau_{k-1})' + |(b^\tau_k-b^\tau_{k-1})'|^2 \,. 
		\end{equation*}}
	
	{Similarly,}
	\begin{multline*}
	|u^\tau_{k-1}+b^\tau_k-b^\tau_{k-1}-w^\tau_k|^2 - |u^\tau_{k-1}-w^\tau_{k-1}|^2 \\
	= (w^\tau_k+w^\tau_{k-1}-2u^\tau_{k-1})(w^\tau_k-w^\tau_{k-1}) + 2(u^\tau_{k-1}-w^\tau_k)(b^\tau_k-b^\tau_{k-1}) + |b^\tau_k-b^\tau_{k-1}|^2 \,.
	\end{multline*}
	By inserting these inequalities into \eqref{proofenergydisc} and iterating we find
	\begin{equation*}
	\begin{split}
	\Phi(u^\tau_k;\Gamma^\tau_{k-1},s^\tau_{k-1},b^\tau_k) &+ \gamma\int_0^1 |u^\tau_k-w^\tau_k|^2\de x
	\leq \Phi(u^\tau_0;b^\tau_0) + \gamma\int_0^1 |u^\tau_0-w^\tau_0|^2 \de x \\
	& \qquad + \sum_{i=1}^k \biggl( \int_0^1 h'((u_{i-1}^\tau)')(b^\tau_{i}-b^\tau_{i-1})' \de x + \int_0^1|(b^\tau_i-b^\tau_{i-1})'|^2\de x \\
	& \qquad\qquad + \gamma\int_0^1 (w^\tau_i+w^\tau_{i-1}-2u^\tau_{i-1})(w^\tau_i-w^\tau_{i-1})\de x \\
	& \qquad\qquad + 2\gamma\int_0^1 (u^\tau_{i-1}-w^\tau_i)(b^\tau_i-b^\tau_{i-1}) \de x + \gamma\int_0^1  |b^\tau_i-b^\tau_{i-1}|^2\de x \biggr)\,.
	\end{split}
	\end{equation*}
	We can rewrite this inequality in terms of the interpolants defined in \eqref{interpolant1}, \eqref{interpolant2}, \eqref{interpolant3}: for any $t\in[0,T]$, denoting by $t_k$ the discretization point such that $t\in[t_k, t_{k+1})$, and observing that
	{\eqref{eqPhikkm} and \eqref{energydisc} give}
	\begin{equation*}
	\mathcal{E}_\tau(t)
	=\Phi(u^\tau_k;\Gamma^\tau_{k-1},s^\tau_{k-1},b^\tau_k) + \gamma\int_0^1 |u^\tau_k-w^\tau_k|^2\de x \,,
	\end{equation*}
	we have
	\begin{equation*}
	\begin{split}
	\mathcal{E}_\tau(t)
	&\leq \mathcal{E}_\tau(0) + \int_{0}^{t_k} \biggl( \int_0^1 h'((u_\tau(r))')(\dot{b}^\tau(r))'\de x + \tau \int_0^1|(\dot{b}^\tau(r))'|^2\de x \\
	& \qquad + \gamma\int_0^1 (w_\tau(r+\tau)+w_\tau(r))\dot{w}^\tau(r)\de x -2\gamma\int_0^1 w_\tau(r+\tau)\dot{b}^\tau(r)\de x \\
	& \qquad + 2\gamma\int_0^1 u_\tau(r) (\dot{b}^\tau(r)-\dot{w}^\tau(r)) \de x + \gamma\tau\int_0^1 |\dot{b}^\tau(r)|^2\de x \biggr) \de r \,.
	\end{split}
	\end{equation*}
	{This concludes the proof.}
\end{proof}


\section{Cohesive quasi-static evolution: the time-continuous evolution} \label{sect:evol2}

The goal of this section is to pass to the limit in the time-discrete evolution as the time-step $\tau$ goes to zero.
As in the previous section, we fix a positive threshold $\bar{s}>0$ for the jumps, the final time $T>0$, a time-dependent boundary displacement $b\in H^1([0,T];\R^2)$, and a lower order term $w\in\AC([0,T];L^\infty(0,1))$. Let $(\tau_n)_{n\in\N}$ be a decreasing sequence of time-discretization steps with $\tau_n\to0$ as $n\to\infty$. Correspondingly, let
\begin{equation*}
0=t_0^n<t_1^n<\ldots<t^n_{N_n}<t^n_{N_n+1}=T
\end{equation*}
be the subdivision of $[0,T]$ with $t_k^n:=k\tau_n$ for $k\in\{0,\ldots,N_n\}$. Let
\begin{equation} \label{piecewiseevolution}
t\mapsto \bigl( u_n(t):=u_{\tau_n}(t),\, \Gamma_n(t):= \Gamma_{\tau_n}(t),\, s_n(t) := s_{\tau_n}(t) \bigr), \qquad \text{for }t\in[0,T],
\end{equation}
be the piecewise constant interpolation, defined in \eqref{interpolant1}, of a time-discrete evolution relative to the boundary data $b^n_k=b(t^n_k)$ and to $w^n_k=w(t^n_k)$, as constructed in the previous section. Let also
\begin{equation}  \label{piecewiseevolution2}
b_n(t):=b_{\tau_n}(t), \qquad b^n(t):=b^{\tau_n}(t), \qquad w_n(t):=w_{\tau_n}(t), \qquad w^n(t):=w^{\tau_n}(t) 
\end{equation}
be the piecewise constant and the piecewise affine interpolants of the maps $b$ and $w$, according to \eqref{interpolant2}--\eqref{interpolant3}. We also consider the associated energy $\mathcal{E}_n(t) := \mathcal{E}_{\tau_n}(t)$, see \eqref{energydisc}. Our main result is the following.

\begin{theorem}[Existence of a quasi-static evolution]\label{thm:evolution}
	Under the assumptions of Section~\ref{sect:setting}, with $b$ and $w$ as in \eqref{assb} and \eqref{assw}, there exists $\bigl(u(t),\Gamma(t),s(t)\bigr)$, for $t\in[0,T]$, with the following properties:
	\begin{enumerate}
		\item\label{item1ev} $u(t)\in\BV(0,1)$, $\sup_{t\in[0,T]}\|u(t)\|_{\infty}\leq \sup_{t\in[0,T]}\max\{ \|b(t)\|_\infty, \|w(t)\|_\infty \}$, $\Gamma(t)\subset[0,1]$ is a finite set, $s(t):\Gamma(t)\to[\bar{s},\infty)$, {$\sup_{t\in[0,T]}\|s(t)\|_{\infty}\leq 2\sup_{t\in[0,T]}\max\{ \|b(t)\|_\infty, \|w(t)\|_\infty \}$};
		\item\label{item2ev} (irreversibility) $\Gamma(t_1)\subset\Gamma(t_2)$ and $s(t_1)\leq s(t_2)$ for all $0\leq t_1\leq t_2\leq T$;
		\item\label{item4ev} $u(0)$ minimizes the functional $\Phi(v;b(0))+\gamma\|v-w(0)\|^2_{L^2(0,1)}$ among all $v\in\BV(0,1)$, and $\Gamma(0)=\{ x\in J^{b(0)}(u(0)) \,:\, |[u(0)](x)|>\bar{s} \}$, $s(0)=|[u](0)|$ on $\Gamma(0)$;
		\item\label{item5ev} $\{ x\in J^{b(t)}(u(t)) : |[u(t)](x)|>\bar{s} \} \subset\Gamma(t)$ and $|[u](t)|\leq s(t)$ on $\Gamma(t)$, for all $t\in(0,T]$;
		\item\label{item6ev} (static equilibrium) for all $t\in(0,T]$, $u(t)$ minimizes the functional
		$$
		\Phi(v;\Gamma(t),s(t),b(t)) + \gamma\int_0^1 |v-w(t)|^2\de x
		$$
		among all $v\in\BV(0,1)$;
		\item\label{item7ev} (non-dissipativity) the total energy
		$$
		\mathcal{E}(t) := \Phi(u(t);\Gamma(t),s(t),b(t)) + \gamma\int_0^1 |u(t)-w(t)|^2\de x
		$$
		satisfies for every $t\in[0,T]$
		\begin{equation} \label{dissip}
		\mathcal{E}(t) = \mathcal{E}(0) + \int_{0}^{t}\int_0^1 \Bigl( h'((u(r))')(\dot{b}(r))' + 2\gamma (u(r)-w(r))(\dot{b}(r)-\dot{w}(r)) \Bigr)\de x\de r .
		\end{equation}
	\end{enumerate}
\end{theorem}

\begin{corollary} {Theorem~\ref{thm:evolution} implies that}
	$(u(t),\Gamma(t),s(t))$ minimizes the functional
	$$
	\Phi(v;\Gamma,s,b(t)) + \gamma\int_0^1 |v-w(t)|^2\de x
	$$
	among all $(v,\Gamma,s)$ such that $\Gamma$ is a finite set with $\Gamma(t)\subset\Gamma$, $s:\Gamma\to[\bar{s},\infty)$ satisfies $s\geq s(t)$ on $\Gamma(t)$, and $v\in\BV(0,1)$ is such that $\{ x\in J^{b(t)}(v) : |[v](x)|>\bar{s}  \} \subset\Gamma$.
\end{corollary}
\begin{proof}
	{By {Theorem~\ref{thm:evolution}\ref{item1ev} and \ref{item5ev}}, 
		$(u(t),\Gamma(t),s(t))$ is admissible. Let now
		$(v,\Gamma,s)$ be admissible. 
		By monotonicity of $\Gdue$ we have
		$\Phi(v;\Gamma(t),s(t),b(t)) 
		\le\Phi(v;\Gamma,s,b(t)) $, and 
{Theorem~\ref{thm:evolution}}\ref{item6ev} then concludes the proof. }
\end{proof}

\begin{proof}[Proof of Theorem~\ref{thm:evolution}]
	We divide the proof into several steps. In the following, we will denote by $C$ a generic constant, possibly depending on $b$, $w$, and $\gamma$, but independent of $n$ and $t$, which might change from line to line.
	As a preliminary remark, we observe that $u_n(t)$ is a minimizer of problem \eqref{mink}, for the value of $k$ such that $t\in[t^n_k,t^n_{k+1})$; then 
	\begin{equation*}
	\begin{split}
	&{\Phi(u_n(t);\Gamma_n(t-\tau_n),s_n(t-\tau_n),b_n(t)) }+ \gamma\int_0^1 |u_n(t)-w_n(t)|^2\de x \\
	&\leq \Phi(v;{\Gamma_n(t-\tau_n),s_n(t-\tau_n)},b_n(t)) + \gamma\int_0^1 |v-w_n(t)|^2\de x
	\qquad\text{for all $v\in\BV(0,1)$.}
	\end{split}\end{equation*}
	Using \eqref{eqPhikkm} on the left-hand side, and monotonicity on the right-hand side, this implies
	\begin{equation} \label{proofev0}\begin{split}
	\Phi(u_n(t);&\Gamma_n(t),s_n(t),b_n(t)) + \gamma\int_0^1 |u_n(t)-w_n(t)|^2\de x \\
	&\leq \Phi(v;{\Gamma_n(t),s_n(t)},b_n(t)) + \gamma\int_0^1 |v-w_n(t)|^2\de x
	\qquad\text{for all $v\in\BV(0,1)$.}
	\end{split}\end{equation}
	
	\medskip\noindent\textit{Step 1: compactness.}
	The first goal is to prove a uniform bound on $\mathcal{E}_n(t)$, independent of $n$ and $t$, by using the energy inequality \eqref{energydisc2}.
	To this aim, notice that by using $v=b(0)$ as test function in \eqref{min0} we easily obtain a uniform bound on the initial energy $\mathcal{E}_n(0)$.
	Then \eqref{energydisc2} yields, as {$|h'(\xi)|\leq\ell$ for all $\xi$},
	\begin{equation*}
	\begin{split}
	\mathcal{E}_n(t)
	& \leq \mathcal{E}_n(0) + \int_{0}^{t_k^n}\int_0^1 \Bigl( h'((u_n(r))')(\dot{b}^n(r))' + 2\gamma (u_n(r)-w_n(r))(\dot{b}^n(r)-\dot{w}^n(r)) \Bigr)\de x\de r + R_n(t) \\
	& \leq C + \ell \int_{0}^{T} \int_0^1 |(\dot{b}^n(r))' | \de x \de r \\
	& \qquad + 2\gamma \sup_{r\in[0,T]} \Bigl( \|u_n(r)\|_\infty + \|w_n(r)\|_\infty \Bigr) \int_0^T \int_0^1 \Bigl( |\dot{b}^n(r)| + |\dot{w}^n(r)| \Bigr){\de x}\de r + R_n(t) .
	\end{split}
	\end{equation*}
	For the remainder $R_n(t)$, defined in \eqref{energydisc3}, we have:
	\begin{equation*}
	R_n(t) \leq C\tau_n\int_0^{T} \|\dot{b}^n(r)\|^2_{H^1(0,1)}\de r + C \sup_{r\in[0,T]}\|w_n(r)\|_\infty \int_0^{T} \Bigl(\|\dot{w}^n(r)\|_{L^1(0,1)} + \|\dot{b}^n(r)\|_{L^1(0,1)} \Bigr) \de r .
	\end{equation*}
	In view of the assumptions on $b$ and $w$ and of Proposition~\ref{prop:discretetime}\ref{item2discr}, all the previous quantities are uniformly bounded, therefore we obtain a uniform estimate on the energies $\mathcal{E}_n(t)$ and in turn, similarly to \eqref{boundbv}, on the $\BV$-norm of the functions $u_n(t)$:
	\begin{equation} \label{proofev1}
	\sup_{n,t}\mathcal{E}_n(t) <\infty, \qquad \sup_{n,t} |Du_n(t)|(0,1) < \infty.
	\end{equation}
	
	Fix now a countable dense set $D\subset[0,T]$, with $0\in D$. By a diagonal argument and by the uniform bounds \eqref{proofev1} we can find a subsequence (which we denote by the same symbol) such that
	\begin{equation} \label{proofev2}
	u_n(t) \wto u(t) \qquad\text{weakly* in }\BV(0,1), \text{ for all }t\in D,
	\end{equation}
	for some $u(t)\in\BV(0,1)$, with $\|u(t)\|_\infty \leq \max_{r\in[0,T]}\max\{ \|b(r)\|_\infty, \|w(r)\|_\infty \}$.
	{This implies in particular that $\sup_t \sup_n \sup_x |s_n(t)(x)|<\infty$.}
	
	Thanks to Proposition~\ref{prop:discretetime}\ref{item3discr}, it is clear that we have a uniform bound on the number of points of $\Gamma_n(t)$:
	\begin{equation} \label{proofev3}
	\sup_{n,t}\hz(\Gamma_n(t)) <\infty.
	\end{equation}
	Up to further subsequences, we can therefore assume that 
	{for any $t\in D$ and $n$ large enough}
	$\Gamma_n(t)=\{ \bar{x}_1^n(t),\ldots, \bar{x}_{N(t)}^n(t) \}$  (with the points $\bar{x}^i_n(t)$ distinct and $N(t)$ independent of $n$), and that each sequence $\bar{x}^n_{i}(t)$ converges as $n\to\infty$; we denote by $\Gamma(t)= \{ \bar{x}_1(t) ,\ldots, \bar{x}_{N(t)}(t) \}$ the set of limit points of these sequences,
	\begin{equation} \label{proofev4}
	\Gamma(t)  := \Bigl\{ \bar{x}_i(t) :=\lim_{n\to\infty} \bar{x}^n_i(t), \; i=1,\ldots,N(t) \Bigr\}\subset [0,1] \qquad\text{for }t\in D.
	\end{equation}
	Notice that all the points $\bar{x}_i(t)$ are distinct, since two different sequences $(\bar{x}^n_i(t))_n$, $(\bar{x}^n_j(t))_n$ cannot converge to the same limit point by Proposition~\ref{prop:discretetime}\ref{item3discr}.
	Moreover, in view of the monotonicity property $\Gamma_n(t_1)\subset\Gamma_n(t_2)$ for $t_1<t_2$, we also have
	\begin{equation} \label{proofev5}
	\Gamma(t_1) \subset \Gamma(t_2) \qquad \text{for all $t_1,t_2\in D$, $t_1<t_2$.}
	\end{equation}
	
	Finally, we also have compactness for the maps $s_n(t)$: by possibly extracting another subsequence, for every $t\in D$ there exists a map $s(t):\Gamma(t)\to[\bar{s},\infty)$, such that
	\begin{equation} \label{proofev6}
	s_n(t)(\bar{x}^n_i(t)) \to s(t)(\bar{x}_i(t)) \qquad\text{for all $i=1,\ldots,N(t)$, $t\in D$.}
	\end{equation}
	Moreover $s(t_1)\leq s(t_2)$ on $\Gamma(t_1)$ for all $t_1,t_2\in D$, $t_1<t_2$, since every map $s_n(t)$ is nondecreasing in time by construction.
	{As usual, we extend $s$ by 0 on $[0,1]\setminus \Gamma(t)$.}
	
	\medskip\noindent\textit{Step 2.} We now prove the following claims:
	\begin{equation} \label{proofev7}
	\bigl\{ x\in J^{b(t)}(u(t)) : |[u(t)](x)|>\bar{s} \bigr\} \subset \Gamma(t) \qquad \text{for all $t\in D$,}
	\end{equation}
	and
	\begin{equation} \label{proofev7b}
	|[u(t)](x)|\leq s(t,x) \qquad \text{for all $x\in\Gamma(t)$ and $t\in D$.}
	\end{equation}
	{We first recall that by Proposition~\ref{prop:nocantor} we have $u_n\in SBV(0,1)$ and $|u'_n|\le\frac\ell2$ almost everywhere.}

	In order to prove \eqref{proofev7}, suppose by contradiction that for some $t\in D$ there exists $\bar{x}\in J^{b(t)}(u(t))\setminus\Gamma(t)$ with $|[u(t)](\bar{x})|\geq\bar{s}+\e$, for some $\e>0$.
	We also assume that $\bar{x}\in(0,1)$, as the boundary case follows by a similar argument. Let $I_\delta:=(\bar{x}-\delta,\bar{x}+\delta)$ be such that $I_\delta\cap \Gamma_n(t)=\emptyset$ for all sufficiently large $n$; then, by definition of $\Gamma_n(t)$, all the jumps of $u_n(t)$ inside $I_\delta$ are smaller than the threshold $\bar{s}$: $|[u_n(t)](x)|\leq\bar{s}$ for every $x\in J_{u_n(t)}\cap I_\delta$. Moreover, 
	{as $\Gamma_n(t-\tau_n)\subset\Gamma_n(t)$,}
	$u_n(t)$ solves the minimum problem
	\begin{multline} \label{proofev9}
	\min\biggl\{ \int_{I_\delta} h(|v'|)\de x + \sum_{x\in J_v\cap I_\delta} { \Gdue(|[v](x)|,0)} + \ell|D^cv|(I_\delta) + \gamma\int_{I_\delta}|v-w_n(t)|^2\de x \,:\, v\in\BV(I_\delta),\,\\ \{v\neq u_n(t)\} \subset\subset I_\delta \biggr\} \,.
	\end{multline}
	By \eqref{proofev2} and lower semicontinuity of the total variation, we have for every $\delta>0$
	\begin{equation*}
	\liminf_{n\to\infty}|Du_n(t)|(I_\delta) \geq |Du(t)|(I_\delta) \geq \bar{s} + \e .
	\end{equation*}
	In view of  Proposition~\ref{prop:nocantor}, it is not possible that the limit jump of $u(t)$ at $\bar{x}$ is created by a nonzero contribution from the absolutely continuous part of the measures $Du_n(t)$: more precisely, we can assert that
	\begin{equation*}
	\text{for every $\delta>0$ sufficiently small} \qquad \liminf_{n\to\infty} \sum_{x\in J_{u_n(t)}\cap I_\delta} \big|[u_n(t)](x)\big| \geq \bar{s} + \frac{\e}{2},
	\end{equation*}
	and each $u_n(t)$ has at least two jumps in $I_\delta$ since $|[u_n(t)](x)|\leq\bar{s}$. In particular, for $n$ large we can find a partition $J_{u_n(t)} \cap I_\delta=A_n\cup B_n$ such that
	\begin{equation}  \label{proofev9b}
	\sum_{x\in A_n}\big| [u_n(t)](x)\big| \geq\frac{\e}{4}, \qquad \sum_{x\in B_n} \big| [u_n(t)](x) \big| \geq\frac{\e}{4}.
	\end{equation}
	We construct a competitor for the minimum problem \eqref{proofev9} by moving all the jumps of $u_n(t)$ in a single point $x_n\in I_\delta\setminus J_{u_n(t)}$: then the minimality of $u_n(t)$ in \eqref{proofev9} and the subadditivity of { $\Gdue$} yield
	{ \begin{equation*}
	\begin{split}
	0
	& \leq \Gdue\Biggl( \bigg| \sum_{x\in J_{u_n(t)}\cap I_\delta} [u_n(t)](x) \bigg|,0 \Biggr) - \sum_{x\in J_{u_n(t)}\cap I_\delta} \Gdue\bigl( |[u_n(t)](x)|,0 \bigr) + C\delta \\
	& \leq \Gdue\Biggl( \sum_{x\in J_{u_n(t)}\cap I_\delta} |[u_n(t)](x)| ,0\Biggr) - \Gdue\biggl( \sum_{x\in A_n} |[u_n(t)](x)|,0 \biggr) - \Gdue\biggl( \sum_{x\in B_n} |[u_n(t)](x)|,0 \biggr) + C\delta \\
	& \leq -c_{\e/4} + C\delta,
	\end{split}
	\end{equation*}}%
	the last inequality following by \eqref{proofev9b} { and \eqref{gsub1}}. This is a contradiction for $\delta$ small enough and completes the proof of \eqref{proofev7}.
	
	We next show \eqref{proofev7b}. Let $\bar{x}_i(t)$ be any point in $\Gamma(t)$. Recalling \eqref{proofev6} and that by construction $|[u_n(t)](\bar{x}^n_i(t))|\leq s_n(t,\bar{x}^n_i(t))$, it is sufficient to prove that
	\begin{equation} \label{proofev9c}
	|[u(t)](\bar{x}_i(t))| \leq \liminf_{n\to\infty} |[u_n(t)](\bar{x}^n_i(t))| \,.
	\end{equation}
	Let us momentarily omit the dependence on $t$. Suppose by contradiction that for some $\e>0$ one has $|[u](\bar{x}_i)| \geq \liminf_{n\to\infty} |[u_n](\bar{x}^n_i)| + \e$. Then, by {Proposition~\ref{prop:nocantor}}, \eqref{proofev2}, and lower semicontinuity of the total variation, we have for 
	{$I_\delta:=(\bar{x}_i-\delta,\bar{x}_i+\delta)$}
	and sufficiently small $\delta$
	\begin{equation*}
	\limsup_{n\to\infty} \sum_{x\in J_{u_n}\cap I_\delta\setminus\{\bar{x}^n_i\}}  \big|[u_n](x)\big| \geq\frac{\e}{2}.
	\end{equation*}
	We construct a competitor by moving all the jumps of $u_n$ at the point $\bar{x}^n_i$: the minimality of $u_n$ in \eqref{proofev0} yields for $n$ large enough
	\begin{equation*}
	\begin{split}
	0
	& \leq {\Gdue}\Biggl( \bigg| \sum_{x\in J_{u_n}\cap I_\delta} [u_n](x) \bigg| , s_n(\bar{x}^n_i) \Biggr) - {\Gdue}(|[u_n](\bar{x}^n_i)|, s_n(\bar{x}^n_i)) - \sum_{x\in J_{u_n}\cap I_\delta\setminus\{\bar{x}^n_i\}} { \Gdue(|[u_n](x)|,0)} + C\delta \\
	& \leq {\Gdue}\Biggl( \sum_{x\in J_{u_n}\cap I_\delta} \big| [u_n](x)\big| , s_n(\bar{x}^n_i) \Biggr) - {\Gdue}(|[u_n](\bar{x}^n_i)|, s_n(\bar{x}^n_i)) - { \Gdue\Biggl( \sum_{x\in J_{u_n}\cap I_\delta\setminus\{\bar{x}^n_i\}} \big| [u_n](x) \big|,0 \Biggr)} + C\delta \\
	& \leq -c_{\e/2} + C\delta,
	\end{split}
	\end{equation*}
	where the last inequality { follows by \eqref{gsub21}}. This is a contradiction for $\delta$ small enough, and proves that \eqref{proofev9c} holds.
	
	\medskip\noindent\textit{Step 3: lower semicontinuity of the energies.}
	We now claim that for every $t\in D$
	\begin{equation} \label{proofev8}
	\Phi(u(t);\Gamma(t),s(t),b(t)) \leq \liminf_{n\to\infty} \Phi(u_n(t);\Gamma_n(t),s_n(t),b_n(t)) \,.
	\end{equation}
	First observe that, in view of the continuity and monotonicity properties of ${\Gdue}$, \eqref{proofev9c}, and \eqref{proofev6}, we have
	\begin{equation} \label{proofev10}
	\sum_{x\in\Gamma(t)} {\Gdue}(|[u(t)](x)|,s(t,x)) \leq \liminf_{n\to\infty} \sum_{x\in\Gamma_n(t)} {\Gdue}\bigl( |[u_n(t)](x)|, s_n(t,x) \bigr) \,.
	\end{equation}
	Consider any relatively open set $A\subset\subset[0,1]\setminus\Gamma(t)$; by lower semicontinuity of the functional
	\begin{multline*}
	\int_{A} h(|u'(t)|)\de x + \sum_{x\in J^{b(t)}(u(t))\cap A} { \Gdue(|[u(t)](x)|,0)} + \ell|D^cu(t)|(A) \\
	\leq\liminf_{n\to\infty}\biggl[ \int_{A}  h(|u'_n(t)|)\de x + \sum_{x\in J^{b_n(t)}(u_n(t))\cap A} { \Gdue(|[u_n(t)](x)|,0)} + \ell|D^cu_n(t)|(A) \biggr]
	\end{multline*}
	so that taking the supremum over all $A$ and recalling \eqref{proofev10} we obtain \eqref{proofev8}.
	
	\medskip\noindent\textit{Step 4: static equilibrium.}
	We now show that the minimality condition \ref{item6ev} holds for every $t\in D$. In order to do this, 
	{we first}
	show that for every $v\in\BV(0,1)$ we can construct $v_n\in\BV(0,1)$ such that $v_n\to v$ in $L^2(0,1)$ as $n\to\infty$ and
	\begin{equation} \label{proofev11}
	\limsup_{n\to\infty} \Bigl[ \Phi(v_n;\Gamma_n(t), s_n(t), b_n(t)) - \Phi(v;\Gamma(t),s(t),b(t)) \Bigr] \leq 0.
	\end{equation}
	Recalling \eqref{proofev4}, this can be done by considering the (possible) jumps of $v$ on the points of $\Gamma(t)$ and moving them to the corresponding points of $\Gamma_n(t)$: more precisely, assuming for simplicity that $\Gamma(t)\subset(0,1)$ (the construction can be straightforwardly adapted if one of the boundary points belongs to $\Gamma(t)$), we define $v_n$ by the conditions
	\begin{equation*}
	v_n^+(0):=v^+(0), \qquad Dv_n := Dv - \sum_{i=1}^{N(t)} [v](\bar{x}_i(t))\delta_{\bar{x}_i(t)} + \sum_{i=1}^{N(t)} [v](\bar{x}_i(t))\delta_{\bar{x}_i^n(t)}.
	\end{equation*}
	Then $v_n\to v$ in $L^2(0,1)$.
	{By \eqref{proofev4}, if $n$ is sufficiently large then $\bar{x}_i^n(t)\ne \bar{x}_j(t)$ for $i\ne j$. We set 
		$\chi_i^n=1$ if $\bar{x}_i^n(t)\ne \bar{x}_i(t)$, and
		$\chi_i^n=0$ if $\bar{x}_i^n(t)=\bar{x}_i(t)$. 
		We estimate}
	\begin{equation*}
	\begin{split}
	&\Phi(v_n;\Gamma_n(t), s_n(t), b_n(t)) - \Phi(v;\Gamma(t),s(t),b(t)) \\
	& = \sum_{i=1}^{N(t)}{\chi_i^n} \Bigl[ {\Gdue}\bigl( |[v](\bar{x}^n_i(t))+[v](\bar{x}_i(t))|, s_n(t,\bar{x}^n_i(t))\bigr) - {\Gdue}\bigl(|[v](\bar{x}_i(t))|,s(t,\bar{x}_i(t))\bigr) - { \Gdue(|[v](\bar{x}_i^n(t))|,0)} \Bigr] \\
	& \: + \sum_{i=1}^{N(t)}{(1-\chi_i^n)} \Bigl[ {\Gdue}\bigl( |[v](\bar{x}_i(t))|, s_n(t,\bar{x}^n_i(t))\bigr) - {\Gdue}\bigl(|[v](\bar{x}_i(t))|,s(t,\bar{x}_i(t))\bigr)  \Bigr] \\
	& \: +  { \Gdue(|v^+(0)-b_n(t,0)|,0) - \Gdue(|v^+(0)-b(t,0)|,0) + \Gdue(|v^-(1)-b_n(t,1)|,0) - \Gdue(|v^-(1)-b(t,1)|,0)} \\
	& { \xupref{gbarsub1}{\leq}} \sum_{i=1}^{N(t)}{\chi_i^n}  \Bigl[ {\Gdue}\bigl( |[v](\bar{x}_i(t))|, s_n(t,\bar{x}^n_i(t))\bigr) - {\Gdue}\bigl(|[v](\bar{x}_i(t))|,s(t,\bar{x}_i(t))\bigr) \Bigr] \\
	& \: + \sum_{i=1}^{N(t)}{(1-\chi_i^n)} \Bigl[ {\Gdue}\bigl( |[v](\bar{x}_i(t))|, s_n(t,\bar{x}_i^n(t))\bigr) - {\Gdue}\bigl(|[v](\bar{x}_i(t))|,s(t,\bar{x}_i(t))\bigr)  \Bigr] \\
	& \: +  { \Gdue(|v^+(0)-b_n(t,0)|,0) - \Gdue(|v^+(0)-b(t,0)|,0) + \Gdue(|v^-(1)-b_n(t,1)|,0) - \Gdue(|v^-(1)-b(t,1)|,0)}.
	\end{split}
	\end{equation*}
	Hence \eqref{proofev11} follows by taking into account \eqref{proofev6}, {the continuity of $\Gdue$,} and that $b_n(t)\to b(t)$.
	
	We are now in position to conclude the proof of \ref{item6ev} for $t\in D$. Let $v\in\BV(0,1)$ and let $v_n$ be the sequence constructed before. Then, using the convergence of $u_n(t)\to u(t)$ and $w_n(t)\to w(t)$ in $L^2(0,1)$,
	\begin{equation*}
	\begin{split}
	\Phi(u(t);\Gamma(t),s(t),b(t)) &+ \gamma\int_0^1 |u(t)-w(t)|^2\de x \\
	& \xupref{proofev8}{\leq} \liminf_{n\to\infty} \biggl[ \Phi(u_n(t);\Gamma_n(t),s_n(t),b_n(t)) + \gamma\int_0^1 |u_n(t)-w_n(t)|^2\de x \biggr] \\
	& \xupref{proofev0}{\leq} \liminf_{n\to\infty} \biggl[ \Phi(v_n;\Gamma_n(t),s_n(t),b_n(t)) + \gamma\int_0^1 |v_n-w_n(t)|^2\de x\biggr] \\
	& \xupref{proofev11}{\leq}  \Phi(v;\Gamma(t),s(t),b(t)) + \gamma\int_0^1 |v-w(t)|^2\de x.
	\end{split}
	\end{equation*}
	Notice in particular that, by taking $v=u(t)$, the previous chain of inequalities yields the convergence of the energies:
	\begin{equation} \label{proofev13}
	\Phi(u(t);\Gamma(t),s(t),b(t)) = \lim_{n\to\infty} \Phi(u_n(t);\Gamma_n(t),s_n(t),b_n(t)) .
	\end{equation}
	
	\medskip\noindent\textit{Step 5: definition of the evolution for $t\notin D$.}
	By the monotonicity property \eqref{proofev5}, the uniform bound on the number $N(t)$ of points in $\Gamma(t)$, and the monotonicity of $t\mapsto s(t)$, there exists a set $D'\subset[0,T]\setminus D$, at most countable, such that
	\begin{equation} \label{proofev50}
	\bigcap_{t'\in D, t'\geq t} \Gamma(t') = \bigcup_{t'\in D, t'\leq t}\Gamma(t'),
	\qquad
	\inf_{t'\in D,t'\geq t} s(t') = \sup_{t'\in D,t'\leq t}s(t')
	\qquad\text{for all $t\in[0,T]\setminus D'$.}
	\end{equation}
	We then define $\Gamma(t)$ and $s(t)$ to be equal to the common values in \eqref{proofev50} for all $t\in[0,T]\setminus D'$. Moreover, by repeating the construction in Steps~1--4 for the points $t\in D'$, and up to a further subsequence, we obtain a triple $(u(t),\Gamma(t),s(t))$ for all $t\in D\cup D'$ such that the conclusions of the previous steps hold for every $t\in D\cup D'$.
	
	It remains to define $u(t)$ for $t\in[0,T]\setminus(D\cup D')$.
	{For $t\in[0,T]$,} we introduce the quantity
	\begin{equation} \label{proofev51}
	\theta_n(t) := \int_0^1 \Bigl( h'((u_n(t))')(\dot{b}(t))' + 2\gamma (u_n(t)-w_n(t))(\dot{b}(t)-\dot{w}(t)) \Bigr)\de x,
	\end{equation}
	as well as
	\begin{equation} \label{proofev52}
	\theta_\infty(t) := \limsup_{n\to\infty} \theta_n(t).
	\end{equation}
	Notice that $|\theta_n(t)|\leq C(\|\dot{b}(t)\|_{W^{1,1}(0,1)} + \|\dot{w}(t)\|_{L^1(0,1)})\in L^1(0,T)$, therefore $\theta_\infty\in L^1(0,T)$ and by Fatou's Lemma
	\begin{equation} \label{proofev53}
	\limsup_{n\to\infty}\int_0^t \theta_n(r)\de r \leq \int_0^t \theta_\infty(r)\de r 
	\hskip1cm{\text{for all $t\in[0,T]$}}.
	\end{equation}
	For a given $t\in[0,T]\setminus(D\cup D')$, we can find a subsequence $(n_j)_j$, dependent on $t$, such that $\theta_{n_j}(t)\to\theta_\infty(t)$,
	and $u_{n_j}(t)$ converges weakly* in $\BV(0,1)$ to a function $u(t)$ as $j\to\infty$. We choose this limit function to define the triple $(u(t),\Gamma(t),s(t))$ for $t\in[0,T]\setminus(D\cup D')$. Notice that the arguments in Steps~2--4 can be repeated at the point $t$ for this subsequence: therefore the evolution $t\mapsto(u(t),\Gamma(t),s(t))$ satisfies the properties \ref{item1ev}--\ref{item6ev} in the statement for all $t\in[0,T]$.
	
	\medskip\noindent\textit{Step 6: non-dissipativity.}
	To complete the proof, it only remains to show the condition \ref{item7ev}. Setting
	\begin{equation} \label{proofev60}
	\theta(t) := \int_0^1 \Bigl( h'((u(t))')(\dot{b}(t))' + 2\gamma (u(t)-w(t))(\dot{b}(t)-\dot{w}(t)) \Bigr)\de x,
	\end{equation}
	we first claim that
	\begin{equation} \label{proofev61}
	\theta(t)=\theta_\infty(t) \qquad\text{for almost every $t\in[0,T]$.} 
	\end{equation}
	In particular, this will give the measurability and integrability of $\theta(t)$ in $[0,T]$. The claim \eqref{proofev61} can be proved by an argument similar to \cite[Lemma~4.11]{DMFraToa}: fix $t\in[0,T]$ and consider any sequence of positive numbers ${\delta_i}\to0$ and any $d\in\R$. 
	{By the definition of $u(t)$ there is a sequence $n_j\to\infty$
		such that} $u_{n_j}(t)(x)+d{\delta_i} x\wto u(t)(x)+d {\delta_i}x$ weakly* in $\BV(0,1)$; by arguing as in Step~3 to prove the lower semicontinuity of the energy along the sequence $u_{n_j}(t)(x)+d{\delta_i} x$, and recalling \eqref{proofev13}, we find
	\begin{multline*}
	\frac{1}{{\delta_i}} \Bigl( \Phi(u(t)+d{\delta_i} x;\Gamma(t),s(t),b(t)+d{\delta_i}x) - \Phi(u(t);\Gamma(t),s(t),b(t)) \Bigr) \\
	\leq \liminf_{j\to\infty} \frac{1}{{\delta_i}} \Bigl( \Phi(u_{n_j}(t)+d{\delta_i} x;\Gamma_{n_j}(t),s_{n_j}(t),b_{n_j}(t)+d{\delta_i}x) - \Phi(u_{n_j}(t);\Gamma_{n_j}(t),s_{n_j}(t),b_{n_j}(t)) \Bigr) .
	\end{multline*}
	By writing the explicit expressions of the previous quantities, we find
	\begin{equation*}
	\frac{1}{{\delta_i}} \int_0^1 \Bigl( h(|u'(t)+{\delta_i}d|)-h(|u'(t)|) \Bigr)\de x
	\leq \liminf_{j\to\infty} \frac{1}{{\delta_i}}\int_0^1 \Bigl(  h(|u_{n_j}'(t)+{\delta_i}d|)-h(|u_{n_j}'(t)|) \Bigr)\de x ,
	\end{equation*}
	{therefore there exists an increasing sequence of integers $j_i\ge i$ such that 
		\begin{equation*}
		\frac{1}{{\delta_i}} \int_0^1 \Bigl(  h(u'(t)+{\delta_i}d)-h(u'(t)) \Bigr)\de x \leq \frac{1}{{\delta_i}}\int_0^1 \Bigl(  
		h(u_{{n_{j_i}}}'(t)+{\delta_i}d)-h(u_{{n_{j_i}}}'(t)) \Bigr)\de x + \frac{1}{i} \,.
		\end{equation*}}
	{Taking the limit $i\to\infty$ we obtain proceeding as in \cite[Lemma~4.11]{DMFraToa},}
	\begin{equation*}
	\begin{split}
	d \int_0^1 h'(u'(t))\de x
	&= \liminf_{{i\to\infty}} \int_0^1 
	\frac{  h(u'(t)+{{\delta_i}} d)-h(u'(t))}{{{\delta_i}}}
	\de x \\
	& \leq \liminf_{{i\to\infty}} \int_0^1
	\frac{ h(u_{{n_{j_i}}}'(t)+{{\delta_i}} d)-
		h(u_{{n_{j_i}}}'(t))}{{{\delta_i}}}
	\de x \\
	& = \liminf_{{i\to\infty}} d\int_0^1 h'(u_{{n_{j_i}}}'(t)+{\tau_i}d)\de x = \liminf_{{i\to\infty}} d \int_0^1 h'(u_{{n_{j_i}}}'(t))\de x
	\end{split}
	\end{equation*}
	for suitable ${\tau_i}:(0,1)\to[0,{{\delta_i}}]$.
	{Taking $d=1$ and $d=-1$ we see that this is actually an equality and that the limit does not depend on the subsequence $j_i$, which implies}
	\begin{equation*}
	\lim_{j\to\infty}\int_0^1 h'((u_{n_j}(t))') \de x =\int_0^1 h'((u(t))') \de x.
	\end{equation*}
	By the strong convergence of $u_{n_j}(t)\to u(t)$ and $w_{n_j}(t)\to w(t)$ in $L^2(0,1)$, we then conclude that $\theta_{n_j}(t)\to\theta(t)$ as $j\to\infty$ and, since $\theta_\infty$ was the limit of the subsequence $\theta_{n_j}$, this shows \eqref{proofev61}.
	
	By Lemma~\ref{lem:energy} we have
	\begin{multline}
	\mathcal{E}_{n_j}(t)
	\leq \mathcal{E}(0) + \int_{0}^{t^{n_j}_k} \theta_{n_j}(r)\de r + R_{n_j}(t) + \int_0^{t^{n_j}_k} \int_0^1 h'((u_{n_j}(r))')(\dot{b}^{n_j}(r)-\dot{b}(r))' \de x \de r \\
	+ 2\gamma \int_0^{t^{n_j}_k} \int_0^1 (u_{n_j}(r)-w_{n_j}(r))(\dot{b}^{n_j}(r)-\dot{b}(r)-\dot{w}^{n_j}(r)+\dot{w}(r)) \de x\de r,\label{eq:Enjleq}
	\end{multline}
	where $t^{n_j}_k$ is the discretization point such that $t\in[t^{n_j}_k, t^{n_j}_{k+1})$.
	One can now check that the last two terms in the previous expression are actually equal to zero, and that $R_{n_j}(t)\to0$ as $j\to\infty$, thanks to the assumptions on $b$ and $w$. Therefore, recalling \eqref{proofev53}, \eqref{proofev61}, and that $\mathcal{E}_{n_j}(t)\to\mathcal{E}(t)$ by \eqref{proofev13},  we conclude that
	\begin{equation}\label{eq:Eleq}
	\mathcal{E}(t) \leq \mathcal{E}(0) + \int_0^t \theta(r)\de r.
	\end{equation}
	
	To conclude the proof it remains to show the opposite inequality. Let us stress that the subsequence $n_j$ depends on $t$, so that we do not have a unique subsequence converging pointwise almost everywhere in $[0,T]$ to $\theta(t)$. This prevents to take directly the lower limit in the opposite of inequality \eqref{eq:Enjleq}, to get the opposite of inequality \eqref{eq:Eleq}. In order to overcome this difficulty, we will first approximate the Lebesgue integral of $\theta$ by Riemann sums.
	
	{We fix $t\in(0,T]$. We first observe that there exists a sequence of subdivisions of $[0,t]$ of the form
		\begin{equation*}
		0=\rho^m_0<\rho^m_1<\ldots<\rho^m_{i_m-1}<\rho^m_{i_m}=t,
		\qquad\text{with}\quad
		\lim_{m\to\infty} \max_{i=1,\ldots,i_m}|\rho^m_{i}-\rho^m_{i-1}|=0,
		\end{equation*}
		with the property that
		\begin{equation} \label{proofev62a}
		\lim_{m\to\infty} \sum_{i=1}^{i_m}\bigg| (\rho^m_{i}-\rho^{m}_{i-1})\theta(\rho^m_i) - \int_{\rho^m_{i-1}}^{\rho^m_i} \theta(r)\de r \bigg| = 0,
		\end{equation}
		\begin{equation} \label{proofev62b}
		\lim_{m\to\infty} \sum_{i=1}^{i_m} \int_0^1 \bigg| (\rho^m_i-\rho^m_{i-1})\dot{b}(\rho^m_i) - \int_{\rho^m_{i-1}}^{\rho^m_i} \dot{b}(r) \de r \bigg| \de x = 0,
		\end{equation}
		\begin{equation} \label{proofev62c}
		\lim_{m\to\infty} \sum_{i=1}^{i_m} \int_0^1 \bigg| (\rho^m_i-\rho^m_{i-1})\dot{w}(\rho^m_i) - \int_{\rho^m_{i-1}}^{\rho^m_i} \dot{w}(r) \de r \bigg| \de x = 0
		\end{equation}
		(see \cite[Lemma~4.12]{DMFraToa}, and also \cite[Lemma~4.12]{DMFraToa2} for a proof, adapting the arguments of \cite[page~63]{Doob}). We now exploit the global stability of $u(\rho^m_{i-1})$ (property \ref{item6ev}), taking the competitor $u(\rho^m_i)+b(\rho^m_{i-1})-b(\rho^m_i)$ and adopting an iteration argument similar to that in the proof of Lemma~\ref{lem:energy}. We repeatedly use that $b\in H^1([0,T],\R^2)$, with $b(t)=(b(t,0),b(t,1))$, and that $b(t):[0,1]\to\R$ denotes the affine function interpolating between the boundary values $b(t,0)$ and $b(t,1)$. Then, using also the monotonicity property \ref{item2ev}, we find
		\begin{multline} \label{proofev63}
		\mathcal{E}(t)
		\geq \mathcal{E}(0) + \sum_{i=1}^{i_m} \int_{\rho^m_{i-1}}^{\rho^m_i} \int_0^1 h'((u(\rho^m_i))')(\dot{b}(r))'\de x \de r \\
		+ 2\gamma\sum_{i=1}^{i_m} \int_{\rho^m_{i-1}}^{\rho^m_i} \int_0^1 \bigl( w(\rho^m_i)-u(\rho^m_i) \bigr) \bigl( \dot{w}(r)-\dot{b}(r) \bigr) \de x \de r - \mathcal{S}_m(t),
		\end{multline}
		where
		\begin{multline*}
		\mathcal{S}_m(t) := \gamma\sum_{i=1}^{i_m}\int_0^1 \bigl( w(\rho^m_i)-w(\rho^m_{i-1}) \bigr)^2\de x - 2\gamma\sum_{i=1}^{i_m}\int_0^1 \bigl( w(\rho^m_i)-w(\rho^m_{i-1}) \bigr)\bigl( b(\rho^m_i)-b(\rho^m_{i-1}) \bigr) \\
		+ \sum_{i=1}^{i_m} \int_0^1 |b'(\rho^m_i)-b'(\rho^m_{i-1})|^2\de x +\gamma\sum_{i=1}^{i_m}\int_0^1 |b(\rho^m_i)-b(\rho^m_{i-1})|^2\de x.
		\end{multline*}
		Notice that, in view of the assumptions on $w$ and $b$, we have $\mathcal{S}_m(t)\to0$ as $m\to\infty$. Recalling the definition \eqref{proofev60} of the map $\theta(t)$, we further obtain from \eqref{proofev63}
		\begin{equation} \label{proofev64}
		\mathcal{E}(t)\geq\mathcal{E}(0) + \sum_{i=1}^{i_m} (\rho^m_i-\rho^m_{i-1})\theta(\rho^m_i) - \mathcal{R}_m(t) - \mathcal{S}_m(t),
		\end{equation}
		with the position
		\begin{multline*}
		\mathcal{R}_m(t) := \sum_{i=1}^{i_m} \int_{\rho^m_{i-1}}^{\rho^m_i}\int_0^1 h'((u(\rho^m_i))')\bigl( \dot{b}(\rho^m_i)-\dot{b}(r)\bigr)'\de x \de r \\
		+ 2 \gamma \sum_{i=1}^{i_m}\int_{\rho^m_{i-1}}^{\rho^m_i}\int_0^1 \bigl( u(\rho^m_i)-w(\rho^m_i)\bigr) \bigl( \dot{b}(\rho^m_i)-\dot{b}(r)-\dot{w}(\rho^m_i)+\dot{w}(r) \bigr)\de x \de r.
		\end{multline*}
		Using the definition of $h$ and the uniform bound in $L^\infty$ on $u$ and $w$ we can estimate
		\begin{align*}
		|\mathcal{R}_m(t)|
		& \leq \ell \sum_{i=1}^{i_m} \bigg| (\rho^m_i-\rho^m_{i-1})(\dot{b}(\rho^m_i))' - \int_{\rho^m_{i-1}}^{\rho^m_i} (\dot{b}(r))'\de r \bigg| \\
		& \qquad + C \sum_{i=1}^{i_m} \int_0^1 \bigg| (\rho^m_i-\rho^m_{i-1}) \bigl( \dot{b}(\rho^m_i)-\dot{w}(\rho^m_i) \bigr) - \int_{\rho^m_{i-1}}^{\rho^m_i} \bigl( \dot{b}(r) - \dot{w}(r) \bigr) \de r \bigg|\de x\,,
		\end{align*}
		and the previous quantity vanishes in the limit as $m\to\infty$ in view of \eqref{proofev62b}--\eqref{proofev62c}.
		Eventually, by passing to the limit as $m\to\infty$ in \eqref{proofev64}, recalling \eqref{proofev62a} and that $\mathcal{R}_m(t)\to0$, $\mathcal{S}_m(t)\to0$, we conclude that
		\begin{equation*}
		\mathcal{E}(t) \geq \mathcal{E}(0) + \int_0^t \theta(r)\de r,
		\end{equation*}
		which completes the proof of \ref{item7ev}.}
\end{proof}

\begin{remark}
	Notice that in the proof of Theorem~\ref{thm:evolution} we do \emph{not} obtain a unique subsequence $(n_j)_j$, independent of $t$, such that the time-discrete evolution $(u_{n_j}(t),\Gamma_{n_j}(t),s_{n_j}(t))$ converges to the limit evolution $(u(t),\Gamma(t),s(t))$ for all $t\in[0,T]$. This in general holds only for $t$ in a countable dense set. It would be possible to prove convergence of a subsequence for every $t\in[0,T]$ if we knew that the minimum problem solved by $u(t)$ has a unique solution.
\end{remark}


\section{Relaxation of the cohesive energy} \label{sect:relaxation}

In order to construct a time-discrete evolution in the cohesive setting including our notion of irreversibility, we made use in Sections~\ref{sect:setting} 
{and \ref{sect:evol}}
of a relaxation result in the spirit of \cite{BB,BBB}. The main difference is that the surface part of our energy contains also information on points representing a preexisting crack, and therefore is not considered in the existing literature; however, the proof is a small variant of the standard theory, and we present it in this section. A similar problem was discussed by Giacomini in \cite{Gia05b}.

Let $b:\{0,1\}\to\R$ be a given boundary datum. Let also $(\Gamma,s)$ be a given pair with $\Gamma\subset[0,1]$ countable and $s:\Gamma\to(0,\infty)$. Let $\Phi_0:\BV(0,1)\to[0,\infty]$ be the functional defined by
\begin{equation}\label{defPhi0bis}
\Phi_0(u):=
\begin{cases}
\displaystyle \int_0^1 |u'|^2\de x + { \sum_{x\in J_u^b\cup\Gamma} {\Gdue}(|[u](x)|,s(x))}  & \text{if }u\in\SBV(0,1), \\
\infty & \text{otherwise,}
\end{cases}
\end{equation}
where $J_u^b$ and the jump $[u](x)$ at the boundary points $x\in\{0,1\}$ are defined in \eqref{Dirichlet} and \eqref{Dirichlet2} respectively.

\begin{theorem} \label{thm:relaxation}
{ Assume that the surface energy density ${\Gdue}$ satisfies assumptions \ref{item1gbar1}-\ref{item2gbar1} of Section~\ref{sect:setting}.} The relaxation of the functional $\Phi_0$ with respect to the weak*-topology of $\BV(0,1)$ is given by $\Phi:\BV(0,1)\to[0,\infty]$,
	\begin{equation} \label{defPhibis}
	\Phi(u) := \int_0^1 h(|u'|)\de x + { \sum_{x\in J_u^b\cup \Gamma} {\Gdue}(|[u](x)|,s(x))} +  \ell|D^cu|(0,1)\,,
	\end{equation}
where the elastic energy density $h$ is defined in \eqref{defh1}.
\end{theorem}

\begin{proof}
	In order to handle the boundary conditions, we work in a larger open interval $I$ containing $[0,1]$, for instance $I:=(-1,2)$, and we extend the boundary values as $b(x)=b(0)$ for $x\leq0$, $b(x)=b(1)$ for $x\geq1$. We consider the functional
	\begin{equation*} \label{proofrelax0}
	\Psi_0(u):=
	\begin{cases}
	\displaystyle \int_0^1 |u'|^2\de x + { \sum_{x\in J_u^b\cup \Gamma} {\Gdue}(|[u](x)|,s(x))} &
	\begin{array}{l}
	\text{if } u\in\SBV(I), \\ 
	u=b \text{ on }I\setminus(0,1),
	\end{array}\\
	\infty & \text{ otherwise in }\BV(I).
	\end{cases}
	\end{equation*}
	The statement is equivalent to proving that the relaxation of the functional $\Psi_0$ with respect to the weak*-topology of $\BV(I)$ is given by 
	\begin{equation} \label{proofrelax1}
	\Psi(u) := \int_0^1 h(|u'|)\de x + {\sum_{x\in J_u^b\cup \Gamma} {\Gdue}(|[u](x)|,s(x))} + \ell|D^cu|(0,1)
	\end{equation}
	if $u\in\BV(I)$ with $u=b$ on $I\setminus(0,1)$, $\Psi(u)=\infty$ otherwise in $\BV(I)$.
	
	\medskip\noindent\textit{Step 1: lower semicontinuity of $\Psi$.}
	Let $u_n\in\BV(I)$ be a sequence converging weakly* in $\BV$ to some function $u$; in particular, $u_n\to u$ in $L^1(I)$ and $Du_n\stackrel{*}{\wto} Du$ in the sense of measures. We can also assume without loss of generality that $u_n=b$ on $I\setminus(0,1)$ for every $n$, and that $\Psi(u_n)$ has a limit as $n\to\infty$. Since $\Psi$ can be obtained as the supremum of functionals corresponding to a \emph{finite} set $\Gamma$, we may also assume that $\Gamma=\{x_1,\ldots,x_m\}$, with $s_i=s(x_i)$.
	
	We now isolate the possible jumps of $u_n$ at the points $x_i\in\Gamma$: namely, we consider the atomic measures
	\begin{equation*}
	\mu_n := \sum_{i=1}^m [u_n](x_i) \delta_{x_i}\,,
	\end{equation*}
	and we can assume that $[u_n](x_i)\to a_i$ for suitable values $a_i\in\R$, or equivalently
	\begin{equation} \label{proofrelax2}
	\mu_n\stackrel{*}{\wto} \mu := \sum_{i=1}^m a_i\delta_{x_i} \,.
	\end{equation}
	We then consider the function $v_n\in\BV(I)$ defined by $v_n(x):=b(0)+(Du_n-\mu_n)((-1,x))$. Notice in particular that $v_n$ is continuous at the points $x_i\in\Gamma$, $i=1,\ldots,m$. Moreover $v_n\to v$ weakly* in $\BV$, with $v\in\BV(I)$ satisfying $Du=Dv+\mu$. With this decomposition we have
	\begin{align*}
	\liminf_{n\to\infty} \Psi(u_n)
	& \geq \liminf_{n\to\infty} \biggl[ \int_0^1 h(|v_n'|)\de x + \sum_{x\in J_{v_n}} { \Gdue(|[v_n](x)|,0)} + \ell|D^c v_n|(0,1)\biggr] \\
	& \qquad + \liminf_{n\to\infty} \sum_{i=1}^m {\Gdue}(|[u_n](x_i)|,s_i) \,.
	\end{align*}
	By standard results the first term on the right-hand side is lower semicontinuous with respect to weak*-convergence in $\BV$, see for instance \cite[Theorem~5.2]{AFP}; therefore, using also \eqref{proofrelax2} and the continuity of ${\Gdue}$,
	\begin{align*}
	\liminf_{n\to\infty}\Psi(u_n)
	&\geq \biggl[ \int_0^1 h(|v'|)\de x + \sum_{x\in J_{v}} { \Gdue(|[v](x)|,0)} + \ell|D^c v|(0,1)\biggr] + \sum_{i=1}^m {\Gdue}(|a_i|,s_i) \\
	& = \int_0^1 h(|u'|)\de x + \sum_{x\in J_{u}\setminus\Gamma} { \Gdue(|[u](x)|,0)} + \ell|D^c u|(0,1) \\
	&\qquad + \sum_{i=1}^m \Bigl( { \Gdue(|[v](x_i)|,0)} + {\Gdue}(|a_i|,s_i) \Bigr)\,.
	\end{align*}
	Finally, as $[u](x_i)=[v](x_i)+a_i$, using the subadditivity property of ${\Gdue}$, we can bound from below the last sum by $\sum_{i=1}^m {\Gdue}(|[u](x_i)|,s_i)$, and in turn we recover the desired inequality $\liminf_{n\to\infty}\Psi(u_n)\geq\Psi(u)$.
	
	\medskip\noindent\textit{Step 2: relaxation.}
	For all open sets $A\subset I$ we consider the localized versions $\Psi_0(\cdot;A)$ and $\Psi(\cdot;A)$ of $\Psi_0$ and $\Psi$ respectively. We denote by $\overline{\Psi}_0(\cdot;A)$ the relaxation of $\Psi_0(\cdot;A)$ with respect to the weak*-topology of $\BV$. Notice that this coincides with the lower semicontinuous envelope of $\Psi_0(\cdot;A)$ with respect to the strong topology of $L^1$, that is
	\begin{equation*}
	\overline{\Psi}_0(u;A) = \inf \Bigl\{ \liminf_{n\to\infty}\Psi_0(u_n;A) \,:\, u_n\to u \text{ in }L^1(A) \Bigr\}
	\end{equation*}
	(see Remark~\ref{rmk:relaxation2} below). The thesis amounts to show that $\overline{\Psi}_0=\Psi$.

	{We further remark that since $\Gdue$ is nondecreasing in the first argument, $\sum_{x\in\Gamma} \Gdue(0,s(x))\le \Psi(u)\le \Psi_0(u)$ for all $u$. In particular, we can assume that  $\sum_{x\in\Gamma} \Gdue(0,s(x))<\infty$ for the rest of the proof.}

	By the lower semicontinuity of $\Psi$ proved in the previous step, the inequality $\Psi\leq\overline{\Psi}_0$ is immediate. We therefore show the opposite inequality. By the same argument as in \cite[Proposition~3.3]{BBB} we have that for every fixed $u\in\BV(I)$ the set function $\overline{\Psi}_0(u;\cdot)$ is the restriction to the family of open subsets of $I$ of a regular Borel measure.
	
	Notice that for every $u\in\SBV(I)$ with $u=b$ on $I\setminus(0,1)$ and every open set $A\subset I$ we have, using the subadditivity of $\Gdue$,
	\begin{equation*}
	\Psi_0(u;A) \leq \int_A |u'|^2\de x + \sum_{x\in J_u\cap A} { \Gdue(|[u](x)|,0)} + \sum_{x\in\Gamma\cap A} {\Gdue}(0,s(x))\,.
	\end{equation*}
	As a consequence of \cite[Theorem~3.1]{BBB} we obtain for every $u\in\BV(I)$ with $u=b$ on $I\setminus(0,1)$ and every open set $A\subset I$
	\begin{equation*}
	\overline{\Psi}_0(u;A) \leq \int_A h(|u'|)\de x + \ell |D^cu|(A) + \sum_{x\in J_u\cap A} { \Gdue(|[u](x)|,0)} + \sum_{x\in\Gamma\cap A} {\Gdue}(0,s(x)) \,.
	\end{equation*}
	In particular, from the previous inequality it follows that
	\begin{equation} \label{proofrelax3}
	\overline{\Psi}_0(u;\cdot)\mres (I\setminus\Gamma) \leq h(|u'|)\Lu + \ell|D^cu| + \sum_{x\in J_u\setminus\Gamma}{ \Gdue(|[u](x)|,0)}\delta_x \,.
	\end{equation}
	
	It remains to evaluate $\overline{\Psi}_0(u;\cdot)\mres\Gamma$. Without loss of generality we may assume that $u\in\BV(I)$ with $u=b$ on $I\setminus(0,1)$. Let $K\subset\Gamma$ be any finite subset of $\Gamma$. For $\e>0$ we take an open set $A_\e\subset I$ with $K\subset A_\e$ and
	\begin{equation*}
	|Du|(A_\e\setminus K) <\e\,, \qquad \sum_{x\in A_\e\cap\Gamma\setminus K} {\Gdue}(0,s(x)) <\e \,.
	\end{equation*}
	We can construct a sequence of functions $u_h\in\SBV(I)$, with $u_h=b$ on $I\setminus(0,1)$, such that $u_h$ is piecewise constant, $|Du_h|(A_\e\setminus K)\leq |Du|(A_\e\setminus K)<\e$, and $u_h\to u$ strongly in $L^\infty(I)$. Then
	\begin{align*}
	\overline{\Psi}_0(u;A_\e)
	& \leq \liminf_{h\to\infty} \Psi_0(u_h;A_\e)
	= \liminf_{h\to\infty} { \sum_{x\in A_\e\cap(J_{u_h}\cup\Gamma)} {\Gdue}(|[u_h](x)|,s(x))} \\
	& \leq \liminf_{h\to\infty} \biggl( \sum_{x\in K} {\Gdue}(|[u_h](x)|,s(x)) + \sum_{x\in A_\e\cap\Gamma\setminus K} {\Gdue}(0,s(x)) + \sum_{x\in (J_{u_h}\setminus K)\cap A_\e} { \Gdue(|[u_h](x)|,0)} \biggr) \\
	& \leq \sum_{x\in K} {\Gdue}(|[u](x)|,s(x)) + \sum_{x\in A_\e\cap\Gamma\setminus K} {\Gdue}(0,s(x)) + \ell \limsup_{h\to0} |Du_h|(A_\e\setminus K) \\
	& \leq  \sum_{x\in K} {\Gdue}(|[u](x)|,s(x)) + (1+\ell)\e \,,
	\end{align*}
	where we used the subadditivity of $\Gdue$ in the second line, and the inequality { $\Gdue(s,0)\leq\ell s$} in the third line.
	By letting $\e\to0$ we obtain
	\begin{align*}
	\overline{\Psi}_0(u;K) \leq \sum_{x\in K} {\Gdue}(|[u](x)|,s(x)) \,,
	\end{align*}
	and since $K$ is an arbitrary finite subset of $\Gamma$ we conclude that 
	\begin{equation} \label{proofrelax4}
	\overline{\Psi}_0(u;\cdot)\mres\Gamma \leq \sum_{x\in\Gamma} {\Gdue}(|[u](x)|,s(x)) \delta_x \,.
	\end{equation}
	The combination of \eqref{proofrelax3} and \eqref{proofrelax4} yields the inequality $\overline{\Psi}_0\leq\Psi$ and concludes the proof of the theorem.
\end{proof}

\begin{remark} \label{rmk:relaxation}
	Observe that the functional $\Phi$ is also the lower semicontinuous envelope of
	\begin{equation*}
	\tilde{\Phi}_0(u):=
	\begin{cases}
	\displaystyle \int_0^1 |u'|^2\de x + { \sum_{x\in J_u^b\cup\Gamma} {\Gdue}(|[u](x)|,s(x))} &
	\begin{array}{l}
	\text{if } u\in\SBV(0,1) \\ 
	\text{with }\mathcal{H}^0(J_u)<\infty,
	\end{array} \\
	\infty & \text{ otherwise,}
	\end{cases}
	\end{equation*}
	which is finite only on functions with a finite number of jumps. Indeed, the functions of the recovery sequence constructed in the second step of the proof of Theorem~\ref{thm:relaxation} are piecewise constant and therefore they satisfy the additional constraint.
\end{remark}

\begin{remark} \label{rmk:relaxation2}
	Notice that the relaxation $\overline{\Phi}$ of $\Phi_0$ with respect to the weak*-topology of $\BV$ coincides with the relaxation of $\Phi_0$ in the $L^1$-topology, which is given by
	\begin{equation} \label{defPhi1}
	\Phi^1(u) := \inf \Bigl\{ \liminf_{n\to\infty}\Phi_0(u_n) \,:\, u_n\to u \text{ in }L^1(0,1) \Bigr\}\,.
	\end{equation}
	Indeed, the inequality $\Phi^1\leq\overline{\Phi}$ is obvious. For the converse, first notice that the same estimate as in \eqref{boundbv} gives that for every sequence $(u_n)_n$ with $\sup_n\|u_n\|_\infty<\infty$ and $\sup_n\Phi_0(u_n)<\infty$ one also has $\sup_n|Du_n|(0,1)<\infty$. Then, given any sequence $u_n\to u$ in $L^1(0,1)$ with $\sup_n\Phi_0(u_n)<\infty$, consider the truncation $u_n^M:= (-M)\vee(u_n\wedge M)$, for $M>0$ sufficiently large. By the previous observation we have $u_n^M\to u^M$ in the weak*-topology of $\BV$, and therefore
	\begin{equation*}
	\overline{\Phi}(u^M) \leq \liminf_{n\to\infty} \Phi_0(u_n^M) \leq \liminf_{n\to\infty}\Phi_0(u_n)\,.
	\end{equation*}
	By passing to the limit as $M\to\infty$, the inequality $\overline{\Phi}\leq\Phi^1$ follows.
\end{remark}



\section{Static phase-field approximation: the cohesive energy of pristine material} \label{sect:gnew}

{ The goal of this and the next section is to derive a cohesive energy density ${\Gdue}(s,s')$ from a (static) phase-field approximation, and to show that the function obtained in this way satisfies the assumptions in Section~\ref{sect:setting}. In this section we focus on the cohesive energy $\Gzero$ of ``pristine'' materials, defined in $\eqref{defgnew}$: this corresponds to the function ${\Gdue}(s,0)$ in the notation of Section~\ref{sect:setting}, i.e. it is the energy at points which were not previously fractured. The definition of $\Gzero$ is based on the phase-field approximation proved in \cite{CFI}, which we state in Section~\ref{subsect:blowup}.}

\subsection{A class of cohesive energies \texorpdfstring{$\boldsymbol{\Gzero}$}{} for pristine {material}}\label{subsect:hyp}
Let $f$ be a function with the following properties:
\begin{equation} \label{assf1}
f\in C^1([0,1);[0,\infty)) \text{ is nondecreasing, }f^{-1}(0)=\{0\},
\end{equation}
\begin{equation} \label{assf2}
\tif(s):=sf(1-s) \text{ is {strictly decreasing}, and }f_1(\sqrt{\cdot}) \text{ is convex,}
\end{equation}
\begin{equation} \label{assf3}
\tif(s) = \ell - \ell_1s + o(s) \qquad\text{as }s\to0^+,
\end{equation}
for some $\ell,\ell_1>0$, where $\lim_{s\to0^+}o(s)/s=0$.
The assumptions \eqref{assf2}--\eqref{assf3} are of technical nature and are slightly stronger than the corresponding assumption in \cite{CFI}: we need to include them in order to guarantee further properties of the surface energy density $\Gzero$ defined below. For instance, prototype pairs $(f,f_1)$ with these properties are 
\begin{equation}\label{eqexamplefa}
f^a(s):=\frac{\ell s}{1-s}
\quad\text{and}\quad
f^a_1(s):=\ell(1-s),\hskip5mm\text{for any $\ell>0$},
\end{equation}
(with $\ell_1=\ell$ and $s\mapsto f_1^a(\sqrt s)=\ell(1-\sqrt s)$ convex) as well as 
\begin{equation}\label{eqexamplefb}
f^b(s):=\frac{(\ell+b(1-s))s^2}{1-s}
\quad\text{and}\quad
f^b_1(s):=(\ell+b s)(1-s)^2,
\hskip5mm\text{for any $\ell>0$, $b\in(-\ell,2\ell)$},
\end{equation}
(with  $\ell_1=2\ell-b$ and
$s\mapsto f_1^b(\sqrt s)=(\ell+b\sqrt s)(1-\sqrt s)^2$ convex).

{ We now introduce a function $\Gzero$, depending on $f$, which will turn out to be a cohesive energy density for a pristine material, i.e., using the notation of Section~\ref{sect:setting}, will satisfy $\Gzero(\cdot)=\Gdue(\cdot,0)$, for a suitable $\Gdue$. 
{The relation of {the functions $f$ and $g_0$} to the phase-field model is discussed in Section \ref{subsect:blowup} below, {see in particular \eqref{defFe}, \eqref{deffe} and \eqref{defF}}.}
We {define} $\Gzero:[0,\infty)\to[0,\infty)$ as follows (see Figure~\ref{figg1}):}
{
\begin{equation} \label{defgnew}
\Gzero(s) := \inf_{(\alpha,\beta)\in\mathcal{U}_1}\gs(\alpha,\beta),
\end{equation}
where 
\begin{equation} \label{defG}
\gs(\alpha,\beta) := \int_{-\infty}^\infty \biggl( s^2f^2(\beta)|\alpha'|^2 + \frac{(1-\beta)^2}{4} + |\beta'|^2 \biggr) \de t \,,
\end{equation}
\begin{multline} \label{defU}
\mathcal{U}_1 := \biggl\{ (\alpha,\beta)\in H^1_{\loc}(\R)\times H^1_{\loc}(\R) \,:\, \alpha'\in L^1(\R),\, \bigg|\int_{-\infty}^\infty \alpha'(t)\de t\bigg| = 1, \\
0\leq\beta\leq1, \, \lim_{|t|\to\infty}\beta(t)=1 \biggr\}.
\end{multline}}

\begin{figure}[tb]
\begin{center}
	\includegraphics[width=6cm]{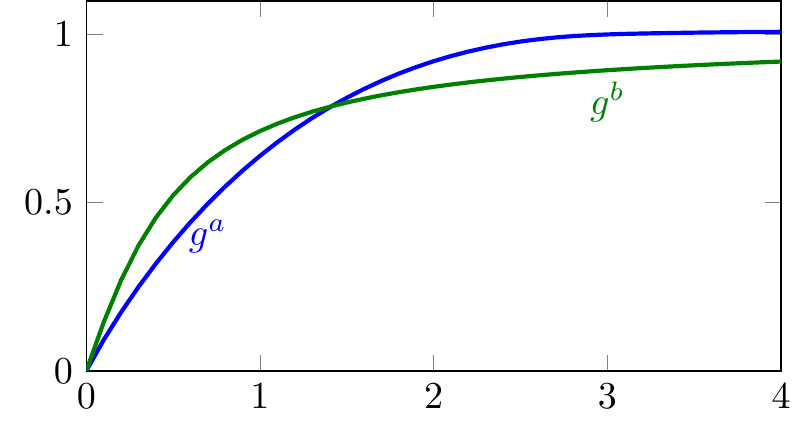}
\end{center}
	\caption{Graph of the function $\Gzero$ for two different choices of $f$. Specifically, {$g_0^a$} is generated using the function $f^a$ defined in \eqref{eqexamplefa} with $\ell=1$, and obeys {$g_0^a(s)=1$} for $s\geq\sfrac$ with $\sfrac\sim 3$; {$g_0^b$} is generated using the function $f^b$ defined in \eqref{eqexamplefb} with $\ell=1.5$, $\ell_1=0.2$ (hence $b=2.8$), and obeys {$g_0^b<1$} for all $s\in[0,\infty)$.}
	\label{figg1}
\end{figure} 


\subsection{Main properties of \texorpdfstring{$\boldsymbol{\Gzero}$}{\Gzero}}\label{subsect:propg}

{ In this section we summarize the main properties satisfied by the function $\Gzero$, defined by \eqref{defgnew}. Since the proofs are rather technical, we postpone all of them to {Section}~\ref{subsect:proofg}.

We first observe that the following equivalent characterizations hold for $\Gzero$:
\begin{align}
\Gzero(s)
&= \inf_{T>0}\inf \biggl\{\int_{-T}^T \biggl(s^2 f^2(\beta)|\alpha'|^2 + \frac{(1-\beta)^2}{4} + |\beta'|^2 \biggr) \de t \: : \: \alpha,\beta\in H^1(-T,T), \nonumber\\
& \qquad\qquad\qquad \alpha(-T)=0,\; \alpha(T)=1,\; 0\leq\beta\leq1,\; \beta(\pm T)=1 \biggr\} \label{eqcar1}\\
&=\inf \biggl\{\int_0^1 |1-\beta|\sqrt{{s^2}f^2(\beta)|\alpha'|^2+|\beta'|^2}\de t \: : \: (\alpha,\beta)\in H^1(0,1)\times H^1(0,1), \nonumber \\
& \qquad\qquad\qquad \alpha(0)=0,\, \alpha(1)={1},\, 0\leq\beta\leq1,\, \beta(0)=\beta(1)=1 \biggr\} \label{eqcar2} \\
&= \inf \biggl\{\int_0^1 \sqrt{s^2(\tif(\sqrt{\gamma}))^2 + \textstyle\frac{|\gamma'|^2}{4}} \de t \: : \: \gamma\in W^{1,1}_0([0,1],[0,1]) \biggr\}.\label{eqcar3}
\end{align}
For a proof, see Proposition~\ref{prop:g2b} and Proposition~\ref{prop:g2}.
See Figure~\ref{figgamb1} for the qualitative behaviour of the optimal profiles in the minimum problems \eqref{eqcar2} and \eqref{eqcar3}.}
We remark that $(1-\beta)f(\beta)=f_1(1-\beta)$ by \eqref{assf3} is continuous at $\beta=1$, so that the integrand in \eqref{eqcar2} is interpreted as $s \ell|\alpha'|$ on the set $\{\beta=1\}$. Notice also that the infimum in \eqref{eqcar2} is invariant under reparametrization of the interval $(0,1)$. 

\begin{figure}[tb]
\begin{center}
	\includegraphics[width=13cm]{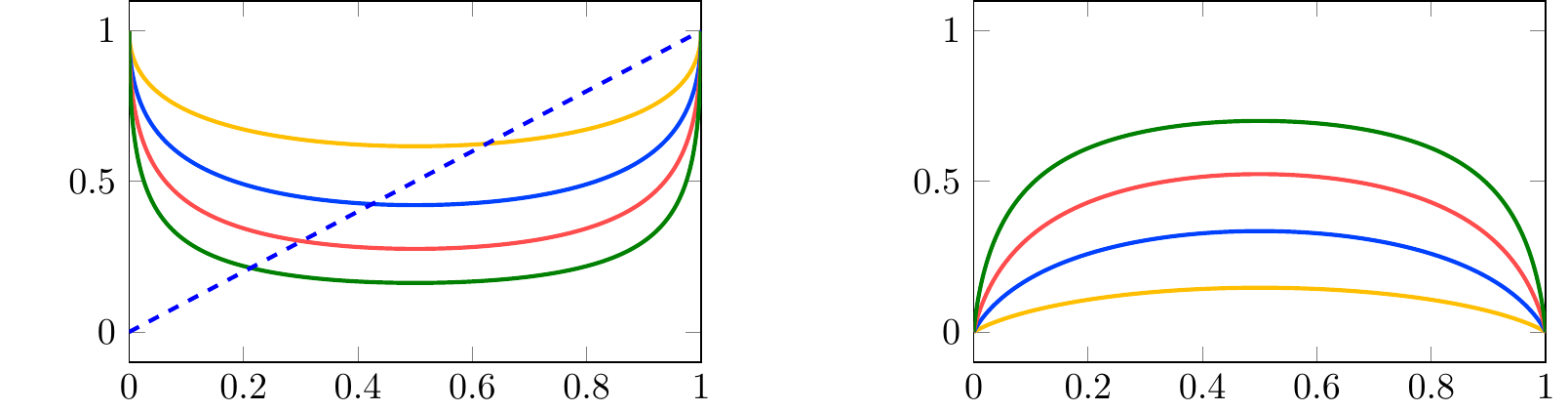}
\end{center}
	\caption{{Left: Graph of an optimal pair $(\alpha,\beta)$ from the { characterization of $\Gzero$ in \eqref{eqcar2},} for $f_1^a(s):=\ell(1-s)$, $\ell=1$, as in  \eqref{eqexamplefa}. The dashed line is $\alpha(t)=t$, the full curves are $\beta$ with (from top to bottom) $s=0.5$, $s=1$, $s=1.5$, $s=2$. Right: Graph of an optimal function $\gamma$ from the { characterization of $\Gzero$ in \eqref{eqcar3}}. From bottom to top, $s=0.5$, $s=1$, $s=1.5$, $s=2$.}}
	\label{figgamb1}
\end{figure}

The next theorem shows that $\Gzero$ enjoys the same properties as the function $\Gdue(\cdot,0)$ in Section~\ref{sect:setting}. The interpretation of $\Gzero$ as a function of type $\Gdue(\cdot,0)$, for a suitable function $\Gdue$, will be clear in Section~\ref{sect:barg}.

\begin{theorem}\label{thm:propg1}
	Let $\Gzero$ be defined by \eqref{defgnew} with a function $f$ satisfying properties
	\eqref{assf1}, \eqref{assf2}, \eqref{assf3}. Then $\Gzero$ satisfies the following properties:
	\begin{enumerate}
		\item \label{item1gbar2} $\Gzero$ is monotone nondecreasing;
		\item \label{item3gbar2} $\Gzero(s)\leq 1\wedge \ell s$
		 {and $\lim_{s\to\infty} \Gzero(s)=1$};
		\item \label{item5gbar2} there exists $\tilde{\ell}>0$ such that
		\begin{equation} \label{gnew1}
		\Gzero(s) = \ell s - \tilde{\ell}s^{5/3} + o(s^{5/3}) \qquad\text{as }s\to0^+;
		\end{equation}
		\item \label{item2gbar2} for every $s_1,s_2\geq0$,
		\begin{equation} \label{gnew2}
		\Gzero(s_1+s_2) \leq \Gzero(s_1) + \Gzero(s_2)\,;
		\end{equation}
		if in addition  $s_1,s_2>0$, the inequality is strict;
		\item \label{item4gbar2} $\Gzero$ is Lipschitz continuous with Lipschitz constant $\ell$.
	\end{enumerate}
\end{theorem}

The function $\Gzero$ satisfies $\Gzero(s)\in[0,1]$ for all $s$. Many properties of $\Gzero$ are different if $\Gzero(s)<1$ or $\Gzero(s)=1$. Therefore we define
\begin{equation}\label{eqdefsfrac}
\sfrac:=\sup \{s: \Gzero(s)<1\}\in(0,\infty]. 
\end{equation}
{ The existence and the properties of optimal profiles in the minimum problem \eqref{defgnew} are discussed in the following theorem. In particular we show that, in the case $\Gzero(s)<1$, the infimum in \eqref{defgnew} is attained by an optimal pair $(\alpha_s,\beta_s)$, and that the minimum value $m_s$ of the optimal profile $\beta_s$ is uniquely determined by $s$, and is a monotone function of $s$.}

\begin{theorem}\label{thm:g1}
Let $\Gzero$ be defined by \eqref{defgnew} with a function $f$ satisfying properties
\eqref{assf1}, \eqref{assf2}, \eqref{assf3}, and let $\sfrac$ be defined as in \eqref{eqdefsfrac}. Then the following hold:
\begin{enumerate}
\item For all $s<s_{\mathrm{frac}}$ one has $\Gzero(s)<1$, and there exists an optimal pair $(\alpha_s,\beta_s)\in\mathcal{U}_1$ such that $\Gzero(s)=\gs(\alpha_s,\beta_s)$; the minimizer of \eqref{defgnew} is unique up to translations, in the sense that if $(\alpha_s,\beta_s)$ and $(\hat\alpha_s,\hat\beta_s)$ are minimizers then there are $a_1, t_1\in\R$ such that $\alpha_s(t)=a_1+\hat\alpha_s(t-t_1)$, $\beta_s(t)=\hat\beta_s(t-t_1)$.
\item For all $s\geq s_{\mathrm{frac}}$ one has $\Gzero(s)=1$, and $\Gzero(s)=\gs(0,\beta_s)$, where $\beta_s(t):=1-e^{-\frac{|t|}{2}}$.
\end{enumerate}
Moreover, the value of the minimum of the optimal profile $\beta_s$,
\begin{equation} \label{minbeta1}
m_s := \min_{t\in\R}\beta_s(t),
\end{equation}
is uniquely determined by $s$. The map $s\mapsto m_s$ is continuous, strictly decreasing in $[0,s_{\mathrm{frac}})$, with $m_0=1$ and $m_s=0$ for $s\geq s_{\mathrm{frac}}$.
\end{theorem}
{ Fig.~\ref{figgamb2} shows the behavior of an optimal pair $(\alpha,\beta)$ in the sense of Theorem~\ref{thm:g1}(i).} 

\begin{figure}
	\begin{center}
	\includegraphics[width=\linewidth]{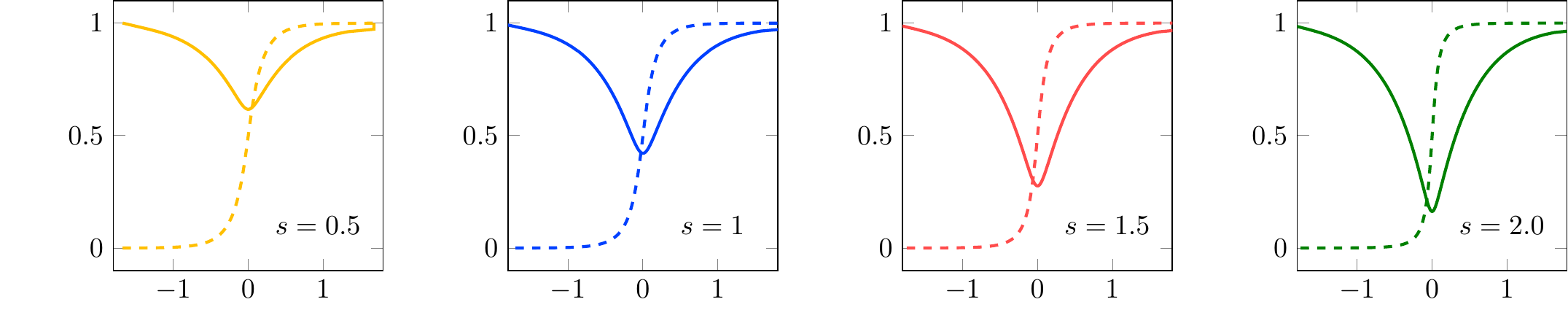}
	\end{center}
	\caption{{Graph of an optimal pair $(\alpha,\beta)$ from the {definition of $\Gzero$ in \eqref{defgnew}} with $f_1^a(s):=\ell(1-s)$, $\ell=1$, as in \eqref{eqexamplefa}, for $s=0.5$, $s=1$, $s=1.5$ and $s=2$ (from left to right). In each graph $\alpha$ is dashed, $\beta$ is the full curve.}}
	\label{figgamb2}
\end{figure}

In the following proposition we discuss under which conditions $\sfrac<\infty$.

\begin{proposition}\label{propsfracinfty2}
	Let $\slopefuno:=\liminf_{t\uparrow1} \frac{f_1(t)}{1-t^2}$, where $f_1$ is the function defined in \eqref{assf2}. If $\slopefuno=0$ then $\Gzero(s)<1$ for all $s$, so that $\sfrac=\infty$. If $\slopefuno>0$ and $s\slopefuno \ge 6$ then $\Gzero(s)=1$, so that  $\sfrac\le 6/\slopefuno$.
\end{proposition}

We observe that the function $f^a$ defined in \eqref{eqexamplefa} has $\slopefuno=\frac12\ell$ and therefore $\sfrac\in(0,12/\ell]$, whereas the function 
$f^b$ defined in \eqref{eqexamplefb} has $\slopefuno=0$ and therefore $\sfrac=\infty$. See also Figure \ref{figg1}.


\subsection{Phase-field approximation and blow-up around jump points} \label{subsect:blowup}
{ We now clarify in which sense the function $\Gzero$, defined in \eqref{defgnew}, is related to a phase-field approximation}.
Following \cite{CFI}, we introduce a family of functionals $\fe:L^1(0,1)\times L^1(0,1)\to[0,\infty]$, depending on a real parameter $\e>0$, by setting
\begin{equation} \label{defFe}
\fe(u,v):=
\begin{cases}
\displaystyle\int_0^1 \Bigl( f_\e^2(v) |u'|^2 + \frac{(1-v)^2}{4\e} + \e|v'|^2 \Bigr) \de x
&
\begin{array}{l}
\text{if } u,v\in H^1(0,1), \\ 
0\leq v\leq 1 \text{ } \Lu\text{-a.e. in }(0,1),
\end{array}\\
\infty
&\text{ otherwise.}
\end{cases}
\end{equation}
Here the function $f_\e:[0,1]\to\R$ is defined as 
\begin{equation} \label{deffe}
f_\e(s) := 1 \wedge \e^{1/2}f(s), \qquad f_\e(1)=1,
\end{equation}
where $f$ is a fixed function satisfying the assumptions \eqref{assf1}--\eqref{assf3}.

We also introduce the limit functional $\f: L^1(0,1)\times L^1(0,1)\to[0,\infty]$:
\begin{equation} \label{defF}
\f(u,v):=
\begin{cases}
\displaystyle \int_0^1 h(|u'|)\de x + \int_{J_u}\Gzero(|[u]|) \de\hz + \ell|D^cu|(0,1)
&
\begin{array}{l}
\text{if } u\in\BV(0,1), \\ 
v= 1 \text{ } \Lu\text{-a.e. in }(0,1),
\end{array}\\
\infty &\text{ otherwise}.
\end{cases}
\end{equation}
The density $h:\R\to[0,\infty)$ is defined as in \eqref{defh1},
\begin{equation} \label{defh}
h(\xi):=
{\begin{cases}
\xi^2
& \text{if } 0\leq |\xi| \leq \frac{\ell}{2},\\
\ell |\xi| - \frac{\ell^2}{4}
& \text{if  } |\xi| > \frac{\ell}{2}
\end{cases}}
\end{equation}
and $\Gzero:[0,\infty)\to[0,\infty)$ is the density defined in \eqref{defgnew}.

 We will also consider the localized versions of the functionals above, by writing $\fe(u,v;A)$ and $\f(u,v;A)$ when the domain $(0,1)$ is replaced by an open set $A\subset(0,1)$. The following $\Gamma$-convergence result is proved in \cite{CFI}.

\begin{theorem}[Conti-Focardi-Iurlano \cite{CFI}] \label{thm:cfi}
	Let $f$ obey \eqref{assf1} and \eqref{assf2}. The functionals $\fe$ defined in \eqref{defFe} $\Gamma$-converge as $\e\to0^+$ to the functional $\f$ defined in \eqref{defF} in $L^1(0,1)\times L^1(0,1)$.
	Moreover, if $(u_\e,v_\e)$ satisfies the uniform bound
	\begin{equation} \label{compactness}
	\sup_\e \Bigl( \fe(u_\e,v_\e) + \|u_\e\|_{L^1(0,1)} \Bigr) < \infty,
	\end{equation}
	then there exists a subsequence $(u_{\e_k},v_{\e_k})_k$ and a function $u\in\BV(0,1)$ such that $u_{\e_k}\to u$ almost everywhere in $(0,1)$ and $v_{\e_k}\to1$ in $L^1(0,1)$.
\end{theorem}

{ We next discuss the behavior of a recovery sequence of $\fe$ near to a limit jump point. In such points the material, originally pristine, develops a fracture. This investigation will motivate the introduction, in Section~\ref{sect:barg}, of an irreversibility constraint on the functionals $\fe$, which amounts to consider the case when a fracture develops in a material that is not pristine, but already pre-fractured in some points.} 

We show that the blow-ups of a recovery sequence $(u_\e,v_\e)$ of $\fe$ around a jump point $\bar{x}$ of the limit $u$ converge to an optimal profile for $\jump$, that is, to an optimal pair for the problem \eqref{defgnew}. Motivated by Proposition~\ref{prop:nocantor}, we will restrict to the case $u\in \SBV(0,1)$, with $|u'|\leq \frac\ell2$, and for simplicity we will further assume $\mathcal{H}^0(J_u)<\infty$. { The proof of the theorem is given in {Section}~\ref{subsect:proofblowup}.}

\begin{theorem}[Behaviour of recovery sequences] \label{thm:blowup}
	Let $f$ obey \eqref{assf1}--\eqref{assf3} and let $\fe$ be defined as in \eqref{defFe}--\eqref{deffe}.
	Let $u\in\SBV(0,1)$ have a finite number of jumps and satisfy $|u'|\leq\frac{\ell}{2}$ almost everywhere in $(0,1)$. Let $(u_\e,v_\e)$ be a recovery sequence for $\fe$ corresponding to $u$, that is, $(u_\e,v_\e)\in H^1(0,1)\times H^1(0,1)$, $0\leq v_\e \leq 1$ almost everywhere, and
	\begin{equation} \label{bua}
	u_\e\to u \text{ in }L^1(0,1),\qquad
	v_\e\to 1 \text{ in }L^1(0,1),\qquad
	\fe(u_\e,v_\e) \to \f(u,1)
	\end{equation}
	as $\e\to0^+$. Suppose in addition that $\sup_\e\|u_\e\|_\infty<\infty$. Then there exists a subsequence $\e_k\to0$ with the following properties:
	\begin{enumerate}
		\item\label{item1bu} For every jump point $\bar{x}\in J_u$ there exist $x_{k}\to \bar{x}$ such that
		\begin{equation} \label{bub}
		w_{k}(x) := v_{\e_k}(x_{k}+\e_k x) \to \beta_{\bar{s}}(x) \qquad\text{ strongly in }H^1_{\loc}(\R),
		\end{equation}
		where $\bar{s}:=|[u](\bar{x})|$ and $\beta_{\bar{s}}$ is an optimal profile for 
		{the definition of $\Gzero(\bar{s})$ in \eqref{defgnew}}, in the sense of Theorem~\ref{thm:g1}.
		Moreover, in the case $\jump<1$ we also have that
		\begin{equation} \label{buc}
		z_{k}(x) := u_{\e_k}(x_{k}+\e_k x) \to 
		{\bar s}
		\alpha_{\bar{s}}(x) \qquad\text{ {strongly} in }H^1_{\loc}(\R),
		\end{equation}
		where $\alpha_{\bar{s}}$ is such that $(\alpha_{\bar{s}},\beta_{\bar{s}})\in{\mathcal{U}_{1}}$ and $\Gzero(\bar{s}) = {\mathcal{G}_{\bar s}}(\alpha_{\bar{s}},\beta_{\bar{s}})$.
		\item\label{item2bu} For every $\eta>0$, setting $I_{\eta}:=\bigcup_{x\in J_u}(x-\eta,x+\eta)$ we have that as $\e\to0$
		\begin{equation} \label{bud}
		\int_{(0,1)\setminus I_\eta} f_{\e}^2(v_{\e})|u'_{\e}|^2\de x \to \int_{(0,1)\setminus I_\eta}|u'|^2\de x,
		\quad
		\int_{(0,1)\setminus I_\eta} \Bigl( \frac{(1-v_{\e})^2}{4\e_k} + \e|v'_{\e}|^2 \Bigr)\de x \to 0.
		\end{equation}
	\end{enumerate}
\end{theorem}


\section{Static phase-field approximation: the cohesive energy of pre-fractured material}\label{sect:barg}

\subsection{A class of cohesive energies \texorpdfstring{$\boldsymbol{\Gdue}$}{} for pre-fractured {material}}\label{subsect:barg}
We now introduce a new surface energy density ${\Gdue}$, depending on two variables $s,s'$, { and detail its connection with the phase-field approximating energies. We will show in particular that the function ${\Gdue}$ defined below satisfies all the assumptions in Section~\ref{sect:setting}, so that for such an energy density a quasi-static evolution can be constructed, as described in Section~\ref{sect:evol2}.

We let $f$ satisfy the assumptions \eqref{assf1}--\eqref{assf3}, as in Section~\ref{subsect:hyp}.} For $s,s'\geq0$ we set
\begin{equation}\label{defgbar}
{\Gdue}(s,s') :=
\inf_{(\alpha,\beta)\in\mathcal{V}_{s'}}\mathcal{G}_s(\alpha,\beta) ,
\end{equation}
where
\begin{equation} \label{defV}
\mathcal{V}_{s'} := \Bigl\{ (\alpha,\beta)\in\mathcal{U}_1 \;:\; \inf\beta\leq m_{s'}\Bigr\},
\end{equation}
$\mathcal{G}_s$ and $\mathcal{U}_1$ have been defined in \eqref{defG} and \eqref{defU} respectively,
and $m_{s'}$ denotes the value of the minimum of an optimal profile for $\Gzero(s')$, see \eqref{minbeta1}. The effect of the second variable $s'$ is to introduce a ``memory effect'' that takes into account the maximal amplitude of the jumps at previous times. { We remark that, when $s'=0$, the constraint is not active since $m_0=1$, and therefore
\begin{equation} \label{ggbar}
{\Gdue}(s,0)=\Gzero(s),
\end{equation}
where $\Gzero$ is the function defined in \eqref{defgnew}.} We refer to Figure~\ref{figgambbar1} and Figure~\ref{figgamb2gbar}  for an illustration.

\begin{figure}
	\begin{center}
		\includegraphics[width=4cm]{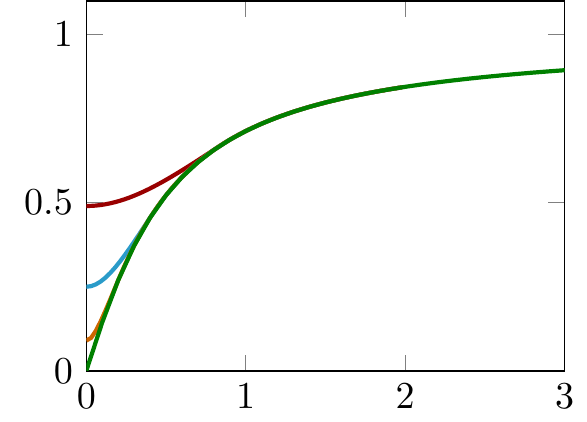}
		\includegraphics[width=10.5cm]{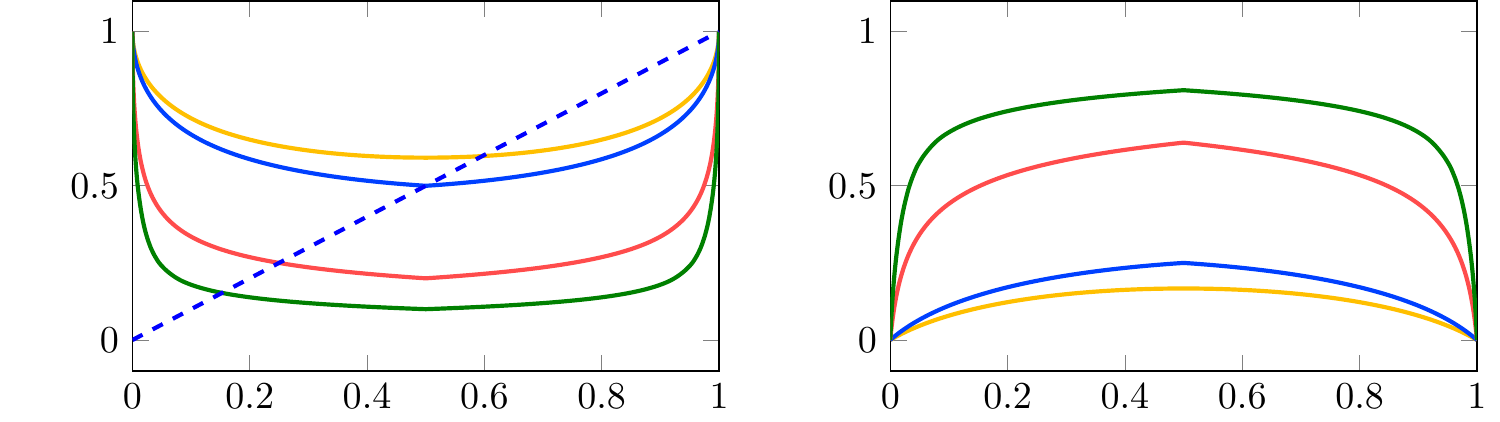}
	\end{center}
	\caption{{Left: Graph of $\Gzero(s)$ and $\Gdue(s,s')$ for $m_{s'}=0.3$, $0.5$ and $0.7$ (from bottom to top) and $f_1^b(s)$ in \eqref{eqexamplefb} with $\ell=1.5$, $\ell_1=0.2$ as in Figure \ref{figggbarintro} and Figure \ref{figg1}. 
			Middle and right: Graph of $(\alpha,\beta)$ and $\gamma$ {from {\eqref{defgbar2ter} and \eqref{defgbar2bis}}} for $s=0.3$ and $m_{s'}=0.1, 0.2, 0.5, 0.7$ and the same $f_1^b$ corresponding to Figure \ref{figgamb1}.
			For $m_{s'}=0.7$ the curves are the same as for $\Gzero$. See also Figure \ref{figgamb2gbar}.}}
	\label{figgambbar1}
\end{figure}

\begin{figure}
	\begin{center}
		\includegraphics[width=12cm]{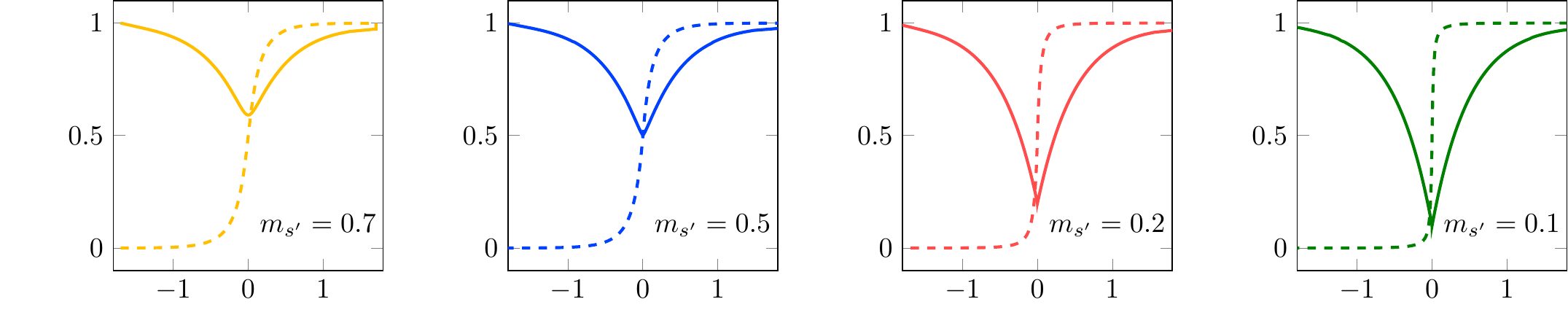}
	\end{center}
	\caption{{Graph of the optimal $(\alpha,\beta)$ entering \eqref{defgbar2} with a very large $T$, for $s=0.3$ and $m_{s'}=0.7, 0.5, 0.2, 0.1$, as
			in Figure \ref{figgambbar1}. For the largest value of $m_{s'}$ one sees that  the constraint is not active and $\beta$ is smooth at $t=0$; for the others the constraint is active and the derivative of $\beta$ jumps. Correspondingly, the profile of $\alpha$ becomes more concentrated.}}
	\label{figgamb2gbar}
\end{figure}


\subsection{Main properties of \texorpdfstring{$\boldsymbol{\Gdue}$}{{\Gdue}}}\label{subsect:propbarg} 
In this section we state the main properties of $\Gdue$, whose definition is given in \eqref{defgbar}. { The corresponding proofs are postponed to Section~\ref{subsect:proofbarg}.} The next theorem establishes that $\Gdue$ enjoys all assumptions listed in Section~\ref{sect:setting}.

\begin{theorem}\label{thm:gbar1}
	Let $\Gdue$ be defined by \eqref{defgbar} with a function $f$ satisfying properties
	\eqref{assf1}, \eqref{assf2}, \eqref{assf3}. Then $\Gdue$ satisfies hypotheses \ref{item1gbar1}--\ref{item2gbar1} of Section~\ref{sect:setting}.
\end{theorem}

Analogously to what happens for $\Gzero$ (see {Section}~\ref{subsect:propg}), the following equivalent characterizations hold for $\Gdue$:
\begin{align}
{\Gdue}(s,s')
&= {\inf_{T>0}}\inf \biggl\{\int_{-T}^T \biggl( {s^2} f^2(\beta)|\alpha'|^2 + \frac{(1-\beta)^2}{4} + |\beta'|^2 \biggr) \de t \: : \: \alpha,\beta\in H^1(-T,T), \nonumber\\
& \qquad\qquad\qquad |\alpha(T)-\alpha(-T)|={1},\; 0\leq\beta\leq1,\; \beta(\pm T)=1,\; \inf\beta\leq m_{s'} \biggr\} \label{defgbar2}\\
&=\inf \biggl\{\int_0^1 |1-\beta|\sqrt{{s^2}f^2(\beta)|\alpha'|^2+|\beta'|^2}\de t \: : \: (\alpha,\beta)\in H^1(0,1)\times H^1(0,1), \nonumber \\
& \qquad\qquad\qquad \alpha(0)=0,\, {\alpha(1)=1},\, 0\leq\beta\leq1,\, \beta(0)=\beta(1)=1,\, \inf\beta\leq m_{s'} \biggr\} \label{defgbar2ter}\\
&= \inf \biggl\{\int_0^1 \sqrt{s^2(\tif(\sqrt{\gamma}))^2 + \textstyle\frac{|\gamma'|^2}{4}} \de t \: : \: \gamma\in W^{1,1}_0([0,1],[0,1]),\, \sup\gamma\geq(1-m_{s'})^2 \biggr\}. \label{defgbar2bis}
\end{align}
Extending $\alpha$ and $\beta$ by constants it is easy to show that the $\inf$ in the first line of \eqref{defgbar2} is nonincreasing in $T$. Finally, the following technical lemma holds.

\begin{lemma}\label{lem:gbarmu1}
For $\mu>0$, let ${\Gdue}^{(\mu)}:[0,\infty)\times[0,\infty)\to[0,\infty)$ be defined by
	\begin{equation}\label{defgbarmu1}
	\begin{split}
	{\Gdue}^{(\mu)}(s,s') &:= {\inf_{T>0}}\inf \biggl\{ \int_{-T}^{T}\biggl( f^2(\beta)|\alpha'|^2 + \frac{(1-\beta)^2}{4} + |\beta'|^2 \biggr) \de t  \: : \: \alpha,\beta\in H^1(-T,T),\\
	& \qquad\qquad |\alpha(T)- \alpha(-T)|=s,\, 0\leq\beta\leq1,\, \beta(\pm T)= 1-\mu,\, \inf\beta\leq m_{s'} \biggr\} \,.
	\end{split}
	\end{equation}
Then we have 
\begin{equation}\label{eq:gbareta}
|{\Gdue}(s,s')-{\Gdue}^{(\mu)}(s,s')|\leq 3\mu^2,
\end{equation}
for every $s,s'\geq0$.
\end{lemma}


\subsection{Phase-field approximation}\label{subsect:constrainedpbl}
{ In this subsection we show that the new energy density $\Gdue$ defined in \eqref{defgbar} appears in the $\Gamma$-limit of the phase-field energies $\fe$ introduced in \eqref{defFe} when we include a suitable irreversibility constraint.
Precisely, below we first introduce an irreversibility constraint at level $\e$ on the functionals $\fe$, in the form of a monotonicity condition on the \emph{minimum values} of the damage variable $v_\e$. This choice is motivated by the blow-up analysis performed in Theorem~\ref{thm:blowup}. Then, we prove that the $\Gamma$-limit of such constrained functionals is of the form \eqref{defPhi}, with $\Gdue$ given by \eqref{defgbar}.}

We assume in the following that a pair $(\Gamma,s)$ as in \eqref{previousjump} is given: $\Gamma$ is a finite subset of $[0,1]$, and $s:\Gamma\to(0,\infty)$. We also fix a boundary datum $b:\{0,1\}\to\R$. In order to deal with the boundary conditions, it is convenient to work in a larger open interval $\Omega$ containing $[0,1]$, for instance $\Omega:=(-1,2)$. For $\e>0$ let $L_\e>0$ be such that $L_\e\to0$ and $\frac{L_\e}{\e}\to\infty$ as $\e\to0$. We introduce a constrained functional, defined on $L^1(\Omega)\times L^1(\Omega)$, by setting
\begin{equation} \label{defFebar}
\febar(u,v;\Gamma,s,b):=
\begin{cases}
\fe(u,v;\Omega)
&\begin{array}{l}
\text{if } v(x)\leq m_{s(x)} \text{ for every }x\in\Gamma, \\ 
\text{$u(x)=b(0)$ for $x<-L_\e$, $u(x)=b(1)$ for $x>1+L_\e$},
\end{array}\\
\infty
&\,\,\,\text{otherwise.}
\end{cases}
\end{equation}
Here $m_s$ denotes the minimum value of an optimal profile $\beta_s$ for $\Gzero(s)$, see Theorem~\ref{thm:g1} and in particular \eqref{minbeta1}.
Recalling the definition \eqref{defPhi} of the relaxed functional $\Phi(\cdot\,;\Gamma,s,b)$, the limit functional is defined on $L^1(\Omega)\times L^1(\Omega)$ by
\begin{equation} \label{defFbar}
\fbar(u,v;\Gamma,s,b):=
\begin{cases}
\Phi(u;\Gamma,s,b)
&\begin{array}{l}
\text{if } u\in \BV(\Omega), \text{ } v=1\text{ a.e.},\\
\text{$u(x)=b(0)$ for $x<0$, $u(x)=b(1)$ for $x>1$},
\end{array}\\
\infty
&\,\,\,\text{otherwise.}
\end{cases}
\end{equation}
The main result of this section is the following.

\begin{theorem} \label{thm:gconv}
Let a finite set $\Gamma\subset[0,1]$, a map $s:\Gamma\to(0,\infty)$, and a boundary Dirichlet datum $b:\{0,1\}\to\R$ be given.
Then the functionals $\febar(\cdot\,;\Gamma,s,b)$ $\Gamma$-converge as $\e\to0^+$ to $\fbar(\cdot\,;\Gamma,s,b)$ in $L^1(\Omega)\times L^1(\Omega)$.
\end{theorem}

In the rest of this section we drop the dependence on $(\Gamma,s,b)$ in the functionals $\febar$ and $\fbar$ to lighten the notation, as these quantities are fixed. In order to prove the theorem we introduce the following standard notions:
\begin{align*}
\fbar'(u,v) &:= \Gamma\text{-}\liminf_{\e\to0}\febar(u,v)\\
&\qquad :=\inf \Bigl\{\liminf_{\e\to0}\febar(u_\e,v_\e) \,:\, (u_\e,v_\e)\to(u,v) \text{ in }L^1(\Omega)\times L^1(\Omega)\Bigr\}\,,\\
\fbar''(u,v) &:= \Gamma\text{-}\limsup_{\e\to0}\febar(u,v)\\
&\qquad :=\inf \Bigl\{\limsup_{\e\to0}\febar(u_\e,v_\e) \,:\, (u_\e,v_\e)\to(u,v) \text{ in }L^1(\Omega)\times L^1(\Omega)\Bigr\}\,.
\end{align*}
The proof of Theorem~\ref{thm:gconv} follows by combining Proposition~\ref{prop:liminf} and Proposition~\ref{prop:limsup} below. 

\begin{proposition}[Liminf inequality]\label{prop:liminf}
For every $(u,v)\in L^1(\Omega)\times L^1(\Omega)$ it holds
\begin{equation} \label{gammaliminf}
\fbar(u,v)\leq \fbar'(u,v)\,.
\end{equation}
\end{proposition}

\begin{proof}
The proof follows the lines of \cite[Proposition~5.1]{CFI}, with the natural modifications required to include the additional constraint.
We denote $\Gamma=\{\bar{x}_1,\ldots,\bar{x}_k\}$, and $\bar{s}_i=s(\bar{x}_i)$.
Fix any sequence $(u_\e,v_\e)\to(u,v)$ in $L^1(\Omega)\times L^1(\Omega)$ with $\sup_\e\febar(u_\e,v_\e)<\infty$. The proof will be achieved by showing that $u\in\BV(\Omega)$, $v=1$ almost everywhere, and
\begin{equation} \label{liminf1}
\fbar(u,1) \leq \liminf_{\e\to0}\fe(u_\e,v_\e;\Omega)\,.
\end{equation}
By possibly passing to a subsequence, we can assume without loss of generality that the $\liminf$ in \eqref{liminf1} is in fact a limit, and that the convergence of $u_\e$ and $v_\e$ is also pointwise almost everywhere. The uniform bound on the energy of $(u_\e,v_\e)$ gives $v=1$ almost everywhere, and $v_\e(\bar{x}_i) \leq m_{\bar{s}_i}<1$. Moreover, $u(x)=b(0)$ for $x<0$, $u(x)=b(1)$ for $x>1$.

By repeating the construction in the first part of the proof of \cite[Proposition~5.1]{CFI}, given any $\delta>0$ one can determine a finite number of points $S=\{ t_1,\ldots,t_L \}\subset\Omega$, with $\Gamma\subset S$, with the following property: for all $\eta>0$ sufficiently small, {$\eta\ll\delta$ and $\eta<1-m_{\bar s_i}$ for all $i$}, setting
\begin{equation*}
S_\eta := \bigcup_{i=1}^L(t_i-\eta,t_i+\eta),
\end{equation*}
one has for all $\e$ sufficiently small (depending on $\eta$)
\begin{equation} \label{liminf2}
(1-\omega(\delta))\int_{\Omega\setminus S_\eta} h(|u_\e'|)\de x \leq \fe(u_\e,v_\e;\Omega\setminus S_\eta)
\end{equation}
where $\omega(\delta)\to0$ as $\delta\to0$ is a modulus of continuity.

The uniform bound on the energies $\fe(u_\e,v_\e)$, together with \eqref{liminf2}, yields
\begin{equation*}
\sup_\eta\sup_{\e} \int_{\Omega\setminus S_\eta}|u_\e'|\de x <\infty\,,
\end{equation*}
and therefore $u\in \BV(\Omega\setminus S_\eta)$ for all $\eta>0$; in turn we have $u\in \BV(\Omega)$ by the finiteness of $S$, and by lower semicontinuity we obtain
\begin{equation} \label{liminf3}
(1-\omega(\delta))\fbar(u,1;\Omega\setminus S_\eta) \leq \liminf_{\e\to0}\fe(u_\e,v_\e;\Omega\setminus S_\eta)\,.
\end{equation}

We now estimate the contribution to the energy coming from the region $S_\eta$. We first consider the points $\bar{x}_i\in\Gamma$, $i=1,\ldots,k$, which in particular belong to $S$ (recall that $\Gamma\subset S$). Let $J^i_\eta{:=}(\bar{x}_i-\eta,\bar{x}_i+\eta)$. We introduce another small parameter $\mu>0$ and we choose $x_1,x_2\in J_\eta^i$, $x_1<\bar x_i<x_2$, with the following properties:
\begin{align}
&v_\e(x_1)\to1, & & v_\e(x_2)\to1, \label{liminf5}\\
&u_\e(x_1)\to u(x_1), & & u_\e(x_2)\to u(x_2), \label{liminf6}\\
&|u(x_1)-u^-(\bar x_i)|<\mu, & & |u(x_2)-u^+(\bar x_i)|<\mu. \label{liminf7}
\end{align}
We define $I{:=}(x_1,x_2)$ and denote by $C_\e$ the connected component of the set
\begin{equation} \label{liminf8}
\{ x\in I : v_\e(x)<1-\eta \}
\end{equation}
containing $\bar{x}_i$, and by $C_\e^j$ the connected components (different from $C_\e$) of the same set \eqref{liminf8} in which $v_\e$ achieves the value $1-\delta$ (recall that $\eta\ll\delta$, {$\eta<1-m_{\bar s_i}$} {and \eqref{liminf5}}). In each such component $C_\e^j$ we have, denoting by $y$ its first endpoint and by $z$ an interior point in which {$v_\e(z)=1-\delta$,}
\begin{align*}
\fe(u_\e,v_\e;C_\e^j)
& \geq \int_{y}^z \biggl( \frac{(1-v_\e)^2}{4\e}+\e|v_\e'|^2 \biggr)\de x \geq \int_y^z (1-v_\e)|v_\e'|\de x \\
& \geq \frac12\bigl(1-v_\e(z)\bigr)^2 - \frac12\bigl(1-v_\e(y)\bigr)^2 = \frac{\delta^2-\eta^2}{2}
{\ge\frac14\delta^2}\,.
\end{align*}
Therefore the number ${N_\e}$ of the components $C^j_\e$ is uniformly bounded by {$\frac{c}{\delta^2}$,} where $c$ is a constant independent of $\e$. Moreover in each set $C^j_\e$ and in $C_\e$ we have $f_\e(v_\e)=\sqrt{\e}f(v_\e)$ for $\e$ small, as $v_\e<1-\eta$. Recalling that $\bar{x}_i\in C_\e$ and $v_\e(\bar{x}_i)\leq m_{\bar{s}_i}$, we can now compute using the rescaling $\alpha_\e(t){:=}u_\e(\bar{x}_i+\e t)$, $\beta_\e(t){:=}v_\e(\bar{x}_i+\e t)$
\begin{align} \label{liminf9}
\fe(u_\e,v_\e;C_\e)
& = \int_{\frac{1}{\e}(C_\e-\bar{x}_i)} \Bigl( f^2(\beta_\e)|\alpha_\e'|^2 + \frac{(1-\beta_\e)^2}{4} + |\beta_\e'|^2 \Bigr)\de t \nonumber \\
& \geq {\Gdue}^{(\eta)} \biggl(\bigg| \int_{C_\e} u_\e'\de x \bigg|, \bar{s}_i\biggr)
\geq {\Gdue}\biggl(\bigg| \int_{C_\e} u_\e'\de x \bigg|, \bar{s}_i\biggr) -3\eta^2\,,
\end{align}
{ where ${\Gdue}^{(\eta)}$ is the function defined in \eqref{defgbarmu1} and the last inequality follows by Lemma~\ref{lem:gbarmu1}.}
On the other components $C_\e^j$ {the same argument gives}
\begin{align} \label{liminf10}
\fe(u_\e,v_\e;C_\e^j) & 
\geq {{\Gdue} \biggl(\bigg| \int_{C_\e^j} u_\e'\de x \bigg|,0 \biggr) -3\eta^2}
= {\Gzero \biggl(\bigg| \int_{C_\e^j} u_\e'\de x \bigg| \biggr) -3\eta^2}\,,
\end{align}
{ where $\Gzero$ has been defined in \eqref{defgnew} and satisfies \eqref{ggbar}.}
Finally, outside the selected components, that is in the set $\tilde{C}_\e:=I\setminus (C_\e\cup\bigcup_{j=1}^{{N_\e}} C_\e^j)$, one has $v_\e\geq1-\delta$ and therefore an estimate analogous to \eqref{liminf2} holds:
\begin{align} \label{liminf11}
\fe(u_\e,v_\e;\tilde{C}_\e) & \geq (1-\omega(\delta))\int_{\tilde{C}_\e}h(|u_\e'|)\de x  \geq (1-\omega(\delta)) \biggl[ \ell\int_{\tilde{C}_\e}|u_\e'|\de x - \frac{\ell^2}{4}|\tilde{C}_\e| \biggr] \nonumber\\
& \geq (1-\omega(\delta))\Gzero \biggl(\bigg| \int_{\tilde{C}_\e} u_\e'\de x \bigg| \biggr) - \frac{\ell^2}{2}\eta\,.
\end{align}
where we used the definition \eqref{defh} of $h$ { and Theorem~\ref{thm:propg1}\ref{item3gbar2}.}
By collecting \eqref{liminf9}--\eqref{liminf11}
\begin{align*}
\fe(u_\e,v_\e;I)
& \geq {\Gdue}\biggl(\bigg| \int_{C_\e} u_\e'\de x \bigg|, \bar{s}_i\biggr) -3\eta^2 + \sum_{j=1}^{N_\e} \Gzero \biggl(\bigg| \int_{C_\e^j} u_\e'\de x \bigg| \biggr) - 3{N_\e}\eta^2 \\
&\qquad\qquad + (1-\omega(\delta))\Gzero \biggl(\bigg| \int_{\tilde{C}_\e} u_\e'\de x \bigg| \biggr) - \frac{\ell^2}{2}\eta \\
& \geq {\Gdue}\biggl(\bigg| \int_{C_\e} u_\e'\de x \bigg|, \bar{s}_i\biggr) + (1-\omega(\delta)) \Gzero \biggl(\bigg| \int_{I\setminus C_\e} u_\e'\de x \bigg| \biggr) -\Bigl(3+\frac{c}{{\delta^2}}\Bigr)\eta^2 - \frac{\ell^2}{2}\eta \\
& \geq (1-\omega(\delta)) {\Gdue}\bigl(|u_\e(x_2)-u_\e(x_1)|, \bar{s}_i\bigr) -\frac{c}{{\delta^2}}\eta^2 - \frac{\ell^2}{2}\eta\,,
\end{align*}
{ where we used \eqref{gnew2} in the second passage and \eqref{gbarsub1} in the third one}. Hence \eqref{liminf6} {and the continuity of $\Gdue$} yield
\begin{equation*}
\liminf_{\e\to0}\fe(u_\e,v_\e;J^i_\eta) \geq (1-\omega(\delta)) {\Gdue}\bigl(|u(x_2)-u(x_1)|, \bar{s}_i\bigr) -\frac{c}{{\delta^2}}\eta^2 - \frac{\ell^2}{2}\eta\,,
\end{equation*}
and finally letting $\mu\to0$
\begin{equation} \label{liminf12}
\liminf_{\e\to0}\fe(u_\e,v_\e;J^i_\eta) \geq (1-\omega(\delta)) {\Gdue}\bigl(|[u]|(\bar x_i) , \bar{s}_i\bigr) -\frac{c}{{\delta^2}}\eta^2 - \frac{\ell^2}{2}\eta\,.
\end{equation}

The inequality \eqref{liminf12} gives an estimate of the contribution to the energy coming from the points in $\Gamma$. For the points $t_i\in S\setminus\Gamma$, we can reproduce the argument above just removing the component $C_\e$ (in this case the argument is the same as in the proof of \cite[Proposition~5.1]{CFI}) and obtain that, for $I_\eta^i=(t_i-\eta,t_i+\eta)$,
\begin{equation} \label{liminf4}
\liminf_{\e\to0}\fe(u_\e,v_\e;I_\eta^i) \geq (1-\omega(\delta)) \Gzero(|[u]|(\bar x_i)) + O(\eta)\,.
\end{equation}

Eventually we collect \eqref{liminf3}, \eqref{liminf12}, and \eqref{liminf4} and we let $\eta\to0$:
\begin{align*}
\liminf_{\e\to0}\fe(u_\e,v_\e;\Omega) \geq (1-\omega(\delta))\fbar(u,1).
\end{align*}
The conclusion \eqref{liminf1} follows by letting $\delta\to0$.
\end{proof}

\begin{proposition}[Limsup inequality]\label{prop:limsup}
For every $(u,v)\in L^1(\Omega)\times L^1(\Omega)$ it holds
\begin{equation} \label{gammalimsup}
\fbar''(u,v)\leq\fbar(u,v)\,.
\end{equation}
\end{proposition}

\begin{proof} 
{Let us consider first} the case {$v\equiv 1$ and} $u\in\SBV(0,1)$, {with $u'\in L^2(0,1)$} and $\mathcal{H}^0(J_u)<\infty$, and {recall that} by assumption also $\mathcal{H}^0(\Gamma)<\infty$.
By a localization argument we can further assume that $\Gamma$ consists of a single point $x_0\in[0,1]$, and that $u$ has at most one jump point, also located at $x_0$. Indeed, in any interval which does not contain any point of $\Gamma$ the conclusion follows directly by Theorem~\ref{thm:cfi}.

Let therefore $\Gamma=\{x_0\}$, $x_0\in[0,1]$, and let $s_0:=s(x_0)$.
Let us also assume for the moment that $u$ only takes the two values $u^{\pm}(x_0)$ in a neighbourhood of $x_0$. Given any $\sigma>0$, { using the characterization \eqref{defgbar2} of ${\Gdue}$} {(and rescaling $\alpha$)} we can find $T>0$ and $\alpha,\beta\in H^1(-T,T)$ such that $\alpha(-T)=u^-(x_0)$, $\alpha(T)=u^+(x_0)$, $0\leq\beta\leq1$, $\beta(\pm T)=1$, $\beta(0)\leq m_{s_0}$, and
\begin{equation*}
\int_{-T}^T \biggl( f^2(\beta)|\alpha'|^2 + \frac{(1-\beta)^2}{4} + |\beta'|^2 \biggr) \de t \leq {\Gdue}(|[u](x_0)|,s_0) + \sigma .
\end{equation*}
We then take as recovery sequence 
\begin{equation*}
u_\e(x):=
\begin{cases}
\alpha\bigl(\frac{x-x_0}{\e}\bigr) & \text{if } x\in A_\e,\\
u(x) & \text{if }x\in \Omega\setminus A_\e,
\end{cases}
\qquad
v_\e(x):=
\begin{cases}
\beta\bigl(\frac{x-x_0}{\e}\bigr) & \text{if } x\in A_\e,\\
1 & \text{if }x\in \Omega\setminus A_\e,
\end{cases}
\end{equation*}
where $A_\e:=(x_0-\e T,x_0+\e T)$. Notice that, since $\frac{L_\e}{\e}\to\infty$ as $\e\to0$, we have $A_\e\subset(x_0-L_\e,x_0+L_\e)$ for $\e$ small enough, and therefore $u_\e$ satisfies the boundary conditions as in \eqref{defFebar} in the case $x_0$ is a boundary point. It is easily seen that $u_\e\to u$ and $v_\e\to1$ in $L^1(0,1)$. Moreover, as $v_\e(x_0)=\beta(0)\leq m_{s_0}$ we have
\begin{align*}
\febar(u_\e,v_\e) &= \int_{\Omega\setminus A_\e} |u'|^2\de x + \int_{A_\e} \biggl(f_\e^2(v_\e)|u_\e'|^2 + \frac{(1-v_\e)^2}{4\e} + \e|v_\e'|^2 \biggr)\de x \\
&\leq \int_0^1 |u'|^2\de x + \int_{-T}^T \biggl(f^2(\beta)|\alpha'|^2 + \frac{(1-\beta)^2}{4} + |\beta'|^2 \biggr)\de t \\
&\leq \int_0^1 |u'|^2\de x + {\Gdue}(|[u](x_0)|,s_0) + \sigma.
\end{align*}
Since $\sigma$ is arbitrary we obtain
\begin{equation} \label{limsup1}
\fbar''(u,1) \leq \limsup_{\e\to0}\febar(u_\e,v_\e) \leq \int_0^1 |u'|^2\de x + {\Gdue}(|[u](x_0)|,s_0).
\end{equation}
In order to remove the assumption that $u$ is locally piecewise constant in a neighbourhood of $x_0$, we consider the sequence
$$
u_j(x) :=
\begin{cases}
u(x_0-\frac1j) & \text{for }x_0-\frac1j < x < x_0, \\
u(x_0+\frac1j) & \text{for }x_0 < x < x_0+\frac1j, \\
u(x) & \text{for } x\in\Omega\setminus(x_0-\frac1j,x_0+\frac1j).
\end{cases}
$$
In view of the previous discussion inequality \eqref{limsup1} holds with $u$ replaced by $u_j$; since $u_j\to u$ in $L^1(0,1)$ {and $u'_j\to u'$ in $L^2(0,1)$} as $j\to\infty$, by lower semicontinuity of $\fbar''$ we conclude that \eqref{limsup1} is still satisfied by $u$.

The inequality \eqref{limsup1}, together with Theorem~\ref{thm:cfi} and a localization argument, proves that for every $u\in\SBV(0,1)$ with {$u'\in L^2(0,1)$ and} $\hz(J_u)<\infty$ we have
\begin{equation*}
\fbar''(u,1) \leq \int_0^1 |u'|^2\de x + \sum_{x\in\Gamma} {\Gdue}(|[u](x)|,s(x)) + \sum_{x\in J^b(u)\setminus\Gamma} \Gzero(|[u](x)|) \,.
\end{equation*}
Then the conclusion \eqref{gammalimsup} follows since $\fbar$ is the lower semicontinuous envelope of the right-hand side, by Theorem~\ref{thm:relaxation} and Remark~\ref{rmk:relaxation}.
\end{proof}

\begin{remark} \label{rmk:cfi}
From Theorem~\ref{thm:gconv}, in the particular case $\Gamma=\emptyset$, it follows that the $\Gamma$-convergence result of \cite{CFI} (see Theorem~\ref{thm:cfi}) continues to hold if we include Dirichlet boundary conditions.
\end{remark}


\section{Static phase-field approximation: proofs} \label{sect:proofs}

{ We prove in this section all properties stated in {Section}s~\ref{subsect:propg} and \ref{subsect:propbarg} of the fracture energy densities $\Gzero$ and ${\Gdue}$, defined in \eqref{defgnew} and \eqref{defgbar} respectively.}
{ We also give the proof of the blow-up result in {Section}~\ref{subsect:blowup}. Along this section, we work under the assumptions \eqref{assf1}--\eqref{assf3} on $f$.}


\subsection{Proof of the statements of Section~\ref{subsect:propg}}\label{subsect:proofg}

{ We detail is a series of propositions the properties of the function $\Gzero$, defined in \eqref{defgnew}. The proofs of the results stated in Section~\ref{subsect:propg} follow by combining the propositions below.

We start by proving the two alternative representation formulas \eqref{eqcar1} and \eqref{eqcar2} for $\Gzero$, and by discussing the existence and the properties of optimal profiles for the minimum problem \eqref{defgnew}.}

\begin{proposition} \label{prop:g2b}
The following properties hold.
	\begin{enumerate}
		\item \label{prop:g2bcar} { The equivalent characterizations \eqref{eqcar1} and \eqref{eqcar2} hold.}
		\item\label{prop:g2bexist} {For $\Gzero(s)<1$ the problem in (\ref{defgnew}) has a minimizer $(\alpha_s,\beta_s)\in\mathcal U_1$.}
		\item\label{prop:g2bming} {Any minimizer obeys $\min \beta_s\ge 1- \sqrt{\Gzero(s)}$.}
		\item\label{prop:g2bmong} {If $0\le s'<s$ and $\Gzero(s)<1$, then $\Gzero(s')<\Gzero(s)$.}
		\item\label{prop:g2bmonoton} 
		For any minimizer $(\alpha_s,\beta_s)$ there is $t_*\in\R$ such that $\beta_s$ is nonincreasing in $(-\infty,t_*)$ and nondecreasing in $(t_*,\infty)$.
		\item\label{prop:g2bC1}
		If $\Gzero(s)<1$, then for any minimizer $(\alpha_s,\beta_s)$ there are $T_-,T_+\in{\R\cup\{\pm\infty\}}$ such that $\beta_s\in C^1((T_-,T_+);{[0,1)})$ and $\alpha_s \in C^1(\R)$, with $\beta_s=1$ and $\alpha_s'=0$ on $\R\setminus(T_-,T_+)$. The map $\alpha_s-\alpha_s(-\infty)$ is a $C^1$ bijection of $(T_-,T_+)$ onto $(0,1)$.
		\item\label{prop:equip} {Any minimizer  $(\alpha_s,\beta_s)$ obeys
			\begin{equation}\label{eqequipartab}
			s^2f^2(\beta_s) |\alpha_s'|^2+|\beta_s'|^2=\frac{(1-\beta_s)^2}{4}  \text{ pointwise in $\R$}.
			\end{equation}}
	\end{enumerate}
\end{proposition}

\begin{proof}
{ To prove \ref{prop:g2bcar}, we show the equivalence of \eqref{defgnew} and \eqref{eqcar1}; the equivalence of \eqref{eqcar1} and \eqref{eqcar2} is proved in \cite[Proposition~4.3]{CFI}. Let $\hat{g}$ denote the infimum in \eqref{eqcar1}. Notice that the minimum problems in \eqref{eqcar1} are decreasing with respect to $T$; the inequality $\Gzero(s)\leq\hat{g}(s)$ is trivial, since if $\alpha,\beta\in H^1(-T,T)$ are admissible functions in \eqref{eqcar1} they can be extended as constants outside $(-T,T)$ to obtain a pair in $\mathcal{U}_1$.

To prove the opposite inequality, we then have to show that for every $(\alpha,\beta)\in\mathcal{U}_1$ the inequality $\hat{g}(s)\leq\gs(\alpha,\beta)$ holds.}
Since translations and truncations of $\alpha$ do not increase the energy, as well as the symmetric reflection with respect to the origin of $\alpha$ and $\beta$, we can assume without loss of generality that $0\leq\alpha\leq1$, $\lim_{t\to-\infty}\alpha(t)=0$, $\lim_{t\to\infty}\alpha(t)=1$.
	
Fix $T>0$ and let
\begin{equation*}
M_T := 1-\sqrt{1-\alpha(T)}\,.
\end{equation*}
Notice that $0\leq M_T\leq1$ and $\lim_{T\to\infty}M_T=1$. We define a new pair $(\alpha_T, \beta_T)$, admissible for the minimum problem \eqref{eqcar1}, by modifying the functions $(\alpha,\beta)$ outside the interval $(-T,T)$ as follows:
\begin{equation*}
\alpha_T(t) :=
\begin{cases}
\alpha(t) & \text{if } t\in(-T,T),\\
\alpha(T) & \text{if } t\in[T,T+1),\\
\text{linear interpolation }& \text{if } t\in[T+1,T+2),\\
1 & \text{if } t \in [T+2,\infty),\\
\end{cases}
\end{equation*}
\begin{equation*}
\beta_T(t) :=
\begin{cases}
\beta(t) & \text{if } t\in(-T,T),\\
\text{linear interpolation } & \text{if } t\in[T,T+1),\\
M_T & \text{if } t\in[T+1,T+2),\\
\text{linear interpolation } & \text{if } t \in [T+2,T+3),\\
1 & \text{if }t\in[T+3,\infty)
\end{cases}
\end{equation*}
(in the interval $(-\infty,-T)$ we do a symmetric construction, with the value $M_T$ replaced by $M'_{T}:=1-\sqrt{\alpha(-T)}$). Then we have $\alpha_T,\beta_T\in H^1(-T-3,T+3)$, $\alpha_T(-T-3)=0$, $\alpha_T(T+3)=1$, $0\leq\beta_T\leq1$, $\beta_T(-T-3)=\beta_T(T+3)=1$, and the pair $(\alpha_T,\beta_T)$ is therefore admissible in the minimum problem \eqref{eqcar1}. Furthermore
\begin{align*}
\omega_+(T) :&= \int_T^{T+3} \biggl( s^2f^2(\beta_T)|\alpha_T'|^2 + \frac{(1-\beta_T)^2}{4} + |\beta_T'|^2\biggr) \de t \\
& = \int_{(T,T+1)\cup(T+2,T+3)} \biggl( \frac{(1-\beta_T)^2}{4} + |\beta_T'|^2 \biggr)\de t +  s^2f^2(M_T)|1-\alpha(T)|^2 + \frac{(1-M_T)^2}{4} \\
& = \frac14\int_0^1\bigl(1-\beta(T)-(M_T-\beta(T))t\bigr)^2\de t + \frac14\int_0^1\bigl(1-M_T-(1-M_T)t\bigr)^2\de t \\
& \qquad + |\beta(T)-M_T|^2 + |1-M_T|^2 + s^2\bigl(f(M_T)(1-M_T)\bigr)^2(1-\alpha(T)) + \frac{(1-M_T)^2}{4}\,,
\end{align*}
and from the fact that $\alpha(T),\beta(T),M_T\to 1$ as $T\to\infty$ we obtain, using assumption \eqref{assf3},
\begin{equation*}
\omega_+(T) \to 0 \quad\text{as }T\to\infty\,.
\end{equation*}
Similarly
\begin{equation*}
\omega_-(T):= \int_{-T-3}^{-T} \biggl( s^2f^2(\beta_T)|\alpha_T'|^2 + \frac{(1-\beta_T)^2}{4} + |\beta_T'|^2\biggr) \de t \to 0 \qquad\text{as }T\to\infty\,.
\end{equation*}
Therefore
\begin{align*}
\gs(\alpha,\beta)
& \geq \int_{-T}^T \biggl( s^2f^2(\beta)|\alpha'|^2 + \frac{(1-\beta)^2}{4} + |\beta'|^2 \biggr) \de t \\
& = \int_{-T-3}^{T+3} \biggl( s^2f^2(\beta_T)|\alpha_T'|^2 + \frac{(1-\beta_T)^2}{4} + |\beta_T'|^2 \biggr) \de t - \omega_+(T) - \omega_-(T) \\
& \geq {\hat{g}(s)} -\omega_+(T) -\omega_-(T)
\end{align*}
and the conclusion follows by sending $T\to\infty$. This concludes the proof of \ref{prop:g2bcar}.

We now observe that for any $\beta$ with $ 1-\beta\in H^1(\R)$ and any $t_0\in\R$ we have
\begin{equation}\label{g=1}
\int_{-\infty}^\infty \Bigl( \frac{(1-\beta)^2}{4} + |\beta'|^2 \Bigr)\de t
\ge -\int_{-\infty}^{t_0} (1-\beta)\beta'\de t +\int_{t_0}^\infty (1-\beta)\beta'\de t 
=(1-\beta(t_0))^2,
\end{equation}
with equality if and only if $\beta(t)=1-(1-\beta(t_0))e^{-|t-t_0|/2}$. Assertion \ref{prop:g2bming} follows immediately.
	
{Let us prove \ref{prop:g2bexist}. Assume now that $\Gzero(s)<1$, which is equivalent to $s<\sfrac$. 
By equation \eqref{g=1}, in the minimization it is sufficient to consider functions $\beta$ which obey $\beta>0$ almost everywhere.
For a fixed $\beta$, \eqref{defgnew} is a linear problem in $\alpha$. One can check that the minimizer is such that $f^2(\beta)\alpha'$ is constant on the set $\{\beta<1\}$, and that $\alpha'=0$ on $\{\beta=1\}$ (up to null sets).
Therefore the minimizer takes the form $\alpha'(t)=c f^{-2}(\beta(t))$ for some $c=c(s,\beta)\in\R$, where to shorten notation we write $f^{-2}(1)=0$. Integrating shows that $1=c\int_{-\infty}^\infty f^{-2}(\beta)\de t$, so that}
\newcommand\Gsbeta{{\hat {\mathcal G}_s}}
\newcommand\Gspbeta{{\hat {\mathcal G}_{s'}}}
\begin{equation*}
{\Gsbeta(\beta):=\min_{\alpha} \mathcal G_s(\alpha,\beta)=s^2\left(\int_{-\infty}^\infty f^{-2}(\beta)\de t\right)^{-1} + \int_{-\infty}^\infty \Bigl(\frac{(1-\beta)^2}4 + |\beta'|^2 \Bigl)\de t.}
\end{equation*}
Let now $\beta_k\in H^1_{\loc}(\R;[0,1])$ be a minimizing sequence for $\Gsbeta$. 
Since $\inf \Gsbeta=\Gzero(s)<1$ we can assume that there is $\delta>0$ such that  $\Gsbeta(\beta_k)\le 1-\delta$ for all $k$. By
\eqref{g=1} we then obtain $1-\delta\ge (1-\beta_k(t))^2$ and therefore
$\beta_k(t)\ge \frac12\delta$ for all $k$ and all $t$. 

The sequence $1-\beta_k$ is bounded in $H^1(\R)$ and has therefore a subsequence converging weakly to some function $1-\beta_s$.
Since the second term in $\Gsbeta$ is convex, if we prove that $\int_{-\infty}^\infty f^{-2}(\beta_k)\de t\to \int_{-\infty}^\infty f^{-2}(\beta_s)\de t$ then 
$\beta_s$ is a minimizer of $\Gsbeta$, and by the formula stated above we can reconstruct $\alpha$.
To prove continuity we observe that (possibly after extracting a further subsequence) $\beta_k\to\beta_s$ pointwise almost everywhere, and that \eqref{assf2}
and  $\beta_k(t)\ge \frac12\delta$
imply
$f(\beta_k)=f_1(1-\beta_k)/(1-\beta_k)\ge f_1(1-\frac12\delta)/(1-\beta_k)$ and therefore
$0\le f^{-2}(\beta_k)\le C_s(1-\beta_k)^2$ pointwise. Since $\int_{-\infty}^\infty(1-\beta_k)^2\de t\le 4$
{and $1-\beta_k\to 1-\beta$ in $L^2(\R)$,}
 by Fatou's Lemma we obtain 
\begin{equation*}
{ \int_{-\infty}^\infty f^{-2}(\beta_s)\de t = \lim_{k\to\infty} 
\int_{-\infty}^\infty f^{-2}(\beta_k)\de t .}
\end{equation*}
Therefore $\beta_s$ is a minimizer of $\Gsbeta$. This concludes the proof of \ref{prop:g2bexist}.
	
Assume now that $s\in(0,\infty)$, $\Gzero(s)<1$ and {$s'\in[0,s)$.} Let $\beta_s$ be a minimizer of $\Gsbeta$. 
By the above estimates $\int_{-\infty}^\infty f^{-2}(\beta_s)\de t\le C_s \int_{-\infty}^\infty (1-\beta_s)^2\de t<\infty$ and therefore
$\Gzero(s')\le \Gspbeta(\beta_s)<\Gsbeta(\beta_s)=\Gzero(s)$.
This concludes the proof of \ref{prop:g2bmong}.

In order to prove \ref{prop:g2bmonoton}, it suffices to show that for any $m\in[0,1]$ the set $I_m:=\{t: \beta_s(t)\le m\}$ is an interval. This will imply the stated monotonicity properties, with $t_*$ any point in the intersection of all $I_m$ for $m>\inf\beta_s$. If $I_m$ were not an interval, there would be {$t_1<t_2\in I_m$} with $m=\beta_s(t_1)=\beta_s(t_2)$ and $\beta_s(t)>m$ for $t\in(t_1,t_2)$. We can construct an admissible competitor $(\alpha,\beta)\in\mathcal{U}_1$ by modifying $(\alpha_s,\beta_s)$ as follows:
\begin{equation*}
\begin{split}
\alpha(t)&:=
\begin{cases}
\alpha_s(t) & \text{if }t\leq t_1,\\
\alpha_s(t_1)+\frac{\alpha_s(t_2)-\alpha_s(t_1)}{\delta}(t-t_1)& \text{if }t_1<t<t_1+\delta,\\
\alpha_s(t+t_2-t_1-\delta) & \text{if }t\geq t_1+\delta,
\end{cases}\\
\beta(t)&:=
\begin{cases}
\beta_s(t) & \text{if }t\leq t_1,\\
m & \text{if }t_1<t<t_1+\delta,\\
\beta_s(t+t_2-t_1-\delta) & \text{if }t\geq t_1+\delta,
\end{cases}
\end{split}
\end{equation*}
where $\delta>0$ is to be chosen. Then, by using Young's inequality and by \eqref{assf2},
\begin{align*}
\mathcal{G}_s(\alpha_s,\beta_s) &- \mathcal{G}_s(\alpha,\beta) 
= \int_{t_1}^{t_2} \Bigl( s^2f^2(\beta_s)|\alpha_s'|^2 + \frac{(1-\beta_s)^2}{4} + |\beta_s'|^2 \Bigr) \de t \\
& \qquad\qquad\qquad\qquad\qquad - \int_{t_1}^{t_1+\delta} \Bigl( s^2f^2(\beta)|\alpha'|^2 + \frac{(1-\beta)^2}{4} + |\beta'|^2 \Bigr) \de t \\
& > \int_{t_1}^{t_2} sf(\beta_s)(1-\beta_s)|\alpha_s'|\de t - \frac{s^2f^2(m)|\alpha_s(t_2)-\alpha_s(t_1)|^2}{\delta} - \frac{(1-m)^2}{4}\delta \\
& \geq sf(m)(1-m)|\alpha_s(t_2)-\alpha_s(t_1)| - \frac{s^2f^2(m)|\alpha_s(t_2)-\alpha_s(t_1)|^2}{\delta} - \frac{(1-m)^2}{4}\delta \,.
\end{align*}
By optimizing in $\delta$ we see that the right-hand side in the previous inequality is zero, and this contradicts the minimality of $(\alpha_s,\beta_s)$.
This concludes the proof of \ref{prop:g2bmonoton}.

Let us prove \ref{prop:g2bC1}. Let $\Gzero(s)<1$ and choose an optimal pair $(\alpha_s,\beta_s)\in\mathcal{U}_1$ for $\Gzero(s)$. 
By \ref{prop:g2bming} and \eqref{assf1} we have $f(\beta_s)\geq f({1-\sqrt{\Gzero(s)}})>0$ on $\R$. 
Let $T_-:=\inf\{t: \beta_s(t)<1\}\in\R\cup\{-\infty\}$, $T_+:=\sup\{t: \beta_s(t)<1\}\in\R\cup\{\infty\}$. By \ref{prop:g2bmonoton} we have $\beta_s<1$ in $(T_-,T_+)$, by finiteness of the integral we have $\alpha_s'=0$ in $\R\setminus (T_-,T_+)$.
By taking variations in the variable $\alpha$ in the minimum problem \eqref{defgnew} we find,
{as in the proof of \ref{prop:g2bexist}},
\begin{equation*}
f^2(\beta_s)\alpha_s'=\mathrm{const.} \qquad\text{almost everywhere in }{(T_-,T_+)}.
\end{equation*}
Therefore $\alpha_s'>0$ in ${(T_-,T_+)}$ 
{with $\lim_{t\uparrow {T_+}}=\lim_{t\downarrow {T_-}}\alpha_s'(t)=0$,}
and in particular $\alpha_s:{(T_-,T_+)}\to(\alpha_s(T_-),\alpha_s(T_+))$ is a $C^1$, bijective map
(with $\alpha_s({T_\pm})$ interpreted as the limit of $\alpha_s(t)$ for $t\to{\pm}\infty$ if ${T_\pm=\pm\infty}$). By taking variations in the variable $\beta$ we obtain instead
\begin{equation*}
\int_{-\infty}^\infty \Bigl( s^2f(\beta_s)f'(\beta_s)|\alpha_s'|^2\vphi - \frac{(1-\beta_s)\vphi}{4} + \beta_s'\vphi' \Bigr)\de t = 0 \qquad\text{for all }\vphi\in C^\infty_{\rm c}{(T_-,T_+)},
\end{equation*}
from which it follows by standard arguments that $\beta_s\in C^1{(T_-,T_+)}$. {Hence \ref{prop:g2bC1} is proved.}

Finally, taking internal variations (in the sense of considering competitors of the type $(\alpha_s(t+\e \varphi(t)), \beta_s(t+\e \varphi(t))$, for $\varphi\in C^1_c(\R)$) one also obtains the usual equipartition result, in the sense that the minimizer fulfills \eqref{eqequipartab}. This proves \ref{prop:equip}.
\end{proof}

{ We next list} in the following proposition some basic properties of the function $\Gzero$, already observed in \cite[Proposition~4.1]{CFI} for the exception of \ref{item5g}, which guarantees that the function $\Gzero$ is strictly below the linear function $\ell s$: for the proof of this last property, we require the additional assumption \eqref{assf3} (which is not needed for the $\Gamma$-convergence result in \cite{CFI}). Notice that this condition, which is used in the proof of the strict subadditivity of $\Gzero$ (Proposition~\ref{prop:g3}), will be significantly improved in Proposition~\ref{prop:gexpansion}.

\begin{proposition} \label{prop:g1}
	The function $\Gzero$ defined in \eqref{defgnew} enjoys the following properties:
	\begin{enumerate}
		\item\label{item1g} $\Gzero(0)=0,$ and $\Gzero$ is subadditive, \emph{i.e.\ }$\Gzero(s_1+s_2)\leq \Gzero(s_1)+\Gzero(s_2)$ for every $s_1,s_2\in\R^+$;
		\item\label{item2g} $\Gzero$ is nondecreasing,
		$\Gzero(s) \leq 1\wedge \ell s$,
		and $\Gzero$ is Lipschitz continuous with Lipschitz constant $\ell$; 
		\item\label{item3g} $\lim_{s\to\infty}\Gzero(s)= 1$;
		\item\label{item4g} $\lim_{s\to0^+}\frac{\Gzero(s)}{s}= \ell$;
		\item\label{item5g} $\Gzero(s)< \ell s$ for all $s>0$.
	\end{enumerate}
\end{proposition}

\begin{proof}
Only the statement \ref{item5g} requires a new proof, the others {follow from} \cite[Proposition~4.1]{CFI} {using Proposition~\ref{prop:g2b}\ref{prop:g2bcar}.}
With fixed $s>0$, let $\sigma\in(0,1)$ and set $\alpha(t):=0$ in $[0,\frac13]$, $\alpha(t):={1}$ in $[\frac23,1]$, and the linear interpolation between these two values in $[\frac13,\frac23]$; set also $\beta(t):=\sigma$ in $[\frac13,\frac23]$, and the linear interpolation between the values $\sigma$ and 1 in each of the two intervals $[0,\frac13]$ and $[\frac23,1]$. By using the pair $(\alpha,\beta)$ as a competitor in the { characterization \eqref{eqcar2}} of $\Gzero$ we find
\begin{equation}\label{proofitem5g}
\begin{split}
\Gzero(s)
& \leq\inf_{\sigma\in(0,1)} \bigl[ s(1-\sigma)f(\sigma)+(1-\sigma)^2 \bigr] \\
& = \inf_{\sigma\in(0,1)} \bigl[ s\ell -s\ell_1(1-\sigma)+o(1-\sigma) + (1-\sigma)^2 \bigr]  \,,
\end{split}
\end{equation}
where we used the assumption \eqref{assf3}. The strict inequality $\Gzero(s)<\ell s$ follows from \eqref{proofitem5g}, by choosing a $\sigma$ sufficiently close to 1.
\end{proof}

{ In the following proposition we prove the representation formula \eqref{eqcar3} for $\Gzero$, which removes the invariance under reparametrization in \eqref{eqcar2} and has a unique minimizer. We recall the definition \eqref{eqdefsfrac} of the threshold $\sfrac$.}

\begin{proposition}\label{prop:g2}
	The following properties hold:
	\begin{enumerate}
		\item \label{prop:g2gamma} { The characterization \eqref{eqcar3} of $\Gzero$ holds.}
		\item \label{prop:g2unique} For $s\in(0,\sfrac)$ the variational problem \eqref{eqcar3} has a unique minimizer $\gamma_s$. It obeys $\gamma_s\in C^1(0,1)$, $\gamma_s(\frac12)=\max \gamma_s=(1-\min\beta_s)^2$, where $\beta_s$ is any minimizer from Proposition~\ref{prop:g2b}\ref{prop:g2bexist}.
		\item\label{prop:g2monoton} If $0<s_1<s_2<\sfrac$ then $0< \gamma_{s_1}(t)<\gamma_{s_2}(t)$ for all $t\in(0,1)$.
		\item\label{prop:g2ms} The function
		\begin{equation}
		m_s:=
		\begin{cases}
		1-\sqrt{\gamma_s(\frac12)},& \text{ if } s\in[0,\sfrac),\\
		0, &\text{ if } s>\sfrac,
		\end{cases}
		\end{equation}
		is continuous and strictly decreasing in $[0,\sfrac)$.
		\item\label{prop:uniqbeta} {For $s<\sfrac$, the minimizer of \eqref{defgnew} is unique up to translations, in the sense that if $(\alpha_s,\beta_s)$ and $(\hat\alpha_s,\hat\beta_s)$ are minimizers then there are $a_1, t_1\in\R$ such that $\alpha_s(t)=a_1+\hat\alpha_s(t-t_1)$, $\beta_s(t)=\hat\beta_s(t-t_1)$.}
	\end{enumerate}
\end{proposition}

\begin{proof}
	\ref{prop:g2gamma}: The identity { \eqref{eqcar3} follows from \eqref{eqcar2}} by reparametrization: indeed, first observe that by a density argument the infimum in \eqref{eqcar2} can be written as
	\begin{equation}  \label{proofg2}
	\begin{split}
	\inf \biggl\{ &\int_0^1 \sqrt{s^2(1-\beta)^2f^2(\beta)|\alpha'|^2 +\frac14\Big|\frac{\de}{\de t}(1-\beta)^2\Big|^2} \de t \: : \: (1-\beta)^2\in C^1_{\mathrm c}(0,1),\, 0\leq\beta\leq 1, \\
	& \qquad\qquad\qquad\qquad\qquad \alpha\in C^{1}([0,1]),\, \alpha(0)=0,\, \alpha(1)=1,\, \alpha'>0 \biggr\}.
	\end{split}
	\end{equation}
	Similarly, the infimum in \eqref{eqcar3} can be considered on the class of $\gamma\in C^1_{\mathrm c}(0,1)$ with $0\leq\gamma\leq1$.
	Then, for $(\alpha,\beta)$ admissible in \eqref{proofg2}, the reparametrization
	\begin{equation}\label{eqdefgammafromalphabeta}
	\gamma(t):=(1-\beta)^2\circ\alpha^{-1}(t)
	\end{equation}
	gives an admissible profile for \eqref{eqcar3} with
	\begin{equation} \label{proofg2b}
	\int_0^1 \sqrt{s^2(1-\beta)^2f^2(\beta)|\alpha'|^2 +\frac14\Big|\frac{\de}{\de t}(1-\beta)^2\Big|^2} {\de t}= \int_0^1 \sqrt{s^2(\tif(\sqrt{\gamma}))^2 + \textstyle\frac{|\gamma'|^2}{4}} \de t .
	\end{equation}
	Conversely, given $\gamma\in C^1_{\mathrm c}(0,1)$ with $0\leq\gamma\leq1$, we consider the pair $(\alpha,\beta)$ defined by $\alpha(t):=t$, $\beta(t):=1-\sqrt{\gamma({t})}$, which is admissible in \eqref{proofg2} and satisfies \eqref{proofg2b}. This proves \ref{prop:g2gamma}.
	
	\ref{prop:g2unique}: Assume now $\Gzero(s)<1$ and let $(\alpha_s,\beta_s)\in \mathcal U_1$ be a minimizer for \eqref{defgnew} with $\alpha_s(-\infty)=0$, which exists by Proposition~\ref{prop:g2b}\ref{prop:g2bexist}.
	We consider the rescaled profile $\gamma_s:[0,1]\to[0,1]$, $\gamma_s(t):=( 1-\beta_s\circ\alpha_s^{-1}(t))^2$: we have $\gamma_s\in C^1(0,1)$, $0\leq\gamma_s\leq1$, $\gamma_s(0)=\gamma_s(1)=0$, and, {with $T_-$, $T_+$ as in Proposition~\ref{prop:g2b}\ref{prop:g2bC1},}
	\begin{align*}
	\Gzero(s) = \mathcal{G}_s(\alpha_s,\beta_s)
	& \geq \int_{T_-}^{T_+} |1-\beta_s|\sqrt{s^2f^2(\beta_s)|\alpha_s'|^2 + |\beta_s'|^2}\de t 
	= \int_0^1 \sqrt{s^2\bigl( \tif(\sqrt{\gamma_s}) \bigr)^2 + \textstyle\frac{|\gamma_s'|^2}{4}}\de t ,
	\end{align*}
	which shows that $\gamma_s$ is a minimizer in problem \eqref{eqcar3}. We have $1-\min \beta_s=\max\sqrt{\gamma_s}$ and by symmetrization the maximum is attained at the point $t=\frac{1}{2}$. Thanks to the convexity assumption \eqref{assf2}, the minimizer $\gamma_s$ of \eqref{eqcar3} is unique.
	
	\ref{prop:g2monoton} and \ref{prop:g2ms}:
	We will exploit further properties of the optimal profile $\gamma_s$ {for $s<\sfrac$}: by computing the Euler-Lagrange equation satisfied by a $C^1$-minimizer of problem \eqref{eqcar3} we find
	\begin{equation*}
	\Biggl(\frac{\gamma_s'}{2\sqrt{s^2\bigl( \tif(\sqrt{\gamma_s}) \bigr)^2 + \frac14|\gamma_s'|^2}}\Biggr)'
	=\frac{s^2\tif(\sqrt{\gamma_s})\tif'(\sqrt{\gamma_s})}{\sqrt{\gamma_s}\sqrt{s^2\bigl( \tif(\sqrt{\gamma_s}) \bigr)^2 + \frac14|\gamma_s'|^2}}
	\qquad \text{in }(0,1),
	\end{equation*}
	which yields the $C^2$-regularity of $\gamma_s$ in $(0,1)$ and, after elementary computations,
	\begin{equation}\label{ELEgamma}
	\gamma_s'' = \frac{ \tif(\sqrt{\gamma_s})\tif'(\sqrt{\gamma_s})}{\sqrt{\gamma_s}} \Biggl( 2s^2 + \frac{|\gamma_s'|^2}{\bigl(\tif(\sqrt{\gamma_s})\bigr)^2} \Biggr) \qquad\text{in }(0,1).
	\end{equation}
	Furthermore, the equation \eqref{ELEgamma} has a first integral:
	\begin{equation*}
	(\tif(\sqrt{\gamma_s}))^2 = c\sqrt{s^2(\tif(\sqrt{\gamma_s}))^2 + \frac14|\gamma_s'|^2}
	\qquad\text{in }(0,1)
	\end{equation*}
	for a constant $c>0$, which can be computed by imposing that $\gamma_s$ has a maximum at the point $t=\frac{1}{2}$ with value $(1-m_s)^2$: this gives $c=\frac1s\tif(1-m_s)$ and in turn
	\begin{equation} \label{ELEgamma3}
	\gamma_s'= \pm 2s\tif(\sqrt{\gamma_s}) \sqrt{ \Bigl( \frac{\tif(\sqrt{\gamma_s})}{\tif(1-m_s)} \Bigr)^2 -1 }
	\qquad\text{in }(0,1)
	\end{equation}
	(with positive sign in $(0,\frac12)$, and negative sign in $(\frac12,1)$).
	
	By using the previous equations we can now show the continuity of the map $s\mapsto m_s$. Indeed, consider a sequence $s_n\to s\in[0,s_{\rm frac})$ and the corresponding optimal profiles $\gamma_{s_n}$. In view of \eqref{ELEgamma3}, the first derivatives $\gamma_{s_n}'$ are uniformly bounded in $L^\infty(0,1)$, hence possibly extracting a subsequence we have $\gamma_{s_n}\wto \gamma$ weakly* in $W^{1,\infty}(0,1)$ (and uniformly), for some $\gamma$ vanishing at the boundary. By continuity of $\Gzero$ we find that $\gamma$ is the unique minimizer in problem \eqref{eqcar3} for $\Gzero(s)$; in particular $(1-m_s)^2=\gamma(\frac12)=\lim_n\gamma_{s_n}(\frac12)=\lim_n(1-m_{s_n})^2$, which proves the continuity of $m_s$.
	
	We then show that the map $s\mapsto m_s$ is injective (then also the strict monotonicity follows). Assume by contradiction that for two different values $s_1<s_2$ one has $m_{s_1}=m_{s_2}$, and let $\gamma:=\gamma_{s_2}-\gamma_{s_1}$ be the difference of the corresponding optimal profiles. By \eqref{ELEgamma3} we have $\gamma'(0)>0$, and hence $\gamma(t)>0$ for all sufficiently small $t$. On the other hand, at the point $t=\frac12$ by assumption $\gamma(\frac12)=\gamma_{s_1}'(\frac12)=\gamma_{s_2}'(\frac12)=0$, therefore we find using \eqref{ELEgamma}
	\begin{equation*}
	\gamma''({\textstyle\frac12})
	= 2(s_2^2-s_1^2)\frac{ \tif(1-m_{s_1})\tif'(1-m_{s_1})}{1-m_{s_1}} < 0
	\end{equation*}
	(by the assumption \eqref{assf2}). We conclude that $\gamma(t)<0$ for $t$ sufficiently close to $\frac12$. {Let then $\bar{t}\in(0,\frac12)$ be the smallest value with $\gamma(\bar{t})=0$,} that is $\gamma_{s_1}(\bar{t})=\gamma_{s_2}(\bar{t})=:\sigma$. By \eqref{ELEgamma3}
	\begin{equation*}
	\gamma'(\bar{t}) = 2(s_2-s_1)f_1(\sqrt{\sigma}) \sqrt{ \Bigl( \frac{\tif(\sqrt{\sigma})}{\tif(1-m_{s_1})} \Bigr)^2 -1 } >0,
	\end{equation*}
	which is a contradiction.
	
	At this point we conclude the proof of \ref{prop:g2monoton}. Fix two values $s_1<s_2<\sfrac$ and let $\gamma:=\gamma_{s_2}-\gamma_{s_1}$ be the difference of the corresponding optimal profiles. Then $m_{s_2}<m_{s_1}$, hence $f_1(1-m_{s_2})<f_1(1-m_{s_1})$ and by \eqref{ELEgamma3} we have $\gamma'(0)>0$, and hence $\gamma(t)>0$ for all sufficiently small $t$. If $\gamma(t)>0$ for all $t\in(0,1)$, we are done. Otherwise, let $\bar t\in(0,1)$ be the smallest value with $\gamma(\bar t)=0$. Then \eqref{ELEgamma3} implies $\gamma'(\bar t)>0$, which is a contradiction.
	
	\ref{prop:uniqbeta}:
	Let $(\alpha_s,\beta_s)$ and $(\hat\alpha_s,\hat\beta_s)\in\mathcal U_1$ be minimizers for \eqref{defgnew}. Without loss of generality we can assume $\alpha_s(-\infty)=\hat\alpha_s(-\infty)=0$. By Proposition~\ref{prop:g2b}\ref{prop:g2bC1}, $\alpha_s$ is a $C^1$ bijection from $(T_-,T_+)$ onto $(0,1)$; $\alpha_s$ and $\beta_s$ are constant on $(-\infty,T_-)$ and $(T_+,\infty)$ (if these intervals are nonempty). We define $\gamma\in C^1(0,1)$ by $\gamma:=(1-\beta_s)^2\circ\alpha_s^{-1}$, so that $\gamma(t)\to0$ as $t\to0$ and $t\to1$. Then $\gamma$ is a competitor in \eqref{eqcar3}.
	Using first equipartition, that was proven in \eqref{eqequipartab}, and then the change of variables $r:=\alpha_s(t)$, we have
		\begin{equation*}
		\begin{split}
		\int_{-\infty}^\infty \biggl( s^2f^2(\beta_s)|\alpha_s'|^2 + \frac{(1-\beta_s)^2}{4} + |\beta_s'|^2 \biggr) \de t
		=& \int_{-\infty}^\infty \sqrt{(1-\beta_s)^2 s^2f^2(\beta_s)|\alpha_s'|^2 + (1-\beta_s)^2|\beta_s'|^2 } \de t\\
		=& \int_{T_-}^{T_+} \sqrt{ s^2f_1^2(\sqrt{\gamma}\circ\alpha_s)|\alpha_s'|^2 + \frac14 |\gamma'|^2\circ\alpha_s |\alpha_s'|^2 } \de t\\
		=& \int_{0}^1 \sqrt{ s^2f_1^2(\sqrt{\gamma}) + \frac14 |\gamma'|^2 } \de r,
		\end{split}
		\end{equation*}
	so that $\gamma$ is a minimizer of \eqref{eqcar3}.
	By uniqueness, the pair $(\hat\alpha_s,\hat\beta_s)$ has to produce the same $\gamma$, so that
		$\beta_s\circ\alpha_s^{-1}=\hat\beta_s\circ\hat\alpha_s^{-1}$.
		The function $\varphi:=\hat\alpha_s^{-1}\circ\alpha_s$ is a $C^1$ bijection from $( T_-,  T_+)$ onto $(\hat T_-,\hat T_+)$. Then
		$\beta_s=\hat\beta_s\circ\varphi$, $\alpha_s=\hat\alpha_s\circ\varphi$.
		Recalling the equipartition condition \eqref{eqequipartab},
		\begin{equation*}
		\begin{split}
		0&=s^2f^2(\beta_s) |\alpha_s'|^2+|\beta_s'|^2-\frac{(1-\beta_s)^2}{4}  \\
		&=\big[s^2f^2(\hat\beta_s) |\hat\alpha_s'|^2 +|\hat\beta_s'|^2\big] \circ\varphi |\varphi'|^2-\frac{(1-\hat\beta_s)^2}{4}  \circ\varphi =\frac{(1-\hat\beta_s)^2\circ\varphi }{4} \big[|\varphi'|^2-1\big]
		\end{split}
		\end{equation*}
		where in the last step we used the equipartition condition \eqref{eqequipartab} for $(\hat\alpha_s,\hat \beta_s)$. We conclude that $|\varphi'|=1$ wherever ${\hat\beta_s(\varphi)<1}$, which implies $\varphi(t)=t+t_1$ for some $t_1\in\R$ and concludes the proof.
\end{proof}

We remark that $(\alpha,\beta)$ minimizing \eqref{eqcar2} is not unique, as the variational problem \eqref{eqcar2}  is invariant under reparametrization. However, if one fixes $\alpha$, for example by setting $\alpha(t)=t$, then for $s<\sfrac$ one easily obtains uniqueness of the corresponding $\beta$. Indeed, any minimizer $\beta$ produces an optimal $\gamma$ via \eqref{eqdefgammafromalphabeta}. Analogously, once $\alpha$ is fixed, monotonicity of $s\mapsto \gamma_s(t)$ implies monotonicity of {$\beta$ with respect to $s$}.

\begin{proposition}[Strict subadditivity] \label{prop:g3}
	For every $s>0$ and $t\in(0,1)$ the function $\Gzero$ satisfies the inequality $\Gzero(ts)>t\Gzero(s)$. In particular $\Gzero$ is strictly subadditive: for every $s_1,s_2>0$
	\begin{equation*}
	\Gzero(s_1+s_2) < \Gzero(s_1) + \Gzero(s_2)\,.
	\end{equation*}
\end{proposition}

\begin{proof}
	{Let $s>0$, $t\in(0,1)$, $\bar s=ts$.
	{If $\Gzero(\bar s)=1$, then {$t\Gzero(s)< 1 =\Gzero(ts)$} and the first assertion is proven. Otherwise,}
	by Proposition~\ref{prop:g2b}\ref{prop:g2bexist} there is
		$(\alpha_{\bar s},\beta_{\bar s})\in \mathcal U_1$ such that
		$\mathcal G_{\bar s}(\alpha_{\bar s},\beta_{\bar s})=\Gzero(\bar s)$.
		Then by rescaling $\tilde{\alpha}(y):=\alpha_{\bar s}(ty)$, $\tilde{\beta}(y):=\beta_{\bar s}(ty)$ we have
		$(\tilde{\alpha},\tilde{\beta})\in\mathcal{U}_1$ and therefore
		\begin{align*}
		\Gzero({\bar s}) &= \int_{-\infty}^\infty \Bigl( t^2s^2 f^2(\beta_{\bar s})|\alpha'_{\bar s}|^2 + \frac{(1-\beta_{\bar s})^2}{4} + |\beta_{\bar s}'|^2 \Bigr)\de x \\
		& = \int_{-\infty}^\infty \Bigl( s^2 f^2(\tilde\beta)|\tilde\alpha'|^2 + \frac{(1-\tilde\beta)^2}{4} + \frac{|\tilde\beta'|^2}{t^2} \Bigr) t\de y \\
		& = t\gs(\tilde{\alpha},\tilde{\beta}) + \Bigl(\frac{1}{t}-t\Bigr)\int_{-\infty}^\infty |\tilde{\beta}'|^2 \de y  \\
		& \geq t\Gzero(s) + \Bigl(\frac{1}{t}-t\Bigr)t
		\int_{-\infty}^\infty |{\beta_{\bar s}}'|^2 \de y  
		>t\Gzero(s)
		\,,
		\end{align*}
		where the last inequality follows by $\beta_{\bar s}\ne 0$}.  This proves the first part of the statement. By writing
	\begin{align*}
	\Gzero(s_1) + \Gzero(s_2) = \Gzero\Bigl((s_1+s_2)\frac{s_1}{s_1+s_2}\Bigr) + \Gzero\Bigl((s_1+s_2)\frac{s_2}{s_1+s_2}\Bigr) > \Gzero(s_1+s_2)
	\end{align*}
	the strict subadditivity of $\Gzero$ follows.
\end{proof}


In the following proposition we perform a careful asymptotic analysis of the function $\Gzero(s)$ as $s\to0^+$. Notice that, in view of Proposition~\ref{prop:g1}, $\Gzero$ is differentiable at the origin, with positive and finite slope; we now compute the second order correction by means of a $\Gamma$-convergence expansion in the minimum problems \eqref{eqcar3}.

\begin{proposition} \label{prop:gexpansion}
	There exists $\tilde{\ell}>0$ such that
	\begin{equation} \label{gexpansion}
	\Gzero(s) = \ell s - \tilde{\ell}s^{5/3} + o(s^{5/3}) \qquad\text{as }s\to0^+.
	\end{equation}
	Moreover, if $\gamma_s$ is a minimizer in problem \eqref{eqcar3}, then the sequence $s^{-4/3}\gamma_s$ converges uniformly in $[0,1]$ {as $s\to0^+$} to the unique solution $\bar{\eta}\in H^1_0(0,1)$ of the minimum problem
	\begin{equation} \label{gexpansion4}
	\min\biggl\{  H(\eta) : \eta\in H^1_0(0,1),\, \eta\geq0 \biggr\} 
	\end{equation}
	where
	\begin{equation} \label{gexp3}
	H(\eta):=\int_0^1\Big(-\frac{\ell_1}{\ell}\sqrt{\eta} + \frac{|\eta'|^2}{8\ell^2} \Big)\de t\,.
	\end{equation}
	In particular
	\begin{equation} \label{gexpansion3}
	(1-m_s)^2 = \max_{t\in[0,1]}\gamma_s(t) \leq cs^{4/3}
	\end{equation}
	for all $s>0$ small enough and for a uniform constant $c>0$.
\end{proposition}

We remark that the minimum in (\ref{gexpansion4}) can be found explicitly, and leads to
\begin{equation}\label{eqtildeell}
	\tilde \ell=
	{-\ell \min H({\tilde\eta})=}
	\ell_1^{4/3}\ell^{1/3} \frac3{10}\Bigl(\frac32\Bigr)^{2/3}\simeq 0.4 \ell_1^{4/3}\ell^{1/3}.
\end{equation}
The computation is straightforward but somewhat cumbersome, we do not report it for brevity since this result is not needed for what follows.

\begin{proof}
	To guess the order of the second term in the expansion of $\Gzero$ as $s\to0$, {we} perform a Taylor expansion in the energy of a minimizer $\gamma_s$ for \eqref{eqcar3}. Let $\gamma_s$ be an optimal profile for problem \eqref{eqcar3}, as in Proposition~\ref{prop:g2}\ref{prop:g2bexist}. Using $\max\gamma_s=\gamma_s(\frac12)=(1-m_s)^2$ and Proposition~\ref{prop:g1}\ref{item2g} gives
	\begin{equation*}
	\ell s \geq \Gzero(s) \geq \frac12\int_0^1 |\gamma_s'|\de t = (1-m_s)^2,
	\end{equation*}
	and therefore
	\begin{equation} \label{gexp-1}
	0\leq \gamma_s(t) \leq \ell s \qquad\text{for all }t\in[0,1].
	\end{equation}
	Furthermore, by the Euler-Lagrange equation \eqref{ELEgamma3} satisfied by $\gamma_s$ we further deduce {for $s$ small}, using \eqref{gexp-1} and \eqref{assf3}, the bound on the derivative
	\begin{equation} \label{gexp0}
	\begin{split}
	|\gamma_s'(t)|
	& = 2s\tif(\sqrt{\gamma_s}) \sqrt{ \Bigl( \frac{\tif(\sqrt{\gamma_s})}{\tif(1-m_s)} \Bigr)^2 -1 } \\
	& \leq 2 s \frac{\tif(\sqrt{\gamma_s})}{\tif(1-m_s)} \sqrt{ 2\ell\ell_1\bigl(1-m_s-\sqrt{\gamma_s}\bigr) +o(1-m_s) }
	\leq c_0s^{\frac54}
	\end{split}
	\end{equation}
	for all $t\in[0,1]$, for some positive constant $c_0$, independent of $s$.
	Since $\gamma_s$ converges to zero uniformly as $s\to0$, we can formally expand the energy of $\gamma_s$ by using \eqref{assf3} and neglecting higher order terms:
	\begin{equation*}
	\begin{split}
	\Gzero(s) = \int_0^1\sqrt{s^2\bigl(\tif(\sqrt{\gamma_s})\bigr)^2 + \textstyle\frac{|\gamma_s'|^2}{4}} \de t
	& \sim \int_0^1 \sqrt{s^2\ell^2 -2s^2\ell\ell_1\sqrt{\gamma_s} + \textstyle\frac{|\gamma_s'|^2}{4} } \de t \\
	& \sim s\ell + s\ell \int_0^1\Bigl( -\frac{\ell_1}{\ell}\sqrt{\gamma_s} + \frac{|\gamma_s'|^2}{8s^2\ell^2} \Bigr).
	\end{split}
	\end{equation*}
	For $s$ small, the two terms in the last integral are of the same order if $\gamma_s\sim s^{4/3}$, so that, in turn, we might expect $\Gzero(s) - \ell s \sim s^{5/3}$.
	
	The previous formal considerations can be made rigorous by means of an asymptotic development by $\Gamma$-convergence. The expected decay of minimizers of \eqref{eqcar3} suggests to consider the rescaling $\eta(t):=s^{-\frac43}\gamma(t)$: with this position the minimum problem \eqref{eqcar3} can be reformulated as
	\begin{equation} \label{gexp1}
	\Gzero(s) = \inf \biggl\{\int_0^1 \sqrt{s^2 \bigl( \tif(s^{2/3}\eta^{1/2}) \bigr)^2 + \textstyle\frac{s^{8/3}|\eta'|^2}{4}} \: : \: \eta\in W^{1,1}_0(0,1),\, 0\leq\eta\leq s^{-4/3} \biggr\}.
	\end{equation}
	We compute the $\Gamma$-limit as $s\to0$, with respect to the strong topology of $L^{1}(0,1)$, of the family of functionals defined by
	\begin{equation} \label{gexp2}
	H_s(\eta) :=
	\frac{1}{s^{2/3}}\biggl(\frac{1}{s\ell}\int_0^1 \sqrt{ s^2 \bigl(\tif(s^{2/3}\eta^{1/2})\bigr)^2 + \frac{s^{8/3}|\eta'|^2}{4}}\de t - 1 \biggr)
	\end{equation}
	if $\eta\in W^{1,1}_0(0,1)$, $0\leq\eta\leq \ell s^{-1/3}$, and $|\eta'|\leq c_0s^{-1/12}$, and $H_s(\eta):=\infty$ otherwise in $W^{1,1}_0(0,1)$. Notice that the previous bounds are satisfied by the (rescaled) minimizer $\gamma_s$, in view of \eqref{gexp-1}--\eqref{gexp0}. We will prove that the $\Gamma$-limit is given by the functional
	$H$ defined as the expression in \eqref{gexp3}
	for $\eta\in H^{1}_0(0,1)$, $\eta\geq0$, and $H(\eta):=\infty$ otherwise in $W^{1,1}_0(0,1)$.
	
	We first prove the $\Gamma$-liminf inequality, together with the compactness of sequences with bounded energy with respect to the weak topology of $H^{1}(0,1)$. Without loss of generality, we consider a subsequence $\eta_s\in W^{1,1}_0(0,1)$, not relabeled, such that $0\leq\eta_s\leq \ell s^{-1/3}$, $|\eta_s'|\leq c_0s^{-1/12}$, and $\lim_{s\to0} H_s(\eta_s)<\infty$. Since $s^{2/3}\eta_s^{1/2}\leq \ell^{1/2}s^{1/2}$, by \eqref{assf3} we have
	\begin{equation*}
	\tif(s^{2/3}\eta_s^{1/2}) = \ell-\ell_1s^{2/3}\eta_s^{1/2} + o({s^{2/3}\eta_s^{1/2}}),
	\end{equation*}
	and therefore
	\begin{equation*}
	\begin{split}
	H_s(\eta_s)
	& =\frac{1}{s^{2/3}}\biggl(\frac{1}{s\ell}\int_0^1 \sqrt{ s^2\bigl( \ell-\ell_1s^{2/3}\eta_s^{1/2}+o(s^{2/3}\eta_s^{1/2}) \bigr)^2 +\frac{s^{8/3}|\eta_s'|^2}{4}}\de t - 1 \biggr)\\
	& =\frac{1}{s^{2/3}}\biggl(\int_0^1 \sqrt{ 1-2\frac{\ell_1}{\ell}s^{2/3}\eta_s^{1/2}+o(s^{2/3}\eta_s^{1/2}) + \frac{s^{2/3}|\eta_s'|^2}{4\ell^2} }\de t - 1 \biggr).
	\end{split}
	\end{equation*}
	Now, since
	\begin{equation*}
	\bigg| -2\frac{\ell_1}{\ell}s^{2/3}\eta_s^{1/2}+o(s^{2/3}\eta_s^{1/2}) + \frac{s^{2/3}|\eta_s'|^2}{4\ell^2} \bigg| \leq cs^{1/2},
	\end{equation*}
	a Taylor's expansion gives
	\begin{equation}\label{gexp4}
	\begin{split}
	H_s(\eta_s)
	& = \int_0^1\biggl(-\frac{\ell_1}{\ell}\eta_s^{1/2} + \frac{|\eta_s'|^2}{8\ell^2} \biggr)\de t + \frac{1}{s^{2/3}}\int_0^1o\bigl(s^{2/3}\eta_s^{1/2}+s^{2/3}|\eta_s'|^2\bigr)\de t.
	\end{split}
	\end{equation}
	Notice that, for $s$ small enough, the second term in the right-hand side of \eqref{gexp4} can be controlled by the first one:
	\begin{equation*}
	H_s(\eta_s) \geq \int_0^1\biggl(-\frac{2\ell_1}{\ell}\eta_s^{1/2} + \frac{|\eta_s'|^2}{16\ell^2} \biggr)\de t.
	\end{equation*}
	Now, for any $\e>0$ we have $\eta_s^{1/2}\leq \e\eta_s^2 + c_\e$ by Young's inequality; therefore, using also Poincar\'e's inequality we end up with
	\begin{equation*}
	H_s(\eta_s) \geq - \frac{2\e\ell_1}{\ell}\int_0^1\eta_s^{2}\de t - \frac{2c_\e\ell_1}{\ell} + \frac{1}{16\ell^2}\int_0^1|\eta_s'|^2\de t \geq \Bigl(\frac{1}{16\ell^2}- \frac{2\e\ell_1}{\ell}\Bigr)\int_0^1 |\eta_s'|^2\de t - \frac{2c_\e\ell_1}{\ell},
	\end{equation*}
	and choosing $\e$ small enough we deduce that the sequence $(\eta_s)_s$ is uniformly bounded in $H^1_0(0,1)$. Up to subsequences, $\eta_s\wto\eta$ weakly in $H^1(0,1)$ (and uniformly) for some $\eta\in H^1_0(0,1)$, $\eta\geq0$. By \eqref{gexp4} we also deduce that
	\begin{equation*}
	\liminf_{s\to0} H_s(\eta_s) \geq \int_0^1 \Bigl(-\frac{\ell_1}{\ell}\eta^{1/2}+\frac{|\eta'|^2}{8\ell^2}\Bigr)\de t = H(\eta).
	\end{equation*}
	
	As for the $\Gamma$-limsup inequality, fixed $\eta\in C^1_c(0,1)$, $\eta\geq0$, we simply define $\eta_s:=\eta$. Hence, \eqref{gexp4} gives for $s$ small enough
	\begin{equation*}
	H_s(\eta_s)=\int_0^1\biggl(-\frac{\ell_1}{\ell}\eta^{1/2}+\frac{|\eta'|^2}{8\ell^2}\biggr)\de t +o(1) = H(\eta)+o(1),
	\end{equation*}
	where $o(1)$ depends on $\eta$ and tends to 0 as $s\to0$. The $\Gamma$-limsup inequality follows by density of $C^1_c(0,1)$ in $H^1_0(0,1)$.
	
	The existence of a minimizer $\bar{\eta}\in H^{1}_0(0,1)$ for the limit functional $H$ follows by a standard application of the direct method of the Calculus of Variations, and its uniqueness by strict convexity of the functional $H$.
	We can now check that the value of the minimum of $H$ is strictly negative: by considering {$\lambda\sin(\pi t)$} as a competitor, with $\lambda>0$, we have
	\begin{equation*}
	\min_{\eta\in W^{1,1}_0(0,1)} H(\eta)\leq 
	-\frac{\ell_1}{\ell} \lambda^{1/2}{\int_0^1
		\sin^{1/2}(\pi t)\de t} + \frac{\lambda^2}{8\ell^2}{\int_0^1
		\pi^2\cos^{2}(\pi t)\de t},
	\end{equation*}
	which is strictly negative {if} $\lambda$ is small enough. We denote $\tilde{\ell}:=-\ell\min H >0$. By \eqref{gexp1} and standard properties of $\Gamma$-convergence we obtain the convergence of the minimizers of $H_s$ to $\bar{\eta}$ and
	\begin{equation*}
	\frac{\Gzero(s)-\ell s}{\ell s^{5/3}} = \min_{\eta\in W^{1,1}_0(0,1)} H_s(\eta) \to \min_{\eta\in W^{1,1}_0(0,1)} H(\eta) = -\frac{\tilde{\ell}}{\ell} \qquad\text{as }s\to0^+,
	\end{equation*}
	so that \eqref{gexpansion} {and \eqref{gexpansion3}} follow.
\end{proof}

{ We are now in position to give the proofs of the results in Section~\ref{subsect:propg}, which essentially collect the main properties of $\Gzero$ from the previous statements.
	
\begin{proof}[Proof of Theorem~\ref{thm:propg1}]
	The result follows by combining Proposition~\ref{prop:g1}, Proposition~\ref{prop:g3}, and Proposition~\ref{prop:gexpansion}.
\end{proof}
	
\begin{proof}[Proof of Theorem~\ref{thm:g1}]
	The first assertion follows from Proposition~\ref{prop:g2b}\ref{prop:g2bexist} and Proposition~\ref{prop:g2}\ref{prop:uniqbeta}.
	The second follows from monotonicity (Proposition~\ref{prop:g1}\ref{item2g}) and an explicit computation (see also \eqref{g=1}).
	The remaining assertions follow from Proposition~\ref{prop:g2}.
\end{proof}
}

\newcommand\betamin{\beta_{\mathrm{min}}}
\begin{proof}[{ Proof of Proposition~\ref{propsfracinfty2}}]
	{First we observe that  $\slopefuno=\liminf_{t\uparrow1} \frac{f_1(\sqrt t)}{1-t}$,
		that by convexity of $f_1(\sqrt\cdot)$ the $\liminf$ is actually a limit, and that, using convexity, $f_1(0)=\ell$ and $f_1(1)=0$, we have
		\begin{equation*}
		\slopefuno (1-t) \le f_1(\sqrt t) \le \ell (1-t)  \quad\text{ for all } t\in [0,1],
		\end{equation*}
		which implies $\slopefuno\in[0,\ell]$ and $f_1(1-b)\ge \slopefuno (1-(1-b)^2)=\slopefuno(2b-b^2)$ for all $b\in[0,1]$.
		We  also observe that, using \eqref{assf2}, the integrand in  \eqref{eqcar2} can be estimated by
		\begin{equation}\label{eqgfracintr}
		\begin{split}
		(1-\beta)\sqrt{s^2 f^2(\beta)(\alpha')^2 + (\beta')^2} 
		&=
		\sqrt{s^2 f_1^2(1-\beta)(\alpha')^2 + \frac14 \big(\frac{\de}{\de t} (1-\beta)^2 \big)^2} \\
		&\ge
		\sqrt{s^2 \slopefuno^2(2\beta-\beta^2)^2(\alpha')^2 + \frac14 \big(\frac{\de}{\de t} (1-\beta)^2 \big)^2} . 
		\end{split}
		\end{equation}}
	
	{To prove the first assertion we construct a competitor for 
		the characterization of $\Gzero$ in \eqref{eqcar2}. Assume $\slopefuno=0$ and fix $s>0$. 
		By definition of $\slopefuno$ we can find $q_s\in (0,1)$
		such that $s f_1(q_s)\le \frac12 (1-q_s^2)$.
		We let $\beta=1-q_s$ in $(\frac13,\frac23)$, $\beta=1$ in $\{0,1\}$, and the linear interpolation in between; 
		we let $\alpha=0$ in $(0,\frac13)$, $\alpha=1$ in $(\frac23,0)$, and the linear interpolation in between. A straightforward computation shows that
		\begin{equation*}
		\Gzero(s)\le \int_0^1
		\sqrt{s^2 f_1^2(1-\beta)(\alpha')^2 + \frac14 \Bigl(\frac{\de}{\de t} (1-\beta)^2 \Bigr)^2} \de t
		=q_s^2 + s f_1(q_s) \le \frac12 (1+q_s^2)<1.
		\end{equation*}}
	
	{ We now turn to the second assertion. We assume $s\slopefuno\ge 6$, let $\alpha\in H^1(0,1)$, $\beta\in H^1(0,1)$ be test functions in the characterization \eqref{eqcar2} of $\Gzero$, denote by $I(\alpha,\beta)$ the integral in \eqref{eqcar2}, and prove $I(\alpha,\beta)\ge1$.}
	
	{Assume first that there is $q\in (0,\frac12]$ such that
		$s\slopefuno \int_{\{q\le \beta<2q\}} {|\alpha'|} \de t\ge 2$. Then \eqref{eqgfracintr} yields
		\begin{equation*}
		\begin{split}
		\int_0^1(1-\beta)\sqrt{s^2 f^2(\beta)(\alpha')^2 + (\beta')^2} \de t
		&\ge \int_{\{q\le \beta<2q\}} s\slopefuno (2\beta-\beta^2){|\alpha'|} \de t 
		+ \int_{\{2q\le \beta\le 1\}} \frac12 \Big|\frac{\de}{\de t} (1-\beta)^2 \Big| \de t\\
		&\ge 2(2q-q^2)
		+ (1-2q)^2
		=1+2q^2>1
		\end{split}
		\end{equation*}
		and the proof is concluded. }
	
	{We  let $\betamin:=\min\beta$. 
		If $\betamin=0$ then necessarily $I(\alpha,\beta)\ge1$ and we are done. We can therefore assume
		$\betamin>0$.
		We also define, for $j\in\N$,  $\beta_j:=\min\{ 2^j\betamin,1\}$ and $E_j:=\{t: \beta_j\le \beta<\beta_{j+1}\}$. 
		We let $J$ be the largest index with $\beta_J<1$ and observe that $(E_j)$, $j=0, \dots, J$, is a partition of $(0,1)$, with
		$\beta(\inf E_j)=\beta(\sup E_j)=\beta_{j+1}$ and $\min \beta(E_j)=\beta_j$.
		We define $a_j:=\int_{E_j}{|\alpha'|}\de t$, which obey $\sum_j a_j{\ge 1}$, and
		\begin{equation*}
		\begin{split}
		g_j:= \int_{E_j}(1-\beta)\sqrt{s^2 f^2(\beta)(\alpha')^2 + (\beta')^2} \de t 
		&\ge
		\int_{E_j}\sqrt{s^2 \slopefuno^2(2\beta_j-\beta_j^2)^2(\alpha')^2 + \frac14 \Bigl(\frac{\de}{\de t} (1-\beta)^2 \Bigr)^2} \de t.
		\end{split}
		\end{equation*}
		We first treat $g_J$. We observe that
		\begin{equation*}
		I(\alpha,\beta)\ge g_J\ge \int_{E_J} s\slopefuno (2\beta_J-\beta_J^2){|\alpha'|} \de t 
		= s \slopefuno  (2\beta_J-\beta_J^2)a_J 
		\ge s \slopefuno \frac34 a_J,
		\end{equation*}
		so that if $s\slopefuno a_J\ge \frac43$ we are done. Therefore we can assume $s\slopefuno a_J<\frac43$ in the following.}
	
	{By Jensen's inequality, for any functions $h,k:E_j\to\R$ we have
		$|E_j|^{-1}\int_{E_j} \sqrt{h^2+k^2}\de t\ge
		\sqrt{(|E_j|^{-1} \int_{E_j} h \de t)^2
			+(|E_j|^{-1} \int_{E_j} k \de t)^2}$ and therefore
		$\int_{E_j} \sqrt{h^2+k^2}\de t\ge
		\sqrt{ (\int_{E_j} |h| \de t)^2
			+(\int_{E_j} |k| \de t)^2}$.
		For $j<J$, 
		$\int_{E_j} \frac12 |\frac{d}{dt} (1-\beta)^2 |\de t \ge  (1-\beta_j)^2-(1-\beta_{j+1})^2
		=  (\beta_{j+1}-\beta_j)(2-\beta_{j+1}-\beta_j)=\beta_j (2-3\beta_j)$. 
		Therefore, again for $j<J$,
		\begin{equation*}
		\begin{split}
		g_j &\ge 
		\sqrt{s^2 \slopefuno^2\beta_j^2(2-\beta_j)^2 a_j^2 + (\beta_j (2-3\beta_j))^2} = \beta_j (2-3\beta_j) \sqrt{1+ x_j^2} 
		\end{split}
		\end{equation*}
		with $x_j:=(s\slopefuno  a_j)(2-\beta_j)/(2-3\beta_j)$. Recalling $s\slopefuno a_j\le 2$ and $\beta_j\le\frac12$ we obtain $x_j\in [0,6]$. 
		We use the fact that $\sqrt{1+x^2}\ge 1 + \frac18 x^2$ for all $x\in[0,6]$ and obtain, inserting first the definition of $x_j$ and then the one of $\beta_j$,
		\begin{equation*}
		\begin{split}
		g_j\ge 
		\beta_j (2-3\beta_j) \Bigl(1+\frac18 \frac{s^2\slopefuno^2 a_j^2(2-\beta_j)^2}{(2-3\beta_j)^2}\Bigr)
		\ge (1-\beta_j)^2-(1-\beta_{j+1})^2 + \betamin \frac18 \frac{(2-\beta_j)^2}{2-3\beta_j}  2^j(s\slopefuno a_j)^2.
		\end{split}
		\end{equation*}
		We observe that $\frac{(2-x)^2}{2-3x} \ge 2$ for $x\in[0,\frac12]$,
		sum over $j$, and recall  that $g_J\ge (1-\beta_J)^2$. This gives
		\begin{equation*}
		\begin{split}
		\sum_j g_j\ge 
		(1-\betamin)^2 + \betamin \frac14 \sum_{j<J} 2^j(s\slopefuno a_j)^2.
		\end{split}
		\end{equation*}
		We recall that we are working under the assumption that {$\sum_j s\slopefuno a_j \geq s\slopefuno \ge 6$} and $s\slopefuno a_J\le \frac43$, which imply
		$4\le \sum_{j<J} s\slopefuno a_j$. By Hölder's inequality,
		\begin{equation*}
		16\le \Bigl(\sum_{j<J} s\slopefuno a_j\Bigr)^2\le \sum_{j<J} 2^j(s\slopefuno a_j)^2 \cdot \sum_{j<J} 2^{-j}
		\le 2\sum_{j<J} 2^j(s\slopefuno a_j)^2. 
		\end{equation*}
		Therefore
		$\sum_{j<J} 2^j(s\slopefuno a_j)^2\ge 8$, so that
		$I(\alpha,\beta)=\sum_jg_j\ge (1-\betamin)^2 + 2\betamin =1+\betamin^2>1$. This concludes the proof.}
\end{proof}


\subsection{Proof of Theorem~\ref{thm:blowup} in Section \ref{subsect:blowup}}\label{subsect:proofblowup}

This section is devoted to the proof of Theorem~\ref{thm:blowup}, which will be achieved through a sequence of lemmas. In the following we assume that $u$, $u_\e$ and $v_\e$ satisfy the assumption of Theorem~\ref{thm:blowup}. {We start by \ref{item1bu}. Since this assertion is} local, we can assume without loss of generality that the jump set of $u$ consists of a single point $\bar{x}\in(0,1)$, $J_u=\{\bar{x}\}$. We let
\begin{equation*}
I_\eta := (\bar{x}-\eta,\bar{x}+\eta),
\end{equation*}
where $\eta>0$ is always supposed to be small enough so that $I_\eta\subset(0,1)$. Notice that, since the endpoints of $I_\eta$ are not in the jump set of $u$, one has the localized convergence
\begin{equation*}
\lim_{\e\to0}\fe(u_\e,v_\e;I_\eta) = \f(u,1;I_\eta).
\end{equation*}
In the following, we will work with a fixed subsequence $\e_k\to0^+$, $k\to\infty$; to lighten the notation we will replace the subscript $\e_k$ by $k$ (thus, for instance, we set $u_k:=u_{\e_k}$, $v_k:=v_{\e_k}$,\ldots). We further assume that for this subsequence $u_k\to u$ and $v_k\to1$ almost everywhere in $(0,1)$, {and choose a continuous representative for $u_k$ and $v_k$.}

\begin{lemma} \label{lem:meps}
{For $\Lu$-a.e. $\eta$ sufficiently small we have $u_k(\bar{x}\pm\eta)\to u(\bar{x}\pm\eta)$, $v_k(\bar{x}\pm\eta)\to1$, and, if} $x_k\in\overline{I}_\eta$ {is} a minimum point of $v_k$ in the interval $I_\eta$, {in the sense that}
\begin{equation} \label{meps}
m_k := \inf_{x\in I_\eta} v_k(x) = v_k(x_k)\,,
\end{equation}
{then} the following properties hold:
\begin{enumerate}
\item\label{item1meps} $f_k(v_k(x_k))=f_k(m_k)\to 0$ as $k\to\infty$;
\item\label{item2meps} $\limsup_{k\to\infty} m_k <1$;
\item\label{item3meps} $x_k\to\bar{x}$ as $k\to\infty$;
\item\label{item4meps} if $\jump<1$, then $\liminf_{k\to\infty}m_k>0$.
\end{enumerate}
\end{lemma}

\begin{proof}
Since we are assuming that $u_k\to u$ and $v_k\to1$ almost everywhere in $(0,1)$, the convergence at the endpoints of the interval $I_\eta$ is automatically satisfied for almost every $\eta$.

\ref{item1meps}:
If for a subsequence we had $f_k(m_k)\geq{\sigma}>0$, then as the energies $\fk(u_k,v_k)$ are equibounded we would get
\begin{equation*}
C\geq \int_{I_\eta} f^2_k(v_k)|u_k'|^2\de x \geq {\sigma}^2 \int_{I_\eta}|u_k'|^2\de x\,.
\end{equation*}
Therefore $u\in H^1(I_\eta)$, which is a contradiction since $\bar{x}\in J_u\cap I_\eta$.

\ref{item2meps}:
Suppose by contradiction that for a subsequence $m_k\to1$. For a given $\delta>0$, we have $v_k(x)\geq 1-\delta$ in $I_\eta$ for $k$ sufficiently large; by assumption \eqref{assf3} this implies that
\begin{equation*}
|(1-v_k)f(v_k)-\ell| \leq \ell \omega(\delta)
\end{equation*}
for some modulus of continuity $\omega(\delta)\to0$ as $\delta\to0$. Therefore for $k$ large enough
\begin{align}\label{meps1}
\fk(u_k,v_k;I_\eta)
& \geq \int_{I_\eta} \Bigl( f_k^2(v_k)|u_k'|^2 + \frac{(1-v_k)^2}{4\e_k} \Bigr)\de x
\geq \int_{I_\eta} |u_k'|^2 \wedge \Bigl( \e_kf^2(v_k)|u_k'|^2 + \frac{(1-v_k)^2}{4\e_k} \Bigr) \de x \nonumber \\
& \geq \int_{I_\eta} \Bigl( |u_k'|^2 \wedge (1-v_k)f(v_k)|u_k'| \Bigr)\de x
\geq (1-\omega(\delta)) \int_{I_\eta} \Bigl( |u_k'|^2\wedge \ell|u_k'| \Bigr)\de x \\
& \geq (1-\omega(\delta)) \biggl( \ell\int_{I_\eta} |u_k'|\de x - \frac{\ell^2}{4}2\eta \biggr)\,. \nonumber
\end{align}
Notice now that { in view of \eqref{gnew1}} we have $\Gzero(s) \leq \ell s - \tilde{g}(s)$ {for $s$ small and} for some continuous and strictly positive function $\tilde{g}(s)>0$; inserting this inequality into \eqref{meps1} we find
\begin{align*}
\fk(u_k,v_k;I_\eta)
& \geq (1-\omega(\delta)) \biggl[ \Gzero\biggl(\Big|\int_{I_\eta}u_k'{\de x}\Big|\biggr) + \tilde{g}\biggl(\Big|\int_{I_\eta}u_k'{\de x}\Big|\biggr) - \frac{\ell^2\eta}{2} \biggr] \\
& = (1-\omega(\delta)) \Bigl[ \Gzero\bigl(| u_k(\bar{x}+\eta)-u_k(\bar{x}-\eta)|\bigr) + \tilde{g}\bigl(| u_k(\bar{x}+\eta)-u_k(\bar{x}-\eta)|\bigr) - \frac{\ell^2\eta}{2} \Bigr] \,.
\end{align*}
By passing to the limit as $k\to\infty$, since the left-hand side converges to the limit energy in $I_\eta$ and by assumption $u_k(\bar{x}\pm\eta)\to u(\bar{x}\pm\eta)$ we have
\begin{align*}
\int_{I_\eta}h(|u'|) {\de x}&+ \jump = \f(u,1;I_\eta) \\
& \geq (1-\omega(\delta)) \Bigl[ \Gzero\bigl(| u(\bar{x}+\eta)-u(\bar{x}-\eta)|\bigr) + \tilde{g}\bigl(| u(\bar{x}+\eta)-u(\bar{x}-\eta)|\bigr) - \frac{\ell^2\eta}{2} \Bigr] \,.
\end{align*}
Therefore, letting first $\delta\to0$ and {then} in turn $\eta\to0$, we conclude that
\begin{equation*}
\jump \geq \jump + \tilde{g}(|[u](\bar{x})|) > \jump \,,
\end{equation*}
which is a contradiction.

\ref{item3meps}:
By contradiction, assume that for a (not relabeled) subsequence $x_k\to\tilde{x}\neq\bar{x}$.  
{Fix $\delta\in(0,\eta)$ such that $v_k(\tilde{x}+\delta)\to1$ 
and $\bar{x}\notin[\tilde{x}-\delta,\tilde{x}+\delta]$. 
For  $k$ large enough, 
$x_k\in(\tilde{x}-\delta,\tilde{x}+\delta)$ and}
\begin{align*}
\fk(u_k,v_k;(\tilde{x}{-\delta},\tilde{x}+\delta))
& \geq \int_{{\tilde x-\delta}}^{\tilde{x}+\delta} \Bigl( \frac{(1-v_k)^2}{4\e_k} + \e_k|v_k'|^2 \Bigr) \de x
\geq \int_{x_k}^{\tilde{x}+\delta} (1-v_k)v_k'\de x \\
& = \frac12\bigl(1-m_k\bigr)^2 - \frac12\bigl(1-v_k(\tilde{x}+\delta)\bigr)^2 \,.
\end{align*}
Notice that the $\liminf$ as $k\to\infty$ of the right-hand side is a strictly positive quantity, in view of property~\ref{item2meps}, {which does not depend on $\delta$}. On the other hand
\begin{equation*}
\lim_{k\to\infty}\fk(u_k,v_k;({\tilde{x}-\delta},\tilde{x}+\delta)) = \f(u,1;{(\tilde{x}-\delta},\tilde{x}+\delta)) = \int_{{\tilde{x}-\delta}}^{\tilde{x}+\delta}h(|u'|)\de x\,,
\end{equation*}
and therefore we conclude that $\int_{{\tilde{x}-\delta}}^{\tilde{x}+\delta}h(|u'|)\de x\geq \frac12\liminf_{k\to\infty}(1-m_k)^2>0$. This is a contradiction since $\delta$ can be chosen arbitrarily small.

\ref{item4meps}:
By arguing as in the previous step and using the fact that $x_k\in I_\eta$ for $k$ large by \ref{item3meps} we have the inequality 
\begin{align*}
\fk(u_k,v_k;I_\eta)
& \geq \int_{\bar{x}-\eta}^{\bar{x}+\eta} (1-v_k)|v_k'|\de x
\geq \frac12\int_{\bar{x}-\eta}^{x_k} \bigl[(1-v_k)^2\bigr]'\de x - \frac12\int_{x_k}^{\bar{x}+\eta} \bigl[(1-v_k)^2\bigr]'\de x \\
& \geq (1-m_k)^2 - \frac12(1-v_k(\bar{x}-\eta))^2 - \frac12(1-v_k(\bar{x}+\eta))^2 \,.
\end{align*}
By letting $k\to\infty$,  and {using that $v_k(\bar{x}\pm\eta)\to1$,} we have
\begin{equation*}
\limsup_{k\to\infty}(1-m_k)^2 \leq {\int_{I_\eta} h(|u'|)\de x + \jump} \,,
\end{equation*}
and the right-hand side is strictly smaller than one provided $\eta$ is small enough.
\end{proof}

In the following we fix $\eta>0$ such that the conclusions of Lemma~\ref{lem:meps} hold, and we define $x_k$ and $m_k$ as in \eqref{meps}. Since these quantities depend on the choice of $\eta$, in the arguments below we will always let $k\to\infty$ and $\eta\to0$ in this order. The following formula, obtained by a simple change of variables, will be often useful: for every fixed $T>0$ and for $k\geq k_T$, for some $k_T$ large enough,
\begin{equation} \label{resc}
\begin{split}
\fk(u_k,v_k;(x_k-\e_kT,x_k+\e_kT)) &= \int_{x_k-\e_kT}^{x_k+\e_kT} \Bigl( f^2_k(v_k)|u_k'|^2 + \frac{(1-v_k)^2}{4\e_k} + \e_k|v_k'|^2 \Bigr)\de x \\
& = \int_{-T}^T \Bigl(\frac{1}{\e_k}f_k^2(w_k)|z_k'|^2 + \frac{(1-w_k)^2}{4} + |w_k'|^2 \Bigr) \de x \,,
\end{split}
\end{equation}
where $w_k$ and $z_k$ are the functions defined in \eqref{bub} and \eqref{buc} respectively.

\begin{lemma}[Compactness of $w_k$] \label{lem:w}
Up to subsequences, $w_k\wto w$ weakly in $H^1_{\loc}(\R)$ {and therefore locally uniformly} for some function $w$ such that $1-w\in H^1(\R)$. In particular $w$ satisfies
\begin{equation*}
\lim_{|x|\to\infty}w(x)=1\,,\qquad w(0)=\lim_{k\to\infty}v_k(x_k)=\lim_{k\to\infty}m_k <1.
\end{equation*}
\end{lemma}

\begin{proof}
The result is an immediate consequence of the uniform bound on $\|1-w_k\|_{H^1(-T,T)}$ for every $T>0$, which follows from \eqref{resc} and the fact that $(u_k,v_k)$ is a recovery sequence and has therefore equibounded energy. The fact that $w(0)<1$ follows from Lemma~\ref{lem:meps}\ref{item2meps}.
\end{proof}

In order to proceed with the proof of Theorem~\ref{thm:blowup}, it is convenient to distinguish between the cases $w(0)=0$ and $w(0)>0$. Indeed, we will see that these conditions correspond to $\jump=1$ {and} $\jump<1$ respectively; according to Theorem~\ref{thm:g1} the behaviour of minimizing sequences for the minimum problem defining $\Gzero$ is different in the two cases.

We first consider the (easier) case $w(0)=0$.
\begin{lemma}\label{lemmabluw0}
If $w(0)=0$, then Theorem~\ref{thm:blowup}\ref{item1bu} holds.
\end{lemma}

\begin{proof}
By \ref{item4meps} of Lemma~\ref{lem:meps} it has to be $\jump=1$. Therefore for $\eta>0$ and $T>0$ we have by {\eqref{resc}}
\begin{align*}
\f(u,1;I_\eta) &= \lim_{k\to\infty}\fk(u_k,v_k;I_\eta) \\
& \geq \liminf_{k\to\infty} \int_{-T}^T \Bigl(\frac{(1-w_k)^2}{4}+|w_k'|^2\Bigr)\de x \geq \int_{-T}^T \Bigl(\frac{(1-w)^2}{4}+|w'|^2\Bigr) \de x\,,
\end{align*}
where the last inequality follows by lower semicontinuity with respect to the weak convergence of $w_k$ to $w$ in $H^1(-T,T)$, proved in Lemma~\ref{lem:w}. By letting first $T\to\infty$ and then $\eta\to0$ we obtain
\begin{equation*}
1= \jump =\lim_{\eta\to0}\f(u,1;I_\eta)\geq \int_{-\infty}^\infty \Bigl(\frac{(1-w)^2}{4}+|w'|\Bigr)\de x \,.
\end{equation*}
As $w(0)=0$, { by Young inequality we conclude that the right-hand side of the previous expression is exactly $1$ and then} $w(x)=1-e^{-\frac{|x|}{2}}$ is the optimal profile for $\jump$.

It remains to prove the strong convergence in $H^1_{\loc}(\R)$. To this aim, it is sufficient to show convergence of the energies: suppose by contradiction that for some $M>0$ and $\sigma>0$
\begin{equation} \label{w0}
\limsup_{k\to\infty}\int_{-M}^M \Bigl( \frac{(1-w_k)^2}{4} + |w_k'|^2 \Bigr)\de x = \int_{-M}^M \Bigl(\frac{(1-w)^2}{4}+|w'|^2 \Bigr)\de x + \sigma\,.
\end{equation}
Then for every $\eta>0$ and $T>M$ one has
\begin{align*}
\f(u,1;I_\eta) &= \lim_{k\to\infty}\fk(u_k,v_k;I_\eta) \geq \limsup_{k\to\infty} \int_{-T}^T\Bigl( \frac{(1-w_k)^2}{4} + |w_k'|^2 \Bigr)\de x\\
& \geq \limsup_{k\to\infty}\int_{-M}^M\Bigl( \frac{(1-w_k)^2}{4} + |w_k'|^2 \Bigr)\de x + \liminf_{k\to\infty}\int_{[-T,T]\setminus[-M,M]}\Bigl( \frac{(1-w_k)^2}{4} + |w_k'|^2 \Bigr)\de x \\
& \geq \int_{-T}^T \Bigl( \frac{(1-w)^2}{4} + |w'|^2 \Bigr)\de x + \sigma\,,
\end{align*}
so that by letting first $T\to\infty$ and then $\eta\to0$ we end up with
\begin{equation*}
\jump \geq \int_{-\infty}^\infty \Bigl( \frac{(1-w)^2}{4} + |w'|^2 \Bigr)\de x + \sigma\,,
\end{equation*}
which is a contradiction. This completes the proof.
\end{proof}

\begin{lemma} \label{lem:w>0a}
{If $w(0)>0$, there is a subsequence such that
$z_k$ converges to some $z\in H^1_{\loc}(\R)\cap L^\infty(\R)$ strongly in $H^1_{\loc}(\R)$,
and $w_k$ converges to $w\in H^1_{\loc}(\R)$ strongly in $H^1_{\loc}(\R)$.
The pair $(\bar s^{-1}z,w)$, 
where $\bar{s}:=[u](\bar{x})$,
is a minimizer of the functional $\g_{|\bar s|}$, introduced in \eqref{defG}, in the class $\mathcal{U}_{1}$.} 
\end{lemma}

\begin{proof}
Preliminaries. We first prove the {weak} compactness of the sequence $z_k$. For every fixed $T>0$ and for $k\geq k_T$ one has
\begin{equation*}
\min_{[-T,T]}w_k = \min_{[x_k-\e_kT,x_k+\e_kT]}v_k = v_k(x_k) = w_k(0) \to w(0)>0,
\end{equation*}
so that
\begin{equation*}
\inf_{k\geq k_T} \min_{[-T,T]} \frac{1}{\e_k}f^2_k(w_k) = \inf_{k\geq k_T} \frac{1}{\e_k}\wedge f^2(w_k(0)) >0.
\end{equation*}
By \eqref{resc} it follows that for every $T>0$
\begin{equation}\label{eqzkbdh1}
\sup_{k\geq k_T} \int_{-T}^T |z_k'|^2\de x <\infty\,.
\end{equation}
Since by assumption the sequence $z_k$ is also uniformly bounded in $L^\infty$, we can extract a further subsequence such that
\begin{equation} \label{convz}
z_k \wto z \qquad\text{weakly in }H^1_{\loc}(\R)
\end{equation}
for some function $z\in H^1_{\loc}(\R)\cap L^\infty(\R)$. 
{Possibly after extracting a subsequence, by \eqref{eqzkbdh1} we can assume that
$|z_k'|^2\mathcal L^1$ converges weakly in measures to a Radon measure $\mu$.
By lower semicontinuity, this implies
\begin{equation}\label{muzpt}
|z'|^2\mathcal L^1\le \mu \hskip5mm\text{ as measures,}
\end{equation}
with $\mu(-T,T)=\int_{-T}^T |z'|^2\de x$ if and only if 
$z_k\to z$ strongly in $H^1(-T,T)$.}

\medskip\noindent\textit{{Step 1.}} We prove that $\mu(\{w=1\})=0$ and
\begin{equation*}
\jump \geq
\int_{\R} f^2(w)\de\mu +\int_{-\infty}^\infty \Bigl(  \frac{(1-w)^2}{4} + |w'|^2 \Bigr)\de x\,.
\end{equation*}
To see this, 
we first observe that for almost any $T>0$ we have $\mu(\{-T,T\})=0$ and, by \eqref{resc},
\begin{align*}
\f(u,1;I_\eta)
& = \lim_{k\to\infty}\fk(u_k,v_k;I_\eta)
{\geq}  \limsup_{k\to\infty} \int_{-T}^T \Bigl(\frac{1}{\e_k}f_k^2(w_k)|z_k'|^2 + \frac{(1-w_k)^2}{4} + |w_k'|^2 \Bigr) \de x.
\end{align*}
For $\delta>0$, 
we let $A_\delta:=\{x\in(-T,T): w(x)<1-\delta\}$. For almost every $\delta$ we have $\mu(\partial A_\delta)=0$.
By uniform convergence (Lemma~\ref{lem:w}), for sufficiently large $k$ we have $w_k(x)<1-\frac12\delta$ on $A_\delta$, therefore
$f(w_k)\to f(w)$ uniformly on $A_\delta$. Since  $\mu(\partial A_\delta)=0$,
$|z_k'|^2\mathcal L^1\wto\mu$, and $f(w)\in C^0(\overline A_\delta)$, 
\begin{equation*}
\int_{A_\delta} f^2(w) \de\mu
=
\lim_{k\to\infty} \int_{A_\delta} f^2(w)|z_k'|^2 \de x
=
\lim_{k\to\infty} \int_{A_\delta} f^2(w_k)|z_k'|^2 \de x
\le
\liminf_{k\to\infty} \int_{-T}^T \frac{1}{\e_k}f_k^2(w_k)|z_k'|^2 \de x.
\end{equation*}
By monotone convergence, this implies
\begin{equation}\label{eqwmin1}
\begin{split}
\int_{(-T,T)\cap\{w<1\}} f^2(w) \de \mu =&
\lim_{\delta\to0}
\int_{A_\delta} f^2(w)\de \mu
\le
\liminf_{k\to\infty} \int_{-T}^T \frac{1}{\e_k}f_k^2(w_k)|z_k'|^2 \de x. 
\end{split}
\end{equation}
Also by uniform convergence of $w_k$ to $w$, for $k$ sufficiently large
we have $1-2 \delta\le w_k(x)$ on $B_\delta:=(-T,T)\setminus A_\delta$. 
Recalling \eqref{assf1},  for large $k$ we have $f^2(1-2\delta)\le \e_k^{-1}\wedge f^2(w_k)=\e_k^{-1} f_k^2(w_k)$ on $B_\delta$.
Since $\mu(\partial B_\delta)=0$, 
\begin{equation*}
\begin{split}
f^2(1-2\delta)\mu(B_\delta)
&=
f^2(1-2\delta) \lim_{k\to\infty} \int_{B_\delta}|z_k'|^2 \de x\le
\liminf_{k\to\infty} \int_{B_\delta} \frac1{\e_k} f_k^2(w_k)|z_k'|^2 \de x\le C.
\end{split}
\end{equation*}
By monotone convergence,
\begin{equation*}
\mu((-T,T)\cap\{w=1\})
=\lim_{\delta\to0} \mu(B_\delta)
\le
\lim_{\delta\to0}
\frac{C}{f^2(1-2\delta)}=0.
\end{equation*}
Since $w_k$ converges to $w$ weakly in $H^1(-T,T)$  (Lemma~\ref{lem:w}), \eqref{eqwmin1} yields
\begin{equation*}
\int_{(-T,T)} f^2(w)\de\mu+
\int_{(-T,T)} \Bigl(
\frac{(1-w)^2}{4} + |w'|^2 \Bigr)\de x
\le \lim_{k\to\infty} \fk(u_k, v_k;I_\eta)=
\f(u,1;I_\eta)
\end{equation*}
for all $T$ and all $\eta$. Sending first $T\to\infty$ and then $\eta\to0$, 
\begin{equation}\label{eqbulemst1f}
\int_\R f^2(w)\de\mu+
\int_{\R}\Bigl( \frac{(1-w)^2}{4} + |w'|^2 \Bigr)\de x
\le \jump.
\end{equation}

\newcommand{\tin}{\text{in}}
\newcommand{\tout}{\text{out}}
\newcommand{\xst}{a}
\newcommand{\xend}{b}
\newcommand{\ukj}{u_{k,j}}
\newcommand{\vkj}{v_{k,j}}
\medskip\noindent\textit{{Step 2.}} By Lemma~\ref{lem:w} and \eqref{convz}, for every $j\in\N$ we can pick $T_j>j$ such that
$|w(\pm T_j)-1|^2\le   2^{-j}$
and $z_k(\pm T_j)\to z(\pm T_j)$ as $k\to\infty$.
Since $u_k$ and $z_k$ are bounded in $L^\infty$, possibly passing to a subsequence
we have 
\begin{equation}\label{eqzTj}
z_+:=\lim_{j\to\infty} z(T_j), \hskip5mm
z_-:=\lim_{j\to\infty} z(-T_j).
\end{equation}
We intend to show that $z_+-z_-=[u](\bar x)$.

To do so, we separate the part of the profile in the inner interval
$(\xst,\xend):=(\xst^j_k,\xend^j_k):=(x_k-\e_kT_j, x_k+\e_k T_j)$ from the part outside 
(for sufficiently large $k$ we can assume $[a-\e_k,b+\e_k]\subset I_\eta$).
Specifically, we set
\begin{equation*}
\ukj^\tin(x):=\begin{cases}
u_k(\xst) & \text{ if } x\le\xst,\\
u_k(x) & \text{ if } \xst<x\le\xend, \\
u_k(\xend) & \text{ if } x>\xend,\\               
\end{cases}
\hskip1cm
\ukj^\tout(x):=\begin{cases}
u_k(x) - u_k(\xst) & \text{ if } x\le\xst,\\
0 & \text{ if } \xst\le x\le\xend, \\
u_k(x)- u_k(\xend) & \text{ if } x>\xend\\               
\end{cases}                  
\end{equation*}
which implies, since $u_k(a)=z_k(-T_j)\to z(-T_j)$
and $u_k(b)=z_k(T_j)\to z(T_j)$,
that
\begin{equation*}
\ukj^\tin(x)\to u^\tin_j(x):=
\begin{cases} 
z(-T_j) & \text{ if } x\le \bar x,\\
z(T_j) & \text{ if } x> \bar x\,, 
\end{cases}
\hskip1cm
\ukj^\tout\to u^\tout_j:=u-u^\tin_j
\end{equation*}
in $L^1(I_\eta)$ {(as $k\to\infty$)}.
Correspondingly, we define
\begin{equation*}
\vkj^\tin(x):=
\begin{cases}
1                                 
& \text{ if } x\le \xst - \e_k,\\
v_k(x) & \text{ if } \xst\le x\le \xend, \\
1 & \text{ if } x\ge \xend+\e_k,\\
\text{affine interp.} & \text{ in between,}
\end{cases}
\hskip0.6cm
\vkj^\tout(x):=
\begin{cases}
v_k(x)                                 
& \text{ if } x\le \xst,\\
1 & \text{ if } \xst+\e_k\le x\le \xend-\e_k, \\
v_k(x) & \text{ if } x\ge \xend,\\
\text{affine interp.} & \text{ in between.}
\end{cases}
\end{equation*}
We observe that
$\vkj^\tin\to1$ with $\min \vkj^\tin=\vkj^\tin(x_k)=m_k\to w(0)<1$; $\vkj^\tout\to1$.
{By Lemma \ref{lem:w} and $|w(\pm T_j)-1|\le 2^{-j}$,
for $k$ sufficiently large (on a scale depending on $j$) 
we have
$|v_k(a_k^j)-1|\le 2^{-j+1}$ and 
$|v_k(b_k^j)-1|\le 2^{-j+1}$.}
The energy then obeys
\begin{equation*}
\fk(\ukj^\tin,\vkj^\tin;I_\eta) 
\le \fk(u_k, v_k; (a,b))+c 2^{-j},
\end{equation*}
\begin{equation*}
\fk(\ukj^\tout,\vkj^\tout;I_\eta) 
\le \fk(u_k, v_k; I_\eta\setminus (a,b))+c 2^{-j},
\end{equation*}
so that
\begin{equation*}
\fk(\ukj^\tin,\vkj^\tin;I_\eta)+
\fk(\ukj^\tout,\vkj^\tout;I_\eta)
\le \fk(u_k,v_k;I_\eta) + c 2^{-j}.
\end{equation*}
Passing to the {limit} and using \eqref{bua},
\begin{equation*}
\f(u^\tin_j,1;I_\eta)+\f(u^\tout_j,1;I_\eta)
\le \liminf_{k\to\infty}\fk(u_k,v_k;I_\eta) + c2^{-j}
=\f(u,1;I_\eta)+ c2^{-j}
\end{equation*}
for all $j$ and $\eta$. Using the explicit form of $u^\tin_j$, $u^\tout_j$ and $\f$, we obtain
\begin{equation}\label{eqgtintout}
\Gzero(|[u^\tin_j]|)+\Gzero(|[u^\tout_j]|)
\le \Gzero(|[u]|)+ c2^{-j},
\end{equation}
where all jumps are evaluated at $\bar x$.
Recalling $\min \vkj^\tin=\vkj^\tin(x_k)=m_k\to w(0)<1$ and $\vkj^\tin=1$ on $\partial I_\eta$,
\begin{equation*}
(1-m_k)^2 \le \int_{I_\eta} |1-\vkj^\tin| |(\vkj^\tin)'| \de x \le 
\int_{I_\eta} \left( \frac{(1-\vkj^\tin)^2}{4\e_k} + \e_k |(\vkj^\tin)'|^2\right) \de x \le 
\fk(\ukj^\tin,\vkj^\tin;I_\eta) 
\end{equation*}
which gives, by the same argument, first
\begin{equation*}
(1-w(0))^2+ \f(u^\tout_j,1;I_\eta) \le \f(u,1;I_\eta)+ c2^{-j}
\end{equation*}
and then
\begin{equation}\label{eqgtintout2}
(1-w(0))^2+ \Gzero(|[u^\tout_j]|)
\le \Gzero(|[u]|)+ c2^{-j}.
\end{equation}

Taking  $j\to\infty$ in \eqref{eqgtintout} and \eqref{eqgtintout2} we obtain, with \eqref{eqzTj},
\begin{equation*}
\Gzero(|z^+-z^-|)+ \Gzero(|[u]-(z^+-z^-)|)
\le \Gzero(|[u]|) \hskip5mm\text{and}\hskip5mm
(1-w(0))^2+\Gzero(|[u]-(z^+-z^-)|)\le \Gzero(|[u]|).
\end{equation*}
Recalling that $\Gzero$ is strictly subadditive on $(0,\infty)$, we see that one of the two terms in the first inequality must vanish. The second inequality excludes the case $z^+=z^-$, therefore $z^+-z^-=[u]$.

In order to show that $z'\in L^1(\R)$, we prove that $z$ is monotone. Assume for definiteness that $z^-\le z^+$. We define
\begin{equation*}
\bar z :=   z_-\vee (z_+\wedge z) \hskip5mm\text{and}\hskip5mm
\hat z (x) := \begin{cases}
\min \bar z([x,0]) & \text{ if }x\le0,\\
\max \bar z([0,x]) & \text{ if } x>0.
\end{cases}
\end{equation*}
{Then $z_-\le \hat z\le \bar z$ on $(-\infty,0)$
and $\bar z\le \hat z \le z_+$ on $(0,\infty)$, hence
$\hat z(\pm t)\to z_\pm$ for $t\to\infty$.} Obviously $\hat z$ is monotone. Further, 
$\hat z\in H^1_{\loc}(\R)$, $0\le \hat z'\le |z'|$, 
$\int_\R \hat z'\de x=z^+-z^-$,
and
\begin{equation*}
\int_{-\infty}^\infty  f^2(w)|\hat z'|^2 \de x 
\le \int_{-\infty}^\infty  f^2(w)|z'|^2 \de x ,
\end{equation*}
with equality if and only if $\hat z=z$ almost everywhere, which {would imply} that $z$ is also monotone.
In particular, $(\bar s^{-1}\hat z,w)\in\mathcal{U}_{1}$ for $\bar{s}=[u](\bar{x})$.
{ Using \eqref{defgnew}}, $|\hat z'|\le |z'|$,
\eqref{muzpt} and \eqref{eqbulemst1f} we obtain
\begin{equation*}
\begin{split}
\jump \le&\int_{-\infty}^\infty \Bigl( f^2(w)|\hat z'|^2 + \frac{(1-w)^2}{4} + |w'|^2 \Bigr)\de x \\
\le& \int_{-\infty}^\infty \Bigl( f^2(w)|z'|^2 + \frac{(1-w)^2}{4} + |w'|^2 \Bigr)\de x \\
\le& \int_{\R}f^2(w)\de\mu +\int_{-\infty}^\infty \Bigl( \frac{(1-w)^2}{4} + |w'|^2 \Bigr)\de x 
\le \jump. 
\end{split}
\end{equation*}
Therefore equality holds throughout. This implies in particular $\hat z=z$, so that $z$ is monotone,
$(\bar s^{-1}z,w)\in\mathcal{U}_1$,  and $\mu=|z'|^2\mathcal L^1$, which gives strong convergence of $z_k$ to $z$ in $H^1(-T,T)$ for all $T$.

\medskip\noindent\textit{{Step 3.}}
In order to complete the proof it only remains to show the strong convergence of $w_k$ to $w$.
We have for all $T>0$
\begin{equation} \label{wc1}
\liminf_{k\to\infty}\int_{-T}^T \Bigl( \frac{(1-w_k)^2}{4}+|w_k'|^2\Bigr)\de x  \geq \int_{-T}^T \Bigl( \frac{(1-w)^2}{4}+|w'|^2\Bigr)\de x
\end{equation}
and
\begin{equation*}
\liminf_{k\to\infty}\int_{-T}^T \frac{1}{\e_k}f^2_{\e_k}(w_k)|z_k'|^2\de x \geq \int_{-T}^T f^2(w)|z'|^2\de x\,.
\end{equation*}
Then arguing as in the case $w(0)=0$ and assuming by contradiction that \eqref{w0} holds, it is easily seen that one gets the contradiction
\begin{equation*}
\jump \geq \int_{-\infty}^\infty \Bigl( f^2(w)|z'|^2 + \frac{(1-w)^2}{4} + |w'|^2 \Bigr)\de x + \sigma\,,
\end{equation*}
proving that \eqref{wc1} actually holds with an equality and a limit on the left-hand side. This gives the strong convergence of $w_k$ to $w$.
\end{proof}

\begin{proof}[Proof of Theorem \ref{thm:blowup}]
 {The first assertion follows from Lemma \ref{lemmabluw0}
 and Lemma \ref{lem:w>0a}. We remark that the
uniqueness result of { Theorem \ref{thm:g1}(i)} implies convergence of the entire sequence (after suitable translations).

It remains to prove the second assertion of Theorem~\ref{thm:blowup}.}
We fix $\eta>0$ and we let $A:=(0,1)\setminus\bigcup_{x\in J_u}[x-\eta,x+\eta]$. We first remark that
\begin{equation*} 
v_\e\to 1 \qquad\text{uniformly on }A\,.
\end{equation*}
Indeed, this follows by the same argument used to prove Lemma~\ref{lem:meps}\ref{item3meps}. Therefore by assumption \eqref{assf3} we can write
\begin{equation*} 
(1-v_\e)f(v_\e) \geq (1-\omega(\e))\ell \qquad\text{on }A,
\end{equation*}
for some modulus of continuity $\omega(\e)\to0$ as $\e\to0$.
We introduce the set $A_\e:=\{x\in A:f_\e(v_\e)(x)<1\}$ and we observe that, since $f_\e(v_\e)=\sqrt{\e}f(v_\e)$ on $A_\e$,
\begin{align} \label{part2bb}
\int_A \biggl( f^2_\e(v_\e)|u_\e'|^2 + \frac{(1-v_\e)^2}{4\e}\biggr)\de x
& \geq  \int_{A_\e}\biggl( \e f^2(v_\e)|u_\e'|^2 + \frac{(1-v_\e)^2}{4\e}\biggr)\de x + \int_{A_\e^c} f^2_\e(v_\e)|u_\e'|^2 \de x \nonumber \\
& \geq \int_{A_\e} (1-v_\e)f(v_\e)|u_\e'|\de x + \int_{A_\e^c} |u_\e'|^2\de x \\
& \geq \ell(1-\omega(\e))\int_{A_\e}|u_\e'|\de x + \int_{A_\e^c} |u_\e'|^2\de x \,. \nonumber
\end{align}
The sequence $u_\e'\chi_{A_\e^c}$ is bounded in $L^2(A)$, and therefore, possibly passing to a subsequence, has a weak limit $p\in L^2(A)$.
The sequence $u_\e'\chi_{A_\e}$ is bounded in $L^1(A)$, and analogously we can assume that
$u_\e'\chi_{A_\e}\wto \mu$ and 
$|u_\e'|\chi_{A_\e}\wto \nu$ weakly in measures, 
with 
\begin{equation*}
|\mu|(A)\le \nu(A)\le\liminf_{\e\to0} \int_A |u_\e'|\chi_{A_\e}\de x. 
\end{equation*}
Since $u_\e'=u_\e'\chi_{A_\e}+u_\e'\chi_{A_\e^c}$ converges distributionally to $u'\in L^\infty(A)$, we have $\mu+p\mathcal L^1=u'\mathcal L^1$, which implies that $\mu=(u'-p)\mathcal L^1$
and $|\mu|(A)=\int_A |u'-p|\de x$. Thus \eqref{part2bb} gives
\begin{equation*}
\begin{split}
\liminf_{\e\to0}\int_A \biggl( f^2_\e(v_\e)|u_\e'|^2 + \frac{(1-v_\e)^2}{4\e}\biggr)\de x
\ge \int_A (\ell |u'-p| + p^2) \de x
\ge \int_A (|u'|^2+ |u'-p|^2) \de x,
\end{split}
\end{equation*}
where in the last step we used that for any $t\in [-\frac\ell2,\frac\ell2]$ and any $y\in\R$ one has
\begin{equation*}
\ell|t-y|+ y^2\ge t^2+(t-y)^2.
\end{equation*}
Indeed, this inequality is the same as $\ell|t-y|\ge 2 t (t-y)$, which is true since $\ell\ge 2 |t|$.

Since we know that 
\begin{equation*}
\begin{split}
\limsup_{\e\to0}\int_A \biggl( f^2_\e(v_\e)|u_\e'|^2 + \frac{(1-v_\e)^2}{4\e}\biggr)\de x
\le 
\f(u,1;A)=  \int_A |u'|^2 \de x,
\end{split}
\end{equation*}
we conclude that $p=u'$. In particular
\begin{equation*}
\begin{split}
\int_A |u'|^2 \de x\le &
\liminf_{\e\to0}\int_{A_\e^c} f^2_\e(v_\e)|u_\e'|^2 \de x
\le 
\liminf_{\e\to0}\int_A f^2_\e(v_\e)|u_\e'|^2 \de x\\
\le& 
\limsup_{\e\to0} \f_\e(u_\e,v_\e;A)
=\f(u,1;A)=\int_A |u'|^2\de x,
\end{split}
\end{equation*}
hence equality holds throughout and \eqref{bud} is proven.
This completes the proof.
\end{proof}


\subsection{Proof of the statements of Section~\ref{subsect:propbarg}}\label{subsect:proofbarg}
Let $\Gdue$ be defined as in \eqref{defgbar}. First, let us observe that 
{$\Gdue(s,0)=\Gzero(s)$ and that}
the characterizations of ${\Gdue}$ given in \eqref{defgbar2}--\eqref{defgbar2bis} follow by the same arguments used { to prove the corresponding characterizations of $\Gzero$ in \eqref{eqcar1}--\eqref{eqcar3}.} We collect the main properties of the function ${\Gdue}$ in the following proposition.

\begin{proposition} \label{prop:gbar}
	The function ${\Gdue}$ enjoys the following properties:
	\begin{enumerate}
		\item\label{item1gbar} ${\Gdue}$ is monotone nondecreasing in both variables, 
		${\Gdue}(s,s')=\Gzero(s)$ if $s\geq s'$; 
		\item\label{item2gbar} for every $s_1,s_2,s'\geq0$,
		\begin{equation} \label{gbarsub}
		{\Gdue}(s_1+s_2,s') \leq \Gzero(s_1) + {\Gdue}(s_2,s')\,;
		\end{equation}
		\item\label{item3gbar} ${\Gdue}(0,s')=(1-m_{s'})^2$, $(1-m_{s'})^2\leq {\Gdue}(s,s') \leq 1\wedge \bigl( (1-m_{s'})^2 +\ell s\bigr)$;
		\item\label{item4gbar} ${\Gdue}$ is continuous in both variables and, for fixed $s'$, the map ${\Gdue}(\cdot,s')$ is Lipschitz continuous with Lipschitz constant $\ell$ (in particular $\Gdue$ is globally continuous);
		\item\label{item5gbar} $\lim_{s\uparrow\infty}{\Gdue}(s,s')= 1$; {for $s'>0$}  $\lim_{s\downarrow0}\frac{{\Gdue}(s,s')-{\Gdue}(0,s')}{s}=0$.
	\end{enumerate}
\end{proposition}

\begin{proof}
	\ref{item1gbar}:
	The monotonicity of ${\Gdue}$ in the first variable follows from $\mathcal{G}_{s_1}(\alpha,\beta)\leq\mathcal{G}_{s_2}(\alpha,\beta)$ for $s_1<s_2$, while the monotonicity in the second variable follows since the map $s'\mapsto m_{s'}$ is decreasing. 
	If $s\geq s'$, the optimal profile $\beta_s$ for $\Gzero(s)$ is admissible in the minimum problem which defines ${\Gdue}(s,s')$, by monotonicity of $m_{s'}$, and therefore the values of the two functions coincide. This completes the proof of \ref{item1gbar}.
	
	\smallskip\ref{item2gbar}: Fix $\sigma>0$ and let $T>0$, $(\alpha_1,\beta_1)$, $(\alpha_2,\beta_2)\in H^1(-T,T)\times H^1(-T,T)$ be such that $\alpha_i(-T)=0$, $\alpha_i(T)=s_i$, $0\leq\beta_i\leq1$, $\beta_i(\pm T)=1$, $\inf\beta_2\leq \beta_{s'}(0)$, and
	\begin{align*}
	\int_{-T}^T \biggl( f^2(\beta_1)|\alpha_1'|^2 + \frac{(1-\beta_1)^2}{4} + |\beta_1'|^2 \biggr)\de t \leq \Gzero(s_1) + \sigma\,,\\
	\int_{-T}^T \biggl( f^2(\beta_2)|\alpha_2'|^2 + \frac{(1-\beta_2)^2}{4} + |\beta_2'|^2 \biggr)\de t \leq {\Gdue}(s_2,s') + \sigma\,,
	\end{align*}
	according to the representation formulas \eqref{eqcar1} and \eqref{defgbar2} {after rescaling $\alpha$}. By defining
	\begin{equation*}
	\alpha(t)=
	\begin{cases}
	\alpha_1(t) & \text{if }t\in[-T,T], \\
	\alpha_2(t-2T)+s_1 & \text{if }t\in[T,3T],
	\end{cases}
	\qquad\qquad
	\beta(t)=
	\begin{cases}
	\beta_1(t) & \text{if }t\in[-T,T], \\
	\beta_2(t-2T) & \text{if }t\in[T,3T],
	\end{cases}
	\end{equation*}
	we have $\alpha,\beta\in H^1(-T,3T)$, with $\alpha(-T)=0$, $\alpha(3T)=s_1+s_2$, $\beta(-T)=\beta(3T)=1$, $0\leq\beta\leq1$, $\inf\beta\leq\beta_{s'}(0)$. By using the pair $({\frac{1}{s_1+s_2}}\alpha,\beta)$ as a competitor in \eqref{defgbar2} we obtain
	\begin{equation*}
	{\Gdue}(s_1+s_2,s') \leq \Gzero(s_1) + {\Gdue}(s_2,s') + 2\sigma\,,
	\end{equation*}
	and the conclusion follows as $\sigma$ is arbitrary.
	
	\smallskip\ref{item3gbar}:
	The value of ${\Gdue}(0,s')$ can be computed by observing that in this case the minimum problem \eqref{defgbar} defining ${\Gdue}$ is independent of $\alpha$, and the optimal profile is given by the function
	\begin{equation} \label{betabar}
	\bar{\beta}_{s'}(t) := 1-(1-m_{s'})e^{-\frac{|t|}{2}} \,.
	\end{equation}
	The lower bound on ${\Gdue}(s,s')$ follows directly by monotonicity. To derive the upper bound with the constant 1, one can simply consider the admissible pair $(\alpha,\beta)\in\mathcal{V}_{s'}$ given by
	\begin{equation*}
	\alpha(t) =
	\begin{cases}
	0 & \text{if }t\in(-\infty,-1),\\
	\frac{t}{2} + \frac12 & \text{if }t\in[-1,1],\\
	1 & \text{if }t\in(1,\infty),
	\end{cases}
	\qquad\qquad
	\beta(t)=
	\begin{cases}
	1-e^{\frac{t+1}{2}} & \text{if }t\in(-\infty,-1),\\
	0 & \text{if }t\in[-1,1],\\
	1-e^{-\frac{t-1}{2}} & \text{if }t\in(1,\infty).
	\end{cases}
	\end{equation*}
	The inequality ${\Gdue}(s,s')\leq {\Gdue}(0,s')+\ell s$ follows directly from \eqref{gbarsub} and $\Gzero(s)\leq\ell s$.
	
	\smallskip\ref{item4gbar}: The continuity of $s'\mapsto{\Gdue}(s,s')$ can be proved by taking into account the continuity of the map $s'\mapsto m_{s'}$, proved in Theorem~\ref{thm:g1}. The Lipschitz continuity of $s\mapsto{\Gdue}(s,s')$ is a consequence of the monotonicity of this map, of the subadditivity inequality \eqref{gbarsub}, and of the bound $\Gzero(s)\leq\ell s$.
	
	\smallskip\ref{item5gbar}: The first limit is a trivial consequence of ${\Gdue}(s,s')=\Gzero(s)$ for $s\geq s'$, and of {Theorem~\ref{thm:propg1}}. To compute the slope of ${\Gdue}(\cdot,s')$ at the origin, notice that given any $\alpha$ 
	such that  $\alpha'$ has compact support and the pair $(\alpha,\bar{\beta}_{s'})$ belongs to $\mathcal{V}_{s'}$, where $\bar{\beta}_{s'}$ is the optimal profile defined in \eqref{betabar}, we have
	\begin{equation*}
	{\Gdue}(s,s') \leq \int_{-\infty}^\infty \biggl( s^2f^2(\bar\beta_{s'})|\alpha'|^2 + \frac{(1-\bar\beta_{s'})^2}{4} + |\bar\beta_{s'}'|^2 \biggr) \de t
	=  \int_{-\infty}^\infty s^2f^2(\bar\beta_{s'})|\alpha'|^2 \de t + (1-m_{s'})^2\,.
	\end{equation*} 
	Dividing by $s$ and letting $s\to0^+$ we obtain the inequality $\limsup_{s\to0^+}\frac{{\Gdue}(s,s')-{\Gdue}(0,s')}{s}\leq0$. The lower bound ${\Gdue}(s,s')\geq{\Gdue}(0,s')$ concludes the proof.
\end{proof}

The next proposition shows that the inequality \eqref{gbarsub} is strict.

\begin{proposition}[Strict subadditivity] \label{prop:gbar2}
For every $s_1>0$, $s_2\geq0$, and $s'>0$, we have
	\begin{equation}\label{ssbarg}
          \Gdue(s_1+s_2,s')<\Gzero(s_1)+\Gdue(s_2,s').
	\end{equation}
\end{proposition}

\begin{proof}
	If $s_1+s_2\geq s'$, then $\Gdue(s_1+s_2,s')=\Gzero(s_1+s_2)$ and \eqref{ssbarg} holds since $\Gzero$ is strictly subadditive by {Theorem~\ref{thm:propg1}}. 
	{If $\Gzero(s_1)=1$ the assertion follows from $\Gdue(s_1+s_2,s')\le 1$ and $\Gdue(s_2,s')>0$.} {Analogously if $\Gdue(s_2,s')=1$.}
	Let us assume now that $s_1+s_2<s'$, {$s_1<\sfrac$} {and $\Gdue(s_2,s')<1$}. Fixed {$\eta\in (0, \frac12(1-\Gdue(s_2,s')))$}, let $\gamma_1,\gamma_2\in W^{1,1}_0([0,1],[0,1])$ be such that $\max\gamma_2\geq(1-m_{s'})^2$ and 
	\begin{equation}\label{ssgbar2}
{1>}\Gzero(s_1)=\int_0^1 \sqrt{s_1^2(\tif(\sqrt{\gamma_1}))^2 + \textstyle\frac{|\gamma_1'|^2}{4}} \de t,
	\end{equation}
	\begin{equation}\label{ssgbar1}
{1>}\Gdue(s_2,s')+\eta\ge\int_0^1 \sqrt{s_2^2(\tif(\sqrt{\gamma_2}))^2 + \textstyle\frac{|\gamma_2'|^2}{4}} \de t.
	\end{equation}
	Notice that this is possible by the characterization \eqref{defgbar2bis}
	and Proposition~\ref{prop:g2}\ref{prop:g2unique}.
	In particular $\gamma_1$ is the unique minimizer of the problem defining $
	\Gzero(s_1)$, {and \eqref{ssgbar1} implies $0\le\gamma_2\le \frac12 (1+\Gdue(s_2,s'))<1$ as well as $\int_0^1|\gamma_2'|\de t\le 2$.}
	
	The function $\gamma:=(\gamma_1+\gamma_2)\wedge 1$ is admissible for $\Gdue(s_1+s_2,s')$, being ${\max\gamma \geq \max\gamma_2\ge}(1-m_{s'})^2$. Therefore,
	\begin{multline*}
	\Gdue(s_1+s_2,s')\leq \int_0^1 \sqrt{(s_1+s_2)^2(\tif(\sqrt{\gamma}))^2 + \textstyle\frac{|\gamma'|^2}{4}} \de t\\
	\leq \int_0^1 \sqrt{s_1^2(\tif(\sqrt{\gamma}))^2 + \textstyle\frac{|\gamma_1'|^2}{4}} \de t+
	\int_0^1 \sqrt{s_2^2(\tif(\sqrt{\gamma}))^2 + \textstyle\frac{|\gamma_2'|^2}{4}} \de t.
	\end{multline*}
	In order to estimate the first integral, we observe that $\gamma\geq\gamma_1$ {by definition}
and that by \eqref{assf2} $f_1(\sqrt{\cdot})$ is {strictly} decreasing, hence
	\begin{equation}\label{equnggammamgamma1}
	\int_0^1 \sqrt{s_1^2(\tif(\sqrt{\gamma}))^2 + \textstyle\frac{|\gamma_1'|^2}{4}} \de t{\le} \int_0^1 \sqrt{s_1^2(\tif(\sqrt{\gamma_1}))^2 + \textstyle\frac{|\gamma_1'|^2}{4}} \de t.	 
	\end{equation}
	As for the second integral, we {first assume $s_2>0$ and} show that there exists a constant $c_{s_1,s_2,s'}{>0}$, only depending on $s_1$, $s_2$ and $s'$, such that 
	\begin{equation}\label{stimacost}\int_0^1 \sqrt{s_2^2(\tif(\sqrt{\gamma}))^2 + \textstyle\frac{|\gamma_2'|^2}{4}} \de t\leq \int_0^1 \sqrt{s_2^2(\tif(\sqrt{\gamma_2}))^2 + \textstyle\frac{|\gamma_2'|^2}{4}} \de t-c_{s_1,s_2,s'}.
	\end{equation}
	This would lead to 
	\begin{multline*}
	\Gdue(s_1+s_2,s')\leq\int_0^1 \sqrt{s_1^2(\tif(\sqrt{\gamma_1}))^2 + \textstyle\frac{|\gamma_1'|^2}{4}} \de t+
	\int_0^1 \sqrt{s_2^2(\tif(\sqrt{\gamma_2}))^2 + \textstyle\frac{|\gamma_2'|^2}{4}} \de t-c_{s_1,s_2,s'}\\
	\leq \Gzero(s_1)+\Gdue(s_2,s')+\eta-c_{s_1,s_2,s'},
	\end{multline*}
	and then to \eqref{ssbarg} as $\eta\to0$.
	
	In order to check \eqref{stimacost}, 
	{we let $\rho:=\frac{1-\Gdue(s_2,s')}2\wedge \frac12 \gamma_1(\frac12)>0$. Then there is  $\delta>0$, only depending on $s_1$, $s_2$ and $s'$,  such that $\gamma_1\geq {\rho}$ in  $J:=(\frac12-\delta,\frac12+\delta)$}.
{This implies $\gamma=1\wedge (\gamma_1+\gamma_2)\ge \gamma_2+\rho$ in $J$. 
Convexity and strict monotonicity of  $f_1(\sqrt{\cdot})$ 
imply convexity and strict monotonicity of 
$(f_1(\sqrt{\cdot}))^2$, and hence monotonicity of its difference quotients. Therefore}
\begin{equation*}
 {f_1^2(\sqrt{\gamma_2})-f_1^2(\sqrt{\gamma})\ge
 f_1^2(\sqrt{\gamma_2})-f_1^2(\sqrt{\gamma_2+\rho})\ge
 f_1^2(\sqrt{1-\rho})-f_1^2(1)=:\tilde c>0 \text{ in $J$}.}
 \end{equation*}
  {We remark that $\tilde c$ only depends on $\rho$ and therefore on $s_1$, $s_2$, $s'$.}
	{Let $I_\eta:=\{t\in J: |\gamma'_2|(t)\le 2\delta^{-1}\}$. From
	$\int_0^1|\gamma_2'|\de t\le 2$ and $\mathcal L^1(J)=2\delta$ we obtain $\mathcal L^1(I_\eta)\ge \delta$.
	Since $\sqrt{A-\e}\le \sqrt{A}-\frac{\e}{2\sqrt A}$ whenever $0\le \e\le A$ and $f_1\le \ell$, 
	}
	\begin{equation*}
	\begin{split}
	\int_{I_\eta} \sqrt{s_2^2(\tif(\sqrt{\gamma}))^2 + \textstyle\frac{|\gamma_2'|^2}{4}} \de t
	& \leq \int_{I_\eta} \sqrt{s_2^2(\tif(\sqrt{\gamma_2}))^2 + \textstyle\frac{|\gamma_2'|^2}{4}-{s_2^2\tilde c }} \de t\\
	& \leq \int_{I_\eta} \sqrt{s_2^2(\tif(\sqrt{\gamma_2}))^2 + \textstyle\frac{|\gamma_2'|^2}{4}} \de t -
	{\frac{{s_2^2\tilde c}\delta}{2\sqrt{s_2^2\ell^2+{\delta^{-2}}}}},
	\end{split}
	\end{equation*}
	where {the last term is nonzero and} only depends on $s_1$, $s_2$ {and $s'$}. Hence \eqref{stimacost} follows.
	
{It remains to deal with the case $s_2=0$. In this situation we can take  {$\gamma_2(t)=2(t\wedge (1-t)) (1-m_{s'})^2$ and $\eta=0$} in \eqref{ssgbar1}, {then $\gamma>\gamma_1$ pointwise in $(0,1)$ and} \eqref{equnggammamgamma1} is a strict inequality; {by monotonicity of $f_1$} \eqref{stimacost} holds with
		$c_{s_1,s_2,s'}=0$. The proof is then concluded as above.}
\end{proof}

{
\begin{proof}[Proof of Theorem~\ref{thm:gbar1}]
All the properties listed in Section~\ref{sect:setting} follow by combining {Theorem~\ref{thm:propg1}}, Proposition~\ref{prop:gbar} and Proposition~\ref{prop:gbar2}, and observing that by definition one has ${\Gdue}(s,0)=\Gzero(s)$.
\end{proof}
}

\begin{proof}[{ Proof of Lemma~\ref{lem:gbarmu1}}]
	Given $T>0$ and an admissible pair $(\alpha_\mu,\beta_\mu)$ for ${\Gdue}^{(\mu)}(s,s')$, we construct an admissible pair $(\alpha,\beta)$ for ${\Gdue}(s,s')$ in $(-T-1,T+1)$ by setting $\alpha\equiv\alpha_\mu(-T)$ in $(-T-1,-T)$, $\alpha=\alpha_\mu$ in $(-T,T)$, $\alpha\equiv\alpha_\mu(T)$ in $(T,T+1)$, and by setting $\beta=\beta_\mu$ in $(-T,T)$ and linearly linked to the value 1 in $(-T-1,-T)$ and $(T,T+1)$. Then by \eqref{defgbar2}, {scaling $\alpha$,}
	\begin{align*}
	{\Gdue}(s,s')
	&\leq \int_{-T-1}^{T+1} \biggl( f^2(\beta)|\alpha'|^2 + \frac{(1-\beta)^2}{4} + |\beta'|^2 \biggr) \de t \\
	&\leq \int_{-T}^{T} \biggl( f^2(\beta_\mu)|\alpha_\mu'|^2 + \frac{(1-\beta_\mu)^2}{4} + |\beta_\mu'|^2 \biggr) \de t +3\mu^2\,,
	\end{align*}
	from which it follows that ${\Gdue}(s,s')\leq {\Gdue}^{(\mu)}(s,s')+3\mu^2$. The other inequality follows by an analogous construction, reversing the roles of ${\Gdue}$ and ${\Gdue}^{(\mu)}$.
\end{proof}


\bigskip
\bigskip
\noindent
{\bf Acknowledgments.}{
We acknowledge support by the Deutsche Forschungsgemeinschaft (DFG, German Research Foundation) through project 211504053 -- CRC 1060 \textit{The mathematics of emergent effects} at the University of Bonn. MB is member of the 2019 INdAM - GNAMPA project \emph{Analysis and optimisation of thin structures}. FI wishes to thank the warm hospitality of the Institute for Applied Mathematics at the University of Bonn where part of this work has been carried out. FI has been a recipient of scholarships from the Fondation Sciences Mathématiques de Paris, Emergence Sorbonne Université, and the Séphora-Berrebi Foundation, and gratefully acknowledges their support. FI acknowledges support through the PEPS CNRS 2019 \textit{Évolution quasi-statique de la rupture cohésive à travers une	approche de champ de phase}.}

\bigskip
{
}


\frenchspacing
\begin{thebibliography}{99}

\bibitem{AleCriOrl} {\sc R.\ Alessi, V.\ Crismale, G.\ Orlando,}
{\it Fatigue effects in elastic materials with variational damage models: a vanishing viscosity approach.}
J. Nonlinear Sci. \textbf{29} (2019) 1041--1094.

\bibitem{AleMarVid14} {\sc R.\ Alessi, J.-J.\ Marigo, S.\ Vidoli,}
{\it Gradient damage models coupled with plasticity and nucleation of cohesive cracks.}
Arch. Ration. Mech. Anal. \textbf{214} (2014), 575--615.

\bibitem{AleMarVid15} {\sc R.\ Alessi, J.-J.\ Marigo, S.\ Vidoli,}
{\it Gradient damage models coupled with plasticity: Variational formulation and main properties.}
Mech. Materials \textbf{80} (2015), 351--367.

\bibitem{Alm} {\sc S.\ Almi,}
{\it Energy release rate and quasi-static evolution via vanishing viscosity in a fracture model depending on the crack opening.}
ESAIM Control Optim. Calc. Var. \textbf{23} (2017), no. 3, 791--826.

\bibitem{ALL} {\sc S.\ Almi, G.\ Lazzaroni, I.\ Lucardesi,}
{\it Crack growth by vanishing viscosity in planar elasticity.}
Mathematics in Engineering \textbf{2} (2020).

\bibitem{AFP} {\sc L.\ Ambrosio, N.\ Fusco, D.\ Pallara,}
{\it Functions of bounded variation and free discontinuity problems.}
Oxford University Press, New York, 2000.

\bibitem{AmbTor90} {\sc L.\ Ambrosio, V.M.\ Tortorelli,}
{\it Approximation of functionals depending on jumps by elliptic functionals via $\Gamma$-convergence.}
Comm. Pure Appl. Math. \textbf{43} (1990), no. 8, 999--1036.

\bibitem{AmbTor92} {\sc L.\ Ambrosio, V.M.\ Tortorelli,}
{\it On the approximation of free discontinuity problems.}
Boll. Un. Mat. Ital. B (7) \textbf{6} (1992), no. 1, 105--123.

\bibitem{ArtCagForSol} {\sc M.\ Artina, F.\ Cagnetti, M.\ Fornasier, F.\ Solombrino,}
{\it Linearly constrained evolutions of critical points and an application to cohesive fractures.}
Math. Models Methods Appl. Sci. \textbf{27} (2017), no. 2, 231--290.

\bibitem{Bar} {\sc G.I. Barenblatt,}
{\it The mathematical theory of equilibrium cracks in brittle fracture.}
Advances in Applied Mechanics \textbf{7} (1962), 55--129.

\bibitem{BB} {\sc G.\ Bouchitt\'e, G.\ Buttazzo,}
{\it New lower semicontinuity results for nonconvex functionals defined on measures.}
Nonlinear Anal. \textbf{15} (1990), 679--692.

\bibitem{BBB} {\sc G.\ Bouchitt\'e, A.\ Braides, G.\ Buttazzo,}
{\it Relaxation results for some free discontinuity problems.}
J. Reine Angew. Math. \textbf{458} (1995), 1--18.

\bibitem{BFM} {\sc B.\ Bourdin, G.\ Francfort, J.-J.\ Marigo,}
{\it The variational approach to fracture.}
Journal of Elasticity \textbf{91} (1--3) (2008), 5--148.


\bibitem{Br02}
{{\sc A. Braides,}
\newblock {\em {$\Gamma$}-convergence for beginners}, volume~22 of {\em Oxford
  Lecture Series in Mathematics and its Applications}.
Oxford University Press, Oxford, 2002.}



\bibitem{BDG} {\sc A.\ Braides, G.\ Dal Maso, A.\ Garroni,}
{\it Variational formulation of softening phenomena in fracture mechanics: the one-dimensional case.}
Arch. Ration. Mech. Anal. \textbf{146} (1999), 23--58.

\bibitem{CCF} {\sc L.\ Caffarelli, F.\ Cagnetti, A.\ Figalli,}
{\it Optimal regularity and structure of the free boundary for minimizers in cohesive zone models.}
{Archive for Rational Mechanics and Analysis \textbf{237} (2020), 299-345.}



\bibitem{Cag} {\sc F.\ Cagnetti,}
{\it A vanishing viscosity approach to fracture growth in a cohesive zone model with prescribed crack path.}
Math. Models Methods Appl. Sci. \textbf{18} (2008), no. 7, 1027--1071.

\bibitem{CagToa} {\sc F.\ Cagnetti, R.\ Toader}
{\it Quasistatic crack evolution for a cohesive zone model with different response to loading and unloading: a Young measures approach.}
ESAIM Control Optim. Calc. Var. \textbf{17} (2011), no. 1, 1--27.

\bibitem{Cha} {\sc A.\ Chambolle,}
{\it A density result in two-dimensional linearized elasticity, and applications.}
Arch. Ration. Mech. Anal. \textbf{167} (2003), no. 3, 211--233.

\bibitem{CFI} {\sc S.\ Conti, M.\ Focardi, F.\ Iurlano,}
{\it Phase field approximation of cohesive fracture models.}
Ann. Inst. H. Poincar\'e Anal. Non Lin\'eaire \textbf{33} (2016), no. 4, 1033--1067.

\bibitem{CriLazOrl} {\sc V.\ Crismale, G.\ Lazzaroni, G.\ Orlando,}
{\it Cohesive fracture with irreversibility: Quasistatic evolution for a model subject to fatigue.}
Math. Models Methods Appl. Sci. \textbf{28} (2018), no. 7, 1371--1412.


\bibitem{Dalmaso1993}
{{\sc  
G.~{Dal Maso}}.
\newblock {\em An introduction to {$\Gamma$}-convergence}.
\newblock Progress in Nonlinear Differential Equations and their Applications,
  8. Birkh\"auser Boston Inc., Boston, MA, 1993.}


\bibitem{DMFraToa2} {\sc G.\ Dal Maso, G.\ Francfort, R.\ Toader,}
{\it Quasistatic crack growth in finite elasticity.}
Preprint (2004). (This paper is the more detailed preprint version of \cite{DMFraToa}.)
\url{https://arxiv.org/abs/math/0401196}

\bibitem{DMFraToa} {\sc G.\ Dal Maso, G.\ Francfort, R.\ Toader,}
{\it Quasistatic crack growth in nonlinear elasticity.}
Arch. Ration. Mech. Anal. \textbf{176} (2005), 165--225.

\bibitem{DMG} {\sc G.\ Dal Maso, A.\ Garroni,}
{\it Gradient bounds for minimizers of free discontinuity problems related to cohesive zone models in fracture mechanics.}
Calc. Var. Partial Differential Equations \textbf{31} (2008), 137--145.

\bibitem{DMLaz} {\sc G.\ Dal Maso, G.\ Lazzaroni,}
{\it Quasistatic crack growth in finite elasticity with non-interpenetration.}
Ann. Inst. H. Poincar\'e Anal. Non Lin\'eaire \textbf{27} (2010), no. 1, 257--290.

\bibitem{DMOrlToa} {\sc G.\ Dal Maso, G.\ Orlando, R.\ Toader,}
{\it Fracture models for elasto-plastic materials as limits of gradient damage models coupled with plasticity: the antiplane case.}
Calc. Var. Partial Differential Equations \textbf{55} (2016), no. 3.

\bibitem{DMToa} {\sc G.\ Dal Maso, R.\ Toader,}
{\it A model for the quasi-static growth of brittle fractures: existence and approximation results.}
Arch. Ration. Mech. Anal. \textbf{162} (2002), no. 2, 101--135.

\bibitem{DMZan} {\sc G.\ Dal Maso, C.\ Zanini,}
{\it Quasi-static crack growth for a cohesive zone model with prescribed crack path.}
Proc. Roy. Soc. Edinburgh Sect. A \textbf{137} (2007), no. 2, 253--279.

\bibitem{DPT1} {\sc G.\ del Piero, L.\ Truskinovsky,}
{\it A one-dimensional model for localized and distributed failure.}
J. Phys. IV \textbf{8} (1998), 95--102.

\bibitem{DPT2} {\sc G.\ del Piero, L.\ Truskinovsky,}
{\it Macro- and micro-cracking in one-dimensional elasticity.}
Int. J. Solids Struct. \textbf{38} (2001), 1135--1148.

\bibitem{Doob} {\sc J.\ L.\ Doob}, {\it Stochastic Processes.} John Wiley \& Sons, Inc., New York; Chapman \& Hall, Limited, London (1953), viii+654 pp.

\bibitem{FrancfortGarroni} {
{\sc G. Francfort, A. Garroni,}
{\it
A Variational view of partial brittle damage evolution.}
Arch. Ration. Mech. Anal. \textbf{182}, 125-152 (2006).}

\bibitem{FraMar} {\sc G.\ Francfort, J.-J.\ Marigo,}
{\it Revisiting brittle fracture as an energy minimization problem.}
J. Mech. Phys. Solids \textbf{46} (1998), no. 8, 1319--1342.

\bibitem{FraLar} {\sc G.\ Francfort, C. Larsen,}
{\it Existence and convergence for quasi-static evolution in brittle fracture.}
Comm. Pure Appl. Math. \textbf{56} (2003), no. 10, 1465--1500.

\bibitem{FreIur} {\sc F.\ Freddi, F.\ Iurlano,}
{\it Numerical insight of a variational smeared approach to cohesive fracture.}
J. Mech. Phys. Solids \textbf{98} (2017), 156--171.

\bibitem{GarroniLarsen}{
 {\sc A. Garroni, C. J. Larsen,}
{\it 
 Threshold-based quasi-static brittle damage evolution.}
 Arch. Ration. Mech. Anal. \textbf{194}, 585–609 (2009).}

\bibitem{Gia05} {\sc A.\ Giacomini,}
{\it Ambrosio-Tortorelli approximation of quasi-static evolution of brittle fractures.}
Calc. Var. Partial Differential Equations \textbf{22} (2005), no. 2, 129--172.

\bibitem{Gia05b} {\sc A.\ Giacomini,}
{\it Size effects on quasi-static growth of cracks.}
SIAM J. Math. Anal. \textbf{36} (2005), no. 6, 1887--1928.

\bibitem{Gri20} {\sc A.\ Griffith,}
{\it The phenomena of rupture and flow in solids.}
Philos. Trans. Roy. Soc. London Ser. A \textbf{221} (1920), 163--198.

{
\bibitem{HelMT} {\sc L.\ Helmst\"adter,}
{\it Relaxation and approximation of free-discontinuity functionals.}
Master's Degree thesis, University of Bonn (2019).
}

{
\bibitem{KneMieZan08} {\sc D.\ Knees, A.\ Mielke, C.\ Zanini,}
{\it On the inviscid limit of a model for crack propagation.}
Math. Models Methods Appl. Sci. \textbf{18} (2008), no. 9, 1529--1569.}


\bibitem{Laz} {\sc G.\ Lazzaroni,}
{\it Quasistatic crack growth in finite elasticity with Lipschitz data.}
Ann. Mat. Pura Appl. (4) \textbf{190} (2011), no. 1, 165--194.

{
\bibitem{MieRou} {\sc A.\ Mielke, T.\ Roub\'{\i}\v{c}ek}
{\it Rate-independent systems.}
Applied Mathematical Sciences \textbf{193}, Springer 2015.}

\bibitem{NegSca} {\sc M.\ Negri, R.\ Scala,}
{\it A quasi-static evolution generated by local energy minimizers for an elastic material with a cohesive interface.}
Nonlinear Anal. Real Word Appl. \textbf{38} (2017), no. 7, 271--305.

\bibitem{NegVit} {\sc M.\ Negri, E.\ Vitali,}
{\it Approximation and characterization of quasi-static $H^1$-evolutions for a cohesive interface with different loading-unloading regimes.}
Interfaces Free Bound. \textbf{20} (2018), no. 1, 25--67.

\bibitem{ThoZan} {\sc M.\ Thomas, C.\ Zanini,}
{\it Cohesive zone-type delamination in visco-elasticity.}
Discrete Contin. Dyn. Syst. Ser. S \textbf{10} (2017), no. 6, 1487--1517.

\end{thebibliography}
\end{document}